\documentclass[11pt,a4paper]{article}
\usepackage[T1]{fontenc}
\usepackage[utf8]{inputenc}
\usepackage{lmodern}
\usepackage[margin=25mm,headheight=15pt]{geometry}
\usepackage{microtype}
\usepackage{amsmath,amssymb,amsthm,mathtools}
\usepackage{tikz-cd}
\usepackage{needspace}
\usepackage{booktabs,tabularx,longtable,array}
\usepackage{enumitem}
\usepackage{xcolor}
\usepackage{fancyhdr}
\usepackage{hyperref}
\usepackage[nameinlink,noabbrev]{cleveref}
\hypersetup{colorlinks=true,linkcolor=blue!42!black,citecolor=blue!42!black,
 urlcolor=blue!42!black,bookmarksnumbered=true,bookmarksdepth=2,pdftitle={Transfer of noetherian forms
with applications to topological categories},
 pdfsubject={Noetherian forms with modular fibres, transfer functors, and topological categories},
 pdfauthor={Kishan Dayaram, Zurab Janelidze, Francois van Niekerk},
 pdfkeywords={Noetherian form, proper factorization system, conservative functor,
 modular lattice, modular connection, essentially algebraic category,
 topological category, topological congruence,
 weak quotient, homeomorphism theorem}}
\setlist[enumerate]{itemsep=0.2em,topsep=0.4em}
\setlist[itemize]{itemsep=0.2em,topsep=0.4em}
\numberwithin{equation}{section}
\newtheorem{theorem}{Theorem}[section]
\newtheorem{lemma}[theorem]{Lemma}
\newtheorem{proposition}[theorem]{Proposition}
\newtheorem{corollary}[theorem]{Corollary}

\theoremstyle{definition}
\newtheorem{definition}[theorem]{Definition}
\newtheorem{example}[theorem]{Example}
\newtheorem{counterexample}[theorem]{Counterexample}
\theoremstyle{remark}
\newtheorem{remark}[theorem]{Remark}
\newtheorem{question}[theorem]{Question}
\crefname{lemma}{Lemma}{Lemmas}
\Crefname{lemma}{Lemma}{Lemmas}
\crefname{proposition}{Proposition}{Propositions}
\Crefname{proposition}{Proposition}{Propositions}
\crefname{Corollary}{corollary}{corollaries}
\Crefname{corollary}{Corollary}{Corollaries}
\crefname{Metatheorem}{metatheorem}{metatheorems}
\Crefname{metatheorem}{Metatheorem}{Metatheorems}
\crefname{Counterexample}{counterexample}{counterexamples}
\Crefname{counterexample}{Counterexample}{Counterexamples}
\crefname{Remark}{remark}{remarks}
\Crefname{remark}{Remark}{Remarks}
\crefname{Question}{question}{questions}
\Crefname{question}{Question}{Questions}
\crefname{theorem}{Theorem}{Theorem}
\crefname{section}{Section}{Section}
\newcommand{\C}{\mathcal C}
\newcommand{\D}{\mathcal D}
\newcommand{\A}{\mathcal A}
\newcommand{\V}{\mathcal V}
\newcommand{\E}{\mathcal E}
\newcommand{\M}{\mathcal M}
\newcommand{\N}{\mathsf N}
\newcommand{\F}{\mathsf F}
\newcommand{\G}{\mathsf G}
\newcommand{\Set}{\mathbf{Set}}
\newcommand{\Ab}{\mathbf{Ab}}
\newcommand{\Mlc}{\mathbf{Mlc}}
\newcommand{\Ltc}{\mathbf{Ltc}}
\newcommand{\Nsb}{\operatorname{Nsb}}
\newcommand{\Top}{\mathbf{Top}}

\newcommand{\Bool}{\mathbf{Bool}}
\newcommand{\Stone}{\mathbf{Stone}}

\newcommand{\FinVect}{\mathbf{FinVect}}
\newcommand{\op}{\mathrm{op}}
\newcommand{\id}{\mathrm{id}}
\newcommand{\Iso}{\mathrm{Iso}}
\newcommand{\All}{\mathrm{All}}

\newcommand{\Quot}{\mathrm{Quot}}
\newcommand{\Sub}{\operatorname{Sub}}
\newcommand{\Ker}{\operatorname{Ker}}
\newcommand{\Img}{\operatorname{Im}}

\newcommand{\ran}{\operatorname{ran}}
\newcommand{\Cl}{\operatorname{Cl}}
\newcommand{\Eq}{\operatorname{Eq}}
\newcommand{\Con}{\operatorname{Con}}
\newcommand{\Hom}{\operatorname{Hom}}

\newcolumntype{Y}{>{\raggedright\arraybackslash}X}

\newcommand{\Equiv}{\mathbf{Equiv}}
\newcommand{\LRB}{\mathbf{LRBMon}}
\newcommand{\Pow}{\mathcal P}
\newcommand{\BP}{\mathsf P}
\newcommand{\two}{\mathbf 2}
\newcommand{\red}{\operatorname{red}}
\newcommand{\eps}{\varepsilon}

\newcommand{\FZ}{\mathbb Z^{(-)}}

\newcommand{\VCon}{\operatorname{Con}_{\mathrm V}}
\newcommand{\Nrm}{\operatorname{Nrm}}
\newcommand{\Cnm}{\operatorname{Cnm}}
\newcommand{\topgen}[1]{\langle #1\rangle_{\mathrm{top}}}
\allowdisplaybreaks[1]
\makeatletter
\let\@author\@empty
\renewcommand{\author}[1]{%
  \ifx\@author\@empty
    \gdef\@author{#1}%
  \else
    \g@addto@macro\@author{\quad #1}%
  \fi}
\renewcommand{\@maketitle}{%
  \newpage\null\vskip 1em
  \begin{center}
    {\LARGE\@title\par}
    \vskip 1em
    {\large\lineskip .5em
      \begin{tabular}[t]{c}\@author\end{tabular}\par}
    \vskip .6em
    {\normalsize\@date\par}
  \end{center}
  \par\vskip 1em}
\makeatother
\title{\bfseries Transfer theory for noetherian forms}
\author{Kishan Dayaram}
\author{Zurab Janelidze\thanks{Zurab Janelidze used Chat GPT-6 Astra
as a research and writing assistant. The authors take full responsibility
for the final manuscript.}}
\author{Francois van Niekerk}
\date{23 September 2026}
\begin{document}
\maketitle
\begin{abstract}
A noetherian form over a category enables one to formulate and prove homomorphism theorems in that category, such as the isomorphism theorems and the diagram lemmas of homological algebra. In this paper we develop a strategy for establishing existence of a noetherian form over a given category, which enables one to find noetherian forms for a broad range of concrete categories. While it was already known that all semi-abelian categories, Grandis exact categories and algebraic categories have noetherian forms, we now establish that so do wide classes of essentially algebraic and topological categories, which include the categories of small categories, groupoids, topological spaces, extended pseudometric spaces, approach spaces, graphs, measurable spaces and axiom-free relational structures.
\end{abstract}


\section{Introduction}\label{sec:intro}

A noetherian form equips a category with an abstract calculus of
substructures and their images and inverse images, generalizing the
familiar calculus of subgroups under group homomorphisms and extending
the perspective of Mac Lane's subobject chasing in abelian categories,
as presented in \emph{Homology}
\cite[Chapter~XII, Sections~2--3]{MacLane1963}.
Its purpose
is to identify common principles from which the isomorphism theorems
and homological diagram lemmas can be proved uniformly across
different mathematical settings \cite{GJ,DGJRV}. This aim continues
a line of structural thinking in algebra exemplified by Noether's
papers on ideal theory and representation theory
\cite{Noether1927,Noether1929}, where relations among submodules,
quotient modules, and homomorphisms organize the theory. Ore's two
papers \emph{On the Foundation of Abstract Algebra} \cite{Ore1935,Ore1936} pursued a common
foundation for algebraic structure theorems through lattices, while
Mac Lane's \emph{Duality for Groups} \cite{MacLane1950} brought
categorical duality into the study of subgroups, quotients, and
homomorphisms. Noetherian forms bring these lattice-theoretic and
categorical perspectives together in a self-dual framework: reversing
arrows and the ordering of substructures exchanges the roles of images
and kernels, and of embeddings and quotients, while leaving the
governing axioms unchanged.

Noetherian forms unify the subobject forms of semi-abelian categories
and the normal-subobject forms of Grandis exact categories
\cite{Janelidze2014,JW2016,GJ}. Semi-abelian categories
\cite{JMT2002} supply the familiar subobject calculus of groups,
rings without identity, Lie algebras, and other group-like structures;
Grandis exact categories \cite{Grandis1992,Grandis2012} supply a
self-dual calculus of normal subobjects and quotients, including
examples described by modular lattices. More recently, the range of
examples has expanded beyond them: a noetherian forms for topoi was constructed in
\cite{JVN}, and one for every algebraic category was obtained in
\cite{VN}. These results show that a category may admit
a noetherian form even when its ordinary subobject calculus does not
satisfy the noetherian axioms.

The present paper provides a further significant expansion of the range of examples of noetherian forms. The new examples are given by wide classes of essentially algebraic categories and topological categories. In particular, we show that each of the following categories have noetherian forms: small categories,
groupoids, simplicial categories, simplicial
groupoids, topological spaces, preorders, equivalence relations, reflexive directed and
symmetric graphs, and structures for arbitrary finitary relational
signatures of positive arity without additional axioms. Among the new examples are also categories enriched over a set-sized commutative unital
quantale and their symmetric variants: extended quasi-pseudometric
spaces, extended quasi-ultrametric spaces, quasi-pseudometric spaces
bounded by~$1$, fuzzy preorders and similarity spaces, and generalized
probabilistic quasi-pseudometric spaces, together with the corresponding
symmetric metric and probabilistic examples. Further new examples are
measurable spaces, pretopological spaces, Moore closure spaces, abstract convexity spaces,
fuzzy and frame-valued topological spaces, approach spaces, totally
bounded uniform spaces, Efremovi\v{c} proximity spaces,
$\ell^\infty$-bornological spaces, and abstract simplicial complexes.
We also deduce the well-powered ideally exact case
\cite{GJIdeallyExact} from van Niekerk's earlier monadic theorem. Further examples come from geometric theories: torsion groups, groups
of prime-power torsion, periodic monoids, nil associative rings without
identity, and locally nilpotent associative rings without identity.
We also obtain forms for classes of relational structures defined by implications satisfying either of two syntactic criteria, and
for total preorders.

The means of reaching these examples is a class of functors that we
call \emph{transfer functors}. For categories equipped with proper
factorization systems, these are conservative functors that preserve
the two factorization classes, pullbacks of two right-class morphisms
with common codomain, and pushouts of two left-class morphisms with
common domain; the indicated diagrams are required to exist in the
source. The transfer theorem is an \emph{if and only if} statement:
given a noetherian form on the target inducing the specified factorization system,
its pullback is noetherian and induces the prescribed source factorization system
exactly when the functor is a transfer functor.
For the essentially algebraic examples, we pass from partially defined
operations to ordinary algebraic structures in a way that detects
isomorphisms and respects subobjects and their intersections.
Universal semigroup constructions provide this passage for categories
and groupoids. For the topological examples, maps into suitable test
objects encode the structure of each object. An extension property
for these maps makes it possible to construct transfer functors to
sets, and then to algebraic categories by free-algebra constructions.
The notion also includes
Grandis's canonical transfer functors from exact categories with
set-sized normal-subobject lattices to modular lattices
\cite{Grandis1984,Grandis1992,Grandis2012}, and
the monadic functors considered in
\cite[Theorem~5.5]{VN}.
A further restriction principle applies in both settings, and more
generally to any category carrying a noetherian form: a full
subcategory inherits such a form whenever it is closed under the
subobjects and quotients specified by the ambient form.
This allows suitable additional axioms to be imposed on both algebraic
and relational examples, yielding the geometric-theory examples
above. For relational structures, we also give direct syntactic
criteria for constructing transfer functors, covering the stated
matrix conditions.

The transfer method also yields characterizations of existence of a noetherian form on a given category. In particular, we show that a
category with a fixed proper factorization system (having the two types of diagrams specified above) admits a compatible
noetherian form if and only if one can be transferred from $\Set$;
the same assertion holds with $\Ab$, $\Top$, the category of Stone
spaces, or \mbox{Grandis's} category $\Mlc$ of bounded modular lattices and
modular connections as target. The target
factorization systems are surjections and injections in $\Set$ and $\Ab$,
continuous surjections and embeddings in $\Top$ and Stone spaces,
and upper-interval quotients and lower-interval embeddings in
$\Mlc$. Further universal targets include every nontrivial
single-sorted finitary variety, every nonzero well-powered abelian
category with small coproducts, every nondegenerate Grothendieck
topos, and the opposites of universal targets, with the corresponding
factorization systems. In general, changing the target category changes the form, so the result above opens a range of constructions of noetherian forms over the
same category. A principal consequence is that
\emph{every category admitting a compatible noetherian form admits
one with complete algebraic modular fibres}, without changing the
factorization system. Nevertheless, this does not mean that such form is the ``best'' one. What is ``best'' is in fact relative as different noetherian forms can give different formulations and
hypotheses for homomorphism theorems, even when they induce the same
factorization system: the
subobject and quotient orders separately are determined, but not how subobjects and quotients compare with each other.

Thus, the usefulness of a form is relative to the theorems one wishes to formulate and
prove. We discuss topological spaces in detail from this perspective.
There the specific form constructed in this paper recovers the congruence calculus and
isomorphism theorems developed by Veldsman
\cite{Veldsman2019,Veldsman2022}. Parallel descriptions for measurable
spaces, preorders, relational structures, and extended pseudometric
spaces show how the abstract theorems acquire concrete meanings in
the different examples.

The second author used Chat GPT-6 Astra to assist with both the
writing and the research for this paper. The main ideas and the
general results about transfer functors, namely the characterization
theorem (\cref{thm:transfer}) and the result that $\Set$ is a
universal target (\cref{thm:characterization}), together with their
complete proofs, were obtained by the third author without the use of
AI. It is notable that this work grew out of a desire to determine
whether the category of topological spaces has a noetherian form,
a problem that had remained open for several years. Chat GPT-6
Astra (Ultra) was initially asked to work on this problem for
several hours, but it was not able to solve the problem or make any meaningful progress,
despite having access to van Niekerk's monadicity result
\cite[Theorem~5.5]{VN}, a predecessor of the transfer theory developed
here. Once the transfer results were supplied to it, Chat GPT-6
Astra was able to solve the problem in a few minutes.

\tableofcontents

\section{Preliminaries}\label{sec:axioms}

\paragraph{Conventions.}
Categories are locally small and cluster fibres are sets, relative to a
fixed universe. A proper factorization system $(\E,\M)$ is orthogonal,
with epimorphisms in $\E$ and monomorphisms in $\M$. We write $f=me$,
where $e\in\E$ and $m\in\M$. A conservative functor reflects
isomorphisms; faithfulness is not assumed. Preservation statements require
the homogeneous diagrams to exist in the source. In particular, ``two
embeddings'' does not mean an arbitrary map pulled back along an embedding.

\subsection{An equivalent transport presentation}

We use the orean and noetherian axioms of
\cite[Definitions~2.3--2.4]{VN} and \cite[Definitions~21 and~48]{JVN}.
The elementary results recalled below are included with proofs to fix
the transport notation.

An orean form may be presented by a bounded lattice \(\Lambda_X\) for every
object \(X\), together with an adjunction of monotone maps
\[
 f_!:\Lambda_X\rightleftarrows\Lambda_Y:f^*
 \qquad(f:X\to Y),
\]
where
\[
 f_!S\le T\quad\Longleftrightarrow\quad S\le f^*T,
 \qquad
 (gf)_!=g_!f_!,\quad (gf)^*=f^*g^*,
\]
and identities act as identities. This is the bounded-lattice/Galois-connection
presentation of the orean axioms in \cite[Section~2]{JVN}.
The associated relation on clusters is \(T\ge_f S\) if and only if
\(f_!S\le T\).

In particular, \(f_!\) preserves binary joins and bottom, while \(f^*\)
preserves binary meets and top. Define
\[
 \Img f=f_!\top_X,\qquad \Ker f=f^*\bot_Y.
\]
A cluster is \emph{conormal} if it is the image of a morphism and
\emph{normal} if it is the kernel of a morphism. These notions need not
coincide, and an arbitrary cluster need be neither.

For a conormal \(S\in\Lambda_X\), an \emph{embedding} of \(S\) is a morphism
\(\iota_S:S_0\to X\) with image \(S\), universal among morphisms into \(X\)
whose images are at most \(S\). Thus
\[
 \Img f\le S\quad\Longleftrightarrow\quad
 f=\iota_S h\text{ for a unique }h.
\]
Dually, a \emph{quotient} of a normal \(K\in\Lambda_X\) is a morphism
\(\pi_K:X\to Q_K\) with kernel \(K\), universal among morphisms out of \(X\)
whose kernels contain \(K\).

\subsection{The noetherian axioms}

An orean form is noetherian when it satisfies
\begin{gather*}
 f^*f_!S=S\vee\Ker f,\qquad
 f_!f^*T=T\wedge\Img f;\tag{N1}\label{eq:N1}\\
 f=me,\quad e\text{ a quotient of }\Ker f,\quad
 m\text{ an embedding of }\Img f;\tag{N2}\label{eq:N2}
\end{gather*}
and
\begin{equation*}
 \begin{gathered}
 \text{the join of two normal clusters is normal, and}\\
 \text{the meet of two conormal clusters is conormal.}
 \end{gathered}
 \tag{N3}\label{eq:N3}
\end{equation*}
Axiom (N2) includes existence of the indicated universal morphisms.
The representatives $e,m$ are chosen with a common intermediate object;
if representatives are fixed in advance, an isomorphism between their
intermediate objects is understood. As images
and kernels are, by definition, witnessed by morphisms, it supplies embeddings
of all conormal clusters and quotients of all normal clusters.

The following collects the standard embedding and quotient properties
\cite[Lemmas~2.6--2.7]{VN}; see also
\cite[Remark~42, Lemma~44, and Section~4]{JVN} for their
factorization-system formulation. We give the argument using (N2) alone.
\begin{lemma}\label[lemma]{lem:basic}
Let $\F$ be an orean form on a category $\C$ satisfying (N2).
Images, kernels, embeddings, and quotients in the following assertions
are those of $\F$. Then:
\begin{enumerate}[label=\textup{(\roman*)}]
\item Embeddings are monomorphisms and quotients are epimorphisms.
\item An embedding of top, or a quotient of bottom, is an isomorphism.
\item A morphism is an embedding exactly when its kernel is bottom, and is a
 quotient exactly when its image is top.
\item The quotient and embedding classes form a proper factorization system.
\item If \(f=me\) with \(e\) a quotient and \(m\) an embedding, then
 \(\Ker f=\Ker e\) and \(\Img f=\Img m\).
\end{enumerate}
\end{lemma}
\begin{proof}
The universal properties give (i). The identity is an embedding of top and a
quotient of bottom, so uniqueness of universal morphisms gives (ii).
If \(\Ker f=\bot\), its (N2)-factorization starts with an isomorphism, giving
one direction of (iii). Conversely, if \(f\) is an embedding, it and the
embedding in its (N2)-factorization represent the same cluster. Their
comparison is an isomorphism; hence the quotient factor is an isomorphism,
and its kernel, which is \(\Ker f\), is bottom. The other half is dual.

Factorization is (N2). To check orthogonality, consider a commuting square
\(ma=be\), where \(e\) is a quotient and \(m\) an embedding. By (iii),
\(\Img e=\top\), so
\[
 \Img b=\Img(be)=\Img(ma)\le\Img m.
\]
There is a unique \(d\) with \(md=b\); cancellation of the monomorphism
\(m\) gives \(de=a\). The resulting factorization classes are the mutual
orthogonals, as follows from factorization and this lifting property.
Finally, (v) follows from \(e_!\top=\top\), \(m^*\bot=\bot\), and
functoriality of transport.
\end{proof}

The interval correspondence is the lattice isomorphism theorem in
\cite[Remark~51]{JVN}. The equality criteria below are its immediate
transport consequences.
\begin{lemma}[Intervals and equality under transport]\label[lemma]{lem:intervals}
Let $\F$ be a noetherian form on a category $\C$, with fibre
$\Lambda_X$ over $X$ and transport $f_!\dashv f^*$ along a morphism
$f:X\to Y$. For every such $f$, all $S,T\in\Lambda_X$, and all
$U,V\in\Lambda_Y$, the following identities hold, where
$\Img f=f_!\top_X$ and $\Ker f=f^*\bot_Y$:
\begin{align}
 \ran(f_!)&=\downarrow\Img f,&
 \ran(f^*)&=\uparrow\Ker f,\label{eq:ranges}\\
 f_!S=f_!T&\Longleftrightarrow S\vee\Ker f=T\vee\Ker f,\label{eq:equivjoin}\\
 f^*U=f^*V&\Longleftrightarrow U\wedge\Img f=V\wedge\Img f.\label{eq:equivmeet}
\end{align}
For an embedding \(m\), direct image identifies its domain fibre with
\(\downarrow\Img m\). Dually, inverse image along a quotient \(e\)
identifies its codomain fibre with \(\uparrow\Ker e\).
\end{lemma}
\begin{proof}
For \(T\le\Img f\), (N1) gives \(f_!f^*T=T\); every direct image is at most
\(\Img f\). The assertion about inverse images is dual. One direction of
\eqref{eq:equivjoin} follows by applying \(f^*\). For the converse, apply
\(f_!\), using preservation of joins and \(f_!\Ker f=\bot\). The meet
identity is dual. For an embedding, (N1) gives \(m^*m_!=\id\), and for a
quotient it gives \(e_!e^*=\id\).
\end{proof}

\subsection{Existence of homogeneous diagrams}

This is \cite[Lemma~2.8]{VN} and its dual, applied using (N2) and (N3).
\begin{lemma}\label[lemma]{lem:homogeneous-exist}
Let $\F$ be a noetherian form on a category $\C$, and let
$\M$ and $\E$ be its embedding and quotient classes. For embeddings
$m:A\to X$ and $n:B\to X$, their pullback exists and its common
composite into $X$ embeds $\Img_\F m\wedge\Img_\F n$.
For quotients $e:X\to Y$ and $d:X\to Z$, their pushout exists and
its common composite out of $X$ is a quotient of
$\Ker_\F e\vee\Ker_\F d$.
\end{lemma}
\begin{proof}
Given embeddings \(m:A\to X\), \(n:B\to X\), take an embedding
\(r:P\to X\) of \(\Img m\wedge\Img n\), using (N3) and (N2).
It factors through both embeddings, say \(r=mp=nq\). Any cone to \(m,n\)
has common composite whose image lies below both images, so factors uniquely
through \(r\). This makes \((p,q)\) a pullback cone. Its projections belong
to \(\M\), by pullback stability of the right class.

Dually, for quotients \(e:X\to Y\), \(d:X\to Z\), the quotient
\(q:X\to P\) of \(\Ker e\vee\Ker d\) factors as \(q=ue=vd\).
A compatible pair of maps out of \(Y,Z\) has a common composite with kernel
containing this join, so factors uniquely through \(q\). Cancellation of
\(e,d\) proves the pushout property. Its other two arrows belong to \(\E\).
\end{proof}

In particular, inverse image of a conormal cluster along an embedding is
conormal, and direct image of a normal cluster along a quotient is normal.
For the first assertion, the pullback above represents \(m^*\Img n\):
apply (N1) to the meet \(\Img m\wedge\Img n\). The second assertion is dual.
These facts do not assert closure under the corresponding operations along
arbitrary morphisms.

\section{The transfer functors}\label{sec:transfer}

\subsection{Definition and elementary properties}\label{sec:transfer-definition}

\begin{definition}\label{def:transfer-functor}
Let $\C$ and $\D$ have specified proper factorization systems
$(\E_\C,\M_\C)$ and $(\E_\D,\M_\D)$, and suppose that both
categories have the corresponding homogeneous pullbacks and pushouts.
A \emph{transfer functor} $T:\C\to\D$ is a functor such that:
\begin{enumerate}[label=\textup{(\roman*)}]
\item $T(\E_\C)\subseteq\E_\D$ and
$T(\M_\C)\subseteq\M_\D$;
\item $T$ is conservative;
\item $T$ preserves pullbacks of two $\M_\C$-morphisms;
\item $T$ preserves pushouts of two $\E_\C$-morphisms.
\end{enumerate}
We write $\C\rightsquigarrow\D$ when such a functor exists.
\end{definition}

The definition depends on the specified factorization systems, but does
not require a choice of noetherian form on either category. Theorem~\ref{thm:transfer}
explains the terminology by characterizing pullback of a compatible form.
Grandis used the term ``transfer functor'' for canonical functors of
normal subobjects \cite{Grandis1984,Grandis1992,Grandis2012}.
Section~\ref{sec:grandis-transfer} shows that his functors for exact
categories are examples of \cref{def:transfer-functor}; the two uses of
the term do not coincide for arbitrary semiexact categories.

We first record the formal closure properties of the definition.
The assertions about adjoints use the standard limit-preservation theorem
and its dual \cite[Proposition~18.9]{AHS}.
\begin{proposition}\label{prop:transfer-elementary}
Consider categories equipped with specified proper factorization
systems and having pullbacks of two right-class morphisms and pushouts
of two left-class morphisms. Between such categories, transfer functors
contain identities and are closed under composition.
Equivalences transporting the specified factorization systems are
transfer functors. If $T:\C\to\D$ is a transfer functor between
systems $(\E_\C,\M_\C)$ and $(\E_\D,\M_\D)$, then
$T^{\op}:\C^{\op}\to\D^{\op}$ is one for the opposite systems
$(\M_\C^{\op},\E_\C^{\op})$ and
$(\M_\D^{\op},\E_\D^{\op})$.

More generally, let $T:\C\to\D$ be a conservative functor between
such categories preserving their specified classes. If $T$ is a left
adjoint, it is a transfer functor precisely when it preserves pullbacks
of two $\M_\C$-morphisms. If $T$ is a right adjoint, it is a transfer
functor precisely when it preserves pushouts of two $\E_\C$-morphisms.
\end{proposition}
\begin{proof}
Each defining property is preserved by composition and by transport
along an equivalence. Passage to opposites interchanges the two classes
and the two kinds of homogeneous diagrams. A left adjoint preserves
all existing pushouts, and a right adjoint preserves all existing
pullbacks, giving the final assertions.
\end{proof}

The next observation follows from uniqueness of factorizations
\cite[Proposition~14.4]{AHS}.
\begin{lemma}\label[lemma]{lem:classreflection}
Let $\C$ and $\D$ have proper factorization systems
$(\E_\C,\M_\C)$ and $(\E_\D,\M_\D)$, and let
$T:\C\to\D$ be conservative with
$T(\E_\C)\subseteq\E_\D$ and
$T(\M_\C)\subseteq\M_\D$. Then $T$ reflects both classes:
\[
 \E_\C=T^{-1}(\E_\D),\qquad \M_\C=T^{-1}(\M_\D).
\]
\end{lemma}
\begin{proof}
Factor \(f=me\). If \(Tf\in\M_\D\), uniqueness of target factorizations
forces \(Te\) to be invertible. Conservativity makes \(e\) invertible,
so \(f\in\M_\C\). The other assertion is dual.
\end{proof}

Reflection of the individual classes is not the same assertion as reflection
of their subobject and quotient orders. In the proof of the transfer theorem,
the homogeneous diagrams are what allow conservativity to imply the latter.

\subsection{The pullback criterion}\label{sec:pullback-criterion}

Let \(T:\C\to\D\) be a functor, with proper factorization systems
\((\E_\C,\M_\C)\) and \((\E_\D,\M_\D)\).
Let \(\F\) be a noetherian form on \(\D\) inducing its specified system,
and put \(\G=T^*\F\).

\subsubsection{What always transfers}

The following pullback description and automatic preservation of the
orean structure and (N1) are due to van Niekerk
\cite[Lemma~3.1, Proposition~3.3, and Corollary~3.4]{VN}.
The pullback form is given by
\begin{equation}\label{eq:basechange}
 \G(X)=\F(TX),\qquad
 f_!^{\G}=(Tf)_!^{\F},\qquad
 f^{*\G}=(Tf)^{*\F}.
\end{equation}
Its fibre lattice is literally the target fibre lattice, so
\[
 \Ker_{\G}f=\Ker_{\F}(Tf),\qquad
 \Img_{\G}f=\Img_{\F}(Tf).
\]
Consequently, \(\G\) is orean and satisfies (N1), for \emph{every} functor
\(T\). No conservativity, factorization preservation, or diagram preservation
is needed for these assertions.

What can change is normality and conormality: in the pullback, these must be
witnessed by arrows in \(\C\), rather than arbitrary arrows in \(\D\).
This accounts for the separate analyses of (N2) and (N3).

\subsubsection{Order reflection is exactly (N2)}

For now assume only
\begin{equation}\label{eq:preserveclasses}
 T(\E_\C)\subseteq\E_\D,\qquad
 T(\M_\C)\subseteq\M_\D.
\end{equation}
Order \(\M\)-subobjects by factorization. Order \(\E\)-quotients by
increasing identification:
\[
 [e]\le[d]\quad\Longleftrightarrow\quad d=he\text{ for some }h.
\]
Define monotone maps, or initially maps of quotient preorders if size has not
yet been established,
\begin{align}
 a_X:\Sub_{\M_\C}(X)&\longrightarrow\F(TX),&
 [m]&\longmapsto\Img_{\F}(Tm),\label{eq:alpha-transfer}\\
 b_X:\Quot_{\E_\C}(X)&\longrightarrow\F(TX),&
 [e]&\longmapsto\Ker_{\F}(Te).\label{eq:beta-transfer}
\end{align}
When these maps reflect order, their domains are essentially sets, since their
codomain is a set.

Factoring any source arrow as \(f=me\) and using
\cref{lem:basic}\textup{(v)} in the target gives
\begin{equation}\label{eq:transfer-ranges}
 \G(X)^{\mathrm c}=\ran a_X,\qquad
 \G(X)^{\mathrm n}=\ran b_X.
\end{equation}

\Needspace{10\baselineskip}
\begin{theorem}[The (N2) criterion]\label{thm:N2}
Let $\C$ and $\D$ have proper factorization systems
$(\E_\C,\M_\C)$ and $(\E_\D,\M_\D)$.
Let $\F$ be a noetherian form on $\D$ inducing
$(\E_\D,\M_\D)$, and let $T:\C\to\D$ satisfy
$T(\E_\C)\subseteq\E_\D$ and
$T(\M_\C)\subseteq\M_\D$. Write $\G=T^*\F$ for the
pullback form, whose fibre over $X$ is $\F(TX)$ and whose transports
along $f$ are those of $\F$ along $Tf$.
For each $X\in\C$, define
\begin{align*}
 a_X:\Sub_{\M_\C}(X)&\to\F(TX),&
 [m]&\mapsto\Img_\F(Tm),\\
 b_X:\Quot_{\E_\C}(X)&\to\F(TX),&
 [e]&\mapsto\Ker_\F(Te).
\end{align*}
Here subobjects are ordered by factorization, and quotients by
$[e]\le[d]$ when $d=he$ for some $h$; their domains may initially
be treated as preorders of representatives. The following are equivalent:
\begin{enumerate}[label=\textup{(\roman*)}]
\item \(\G\) satisfies (N2).
\item Every \(a_X\) and \(b_X\) reflects order.
\item The following two factorization-reflection properties hold:
\begin{align*}
 Tm_1\text{ factors through }Tm_2
 &\Longrightarrow m_1\text{ factors through }m_2,\\
 Te_2\text{ factors through }Te_1
 &\Longrightarrow e_2\text{ factors through }e_1,
\end{align*}
where the \(m_i\in\M_\C\) have common codomain and the
\(e_i\in\E_\C\) have common domain.
\end{enumerate}
Whenever these conditions hold, the factorization system of \(\G\) is exactly
\((\E_\C,\M_\C)\).
\end{theorem}
\begin{proof}
In the target form, comparison of images of embeddings is equivalent to
factorization of the embeddings, and comparison of kernels of quotients is
equivalent to factorization of the quotients. Thus (ii) and (iii) are equivalent.

Assume (ii). If \(m:M\to X\) belongs to \(\M_\C\), take any
\(f:Y\to X\) with \(\Img_\G f\le\Img_\G m\), and factor it as
\(f=ne\). Then
\[
 \Img_\F(Tn)\le\Img_\F(Tm).
\]
Order reflection yields \(n=mh\), so \(f=mhe\); uniqueness follows because
\(m\) is monic. Therefore \(m\) is a \(\G\)-embedding of its image.
The dual argument shows that every \(e\in\E_\C\) is a \(\G\)-quotient
of its kernel. Factoring any arrow in the specified source system now gives
its required (N2)-factorization, because images and kernels agree with those
of the factors after applying \(T\).

Conversely, assume (i). For \(m\in\M_\C\), preservation of the target
class gives \(\Ker_\G m=\Ker_\F(Tm)=\bot\). By
\cref{lem:basic}, \(m\) is a \(\G\)-embedding. Its universal property
reflects the comparisons in (ii). The quotient argument is dual.

We have obtained both inclusions \(\M_\C\subseteq\M_\G\) and
\(\E_\C\subseteq\E_\G\). Both pairs are factorization systems, so
orthogonality reverses either inclusion and forces equality of both classes.
\end{proof}

For right adjoints, this observation is implicit in the class-reflection
result \cite[Proposition~4.4]{VN} and the isomorphism criterion
\cite[Lemma~2.7]{VN}. The argument below requires no adjoint.
\begin{proposition}[Conservativity forced by (N2)]\label[proposition]{prop:N2conservative}
Let $T:\C\to\D$ be a functor and $\F$ a noetherian form on
$\D$. If the pullback form $\G=T^*\F$, with fibres
$\G(X)=\F(TX)$ and transports along $f$ given by those along $Tf$,
satisfies (N2), then $T$ is conservative. No factorization system on
$\C$ is assumed in advance.
\end{proposition}
\begin{proof}
If \(Tf\) is invertible, then \(\Ker_\G f=\bot\) and
\(\Img_\G f=\top\). Its (N2)-factorization is an embedding of top after
a quotient of bottom. Both are isomorphisms by \cref{lem:basic}.
\end{proof}

\subsubsection{Range closure is exactly (N3)}

\begin{proposition}[The (N3) criterion]\label[proposition]{prop:N3}
Let $\C$ and $\D$ have proper factorization systems
$(\E_\C,\M_\C)$ and $(\E_\D,\M_\D)$.
Let $\F$ be a noetherian form on $\D$ inducing
$(\E_\D,\M_\D)$, and let $T:\C\to\D$ satisfy
$T(\E_\C)\subseteq\E_\D$ and
$T(\M_\C)\subseteq\M_\D$. Write $\G=T^*\F$ for the
pullback form, whose fibre over $X$ is $\F(TX)$ and whose transports
along $f$ are those of $\F$ along $Tf$.
For each $X\in\C$, put
\begin{align*}
 A_X&=\{\Img_\F(Tm):m\in\M_\C,\ \operatorname{cod}m=X\},\\
 B_X&=\{\Ker_\F(Te):e\in\E_\C,\ \operatorname{dom}e=X\}.
\end{align*}
Then $\G$ satisfies (N3) if and only if every $A_X$ is closed under
binary meets and every $B_X$ under binary joins in $\F(TX)$.
This equivalence does not require (N2) or conservativity.
\end{proposition}
\begin{proof}
This is (N3) applied to \eqref{eq:transfer-ranges}.
\end{proof}

The following argument extends the homogeneous-diagram proof used for
right adjoints in \cite[Propositions~4.5--4.6]{VN}.
\begin{proposition}\label[proposition]{prop:squares-N3}
Let $\C$ and $\D$ have proper factorization systems
$(\E_\C,\M_\C)$ and $(\E_\D,\M_\D)$.
Let $\F$ be a noetherian form on $\D$ inducing
$(\E_\D,\M_\D)$, and let $T:\C\to\D$ satisfy
$T(\E_\C)\subseteq\E_\D$ and
$T(\M_\C)\subseteq\M_\D$. Write $\G=T^*\F$ for the
pullback form, whose fibre over $X$ is $\F(TX)$ and whose transports
along $f$ are those of $\F$ along $Tf$.
If pullbacks of two $\M_\C$-morphisms and pushouts of two
$\E_\C$-morphisms exist in $\C$ and are preserved by $T$, then
$\G$ satisfies (N3). Conservativity of $T$ is not required.
\end{proposition}
\begin{proof}
If \(r:P\to X\) is the common composite in a pullback of
\(m,n\in\M_\C\), preservation and \cref{lem:homogeneous-exist} in the
target give
\[
 \Img_\F(Tr)=\Img_\F(Tm)\wedge\Img_\F(Tn).
\]
As \(r\in\M_\C\), this proves meet closure of \(\ran a_X\).
The common composite \(q\) in a pushout of \(e,d\in\E_\C\) similarly
satisfies
\[
 \Ker_\F(Tq)=\Ker_\F(Te)\vee\Ker_\F(Td),
\]
proving join closure of \(\ran b_X\).
\end{proof}

\subsubsection{The complete equivalence}

For right adjoints, the relation between (N3) and preservation of
quotient pushouts is established in \cite[Proposition~4.6]{VN},
and the monadic criterion is \cite[Theorem~5.5]{VN}.
The following criterion treats arbitrary functors and a prescribed
source factorization system.
\begin{theorem}[Transfer and necessity]\label{thm:transfer}
Let $\C$ and $\D$ be categories with proper factorization systems
$(\E_\C,\M_\C)$ and $(\E_\D,\M_\D)$, let $\F$ be
a noetherian form on $\D$ inducing its specified system, and let
$T:\C\to\D$ be a functor. The pullback form $T^*\F$ has fibre
$\F(TX)$ over $X$ and uses the $\F$-transports along $Tf$ for
each morphism $f$ of $\C$. The following are equivalent.
\begin{enumerate}[label=\textup{(\roman*)}]
\item \(T^*\F\) is a noetherian form whose underlying factorization system
 is \((\E_\C,\M_\C)\).
\item \(T\) preserves the specified factorization systems, is conservative,
 and pullbacks of two \(\M_\C\)-morphisms and pushouts of two
 \(\E_\C\)-morphisms exist in \(\C\) and are preserved by \(T\).
\end{enumerate}
\end{theorem}
\begin{proof}
Assume (ii). Orean structure and (N1) transfer automatically, and (N3) follows
from \cref{prop:squares-N3}. To obtain (N2), it suffices by \cref{thm:N2}
to reflect the two factorization orders.

Suppose \(Tm\) factors through \(Tn\), for
\(m:A\to X\), \(n:B\to X\) in \(\M_\C\). Form
\[
\begin{tikzcd}
P\ar[r,"q"]\ar[d,"p"'] & B\ar[d,"n"]\\
A\ar[r,"m"'] & X.
\end{tikzcd}
\]
In the preserved pullback, \(Tp\) is an isomorphism, since \(Tm\) factors
through the monomorphism \(Tn\). Conservativity makes \(p\) invertible,
and \(m=nqp^{-1}\). The dual pushout argument gives quotient-order
reflection. Now apply \cref{thm:N2}.

Conversely, assume (i). For \(m\in\M_\C\),
\(\Ker_\F(Tm)=\Ker_\G m=\bot\), hence \(Tm\in\M_\D\).
Dually, quotients are preserved. Conservativity follows from
\cref{prop:N2conservative}.

The homogeneous diagrams exist in \(\C\) by
\cref{lem:homogeneous-exist}. In the pullback of two embeddings, the common
composite \(r\) embeds the meet of their \(\G\)-images. Since the fibre
lattices in \eqref{eq:basechange} are identical,
\[
 \Img_\F(Tr)=\Img_\F(Tm)\wedge\Img_\F(Tn).
\]
The map \(Tr\) is an embedding in \(\D\). Its universal property therefore
makes the image square the target pullback. Dually, the image of the quotient
of the join of two kernels is the corresponding target quotient, giving
pushout preservation.
\end{proof}

\begin{corollary}[Two useful partial equivalences]\label[corollary]{cor:partialiff}
Let $\C$ and $\D$ have proper factorization systems
$(\E_\C,\M_\C)$ and $(\E_\D,\M_\D)$.
Let $\F$ be a noetherian form on $\D$ inducing
$(\E_\D,\M_\D)$, and let $T:\C\to\D$ satisfy
$T(\E_\C)\subseteq\E_\D$ and
$T(\M_\C)\subseteq\M_\D$. Write $\G=T^*\F$ for the
pullback form, whose fibre over $X$ is $\F(TX)$ and whose transports
along $f$ are those of $\F$ along $Tf$.
Call pullbacks of two $\M_\C$-morphisms and pushouts of two
$\E_\C$-morphisms the homogeneous diagrams in $\C$.
Then:
\begin{enumerate}[label=\textup{(\roman*)}]
\item If the homogeneous diagrams exist in $\C$ and are preserved by $T$, then \(T^*\F\)
 is orean and satisfies (N1) and (N3), and
 \[
 T^*\F\text{ satisfies (N2)}\quad\Longleftrightarrow\quad T\text{ is conservative}.
 \]
\item If $T^*\F$ satisfies (N2), then it satisfies (N3) if and only if
 the homogeneous diagrams exist in $\C$ and are preserved by $T$.
\end{enumerate}
\end{corollary}
\begin{proof}
Orean structure and (N1) always transfer. In (i), (N3) follows from
\cref{prop:squares-N3}, so \cref{thm:transfer} identifies (N2) with
conservativity. In (ii), (N2) already forces conservativity by
\cref{prop:N2conservative}; apply \cref{thm:transfer} again.
\end{proof}

\subsubsection{Why existence must be stated}

\begin{counterexample}\label[counterexample]{ex:missingpullback}
Preservation only of homogeneous diagrams that happen to exist is insufficient.
Let \(\D\) be the four-element Boolean lattice \(\{0,a,b,1\}\), regarded as
a category, with system \((\Iso,\All)\). Give it the subobject form, whose
fibre over \(x\) is \([0,x]\). Along \(x\le y\), direct image is inclusion
and inverse image is meet with \(x\). This is noetherian: the only normal
cluster is \(0\), every cluster is conormal, and (N1)--(N3) follow immediately.

Let \(\C\) be the full subcategory on \(\{a,b,1\}\), with the same type of
factorization system, and let \(T\) be its inclusion. It is conservative,
preserves the two classes, and preserves all homogeneous pullbacks and
pushouts which exist in \(\C\). But the pullback of
\(a\to1\leftarrow b\) does not exist in \(\C\).

The fibre of \(T^*\F\) over \(1\) is still \(\{0,a,b,1\}\), while its
conormal clusters there are only \(a,b,1\). Their meet \(a\wedge b=0\)
is not conormal. Thus (N3) fails.
\end{counterexample}

\subsection{The cluster-set representation theorem}\label{sec:clusters}

Let \(\N\) be a noetherian form on \(\C\), with system \((\E,\M)\).
Define
\[
 \Cl_\N:\C\longrightarrow\Set,\qquad
 \Cl_\N(X)=\Lambda_X,\qquad \Cl_\N(f)=f_!.
\]
The orean composition laws make this a covariant functor. It uses
\emph{all} clusters, not just the normal or conormal clusters.

\begin{theorem}[Cluster-set representation]\label{thm:cluster}
Let $\N$ be a noetherian form on a category $\C$, with cluster
fibres $\Lambda_X$, direct images $f_!$, and induced quotient--embedding
system $(\E,\M)$. Define the cluster-set functor by
\[
 \Cl_\N:\C\to\Set,\qquad
 \Cl_\N(X)=\Lambda_X,\qquad \Cl_\N(f)=f_!.
\]
Then $\Cl_\N$ is conservative and, for every morphism $f$ of $\C$,
preserves and reflects the factorization classes:
\[
 f\in\E\iff f_!\text{ is surjective},\qquad
 f\in\M\iff f_!\text{ is injective},
\]
It also preserves pullbacks of two $\M$-morphisms and pushouts of
two $\E$-morphisms.
\end{theorem}
\begin{proof}
We separate the assertions.

\emph{Injections and surjections.}
By \cref{lem:intervals}, the range of $f_!$ is
$\downarrow\Img f$, so $f_!$ is surjective exactly when
$\Img f=\top$. Equation~\eqref{eq:equivjoin} gives injectivity
exactly when $\Ker f=\bot$. The class assertions now follow from
\cref{lem:basic}.
A bijective \(f_!\) therefore makes \(f\in\E\cap\M\), proving
conservativity.

\emph{Pullbacks.}
For \(m:A\to X\), \(n:B\to X\) in \(\M\), put
\(M=\Img m\), \(N=\Img n\), and let \(r:P\to X\) embed \(M\wedge N\).
By \cref{lem:intervals}, applying \(\Cl_\N\) to their pullback identifies
it with the square of inclusions
\[
\begin{tikzcd}
\downarrow(M\wedge N)\ar[r,hook]\ar[d,hook] & \downarrow N\ar[d,hook]\\
\downarrow M\ar[r,hook] & \Lambda_X.
\end{tikzcd}
\]
It is a set-theoretic pullback because
\(\downarrow(M\wedge N)=\downarrow M\cap\downarrow N\).

\emph{Pushouts.}
For \(e:X\to Y\), \(d:X\to Z\) in \(\E\), put
\(K=\Ker e\), \(L=\Ker d\). Their pushout has common composite
\(q:X\to P\) with \(\Ker q=K\vee L\), by
\cref{lem:homogeneous-exist}. On the set \(\Lambda_X\), define
\[
 S\sim_K T\quad\Longleftrightarrow\quad S\vee K=T\vee K.
\]
By \eqref{eq:equivjoin}, these are precisely the kernel equivalence relations
of \(e_!,d_!,q_!\), for \(K,L,K\vee L\), respectively.

In any join-semilattice,
\begin{equation}\label{eq:join-equivalences}
 \sim_K\vee_{\Eq}\sim_L=\sim_{K\vee L}.
\end{equation}
Both relations on the left are contained in the right. For the converse,
connect any \(S\) to \(S\vee K\vee L\) by
\[
 S\sim_K S\vee K\sim_L S\vee K\vee L.
\]
If \(S\vee K\vee L=T\vee K\vee L\), the two such chains connect \(S\)
and \(T\).

A pushout of two surjections from a common set is the quotient by the join of
their kernel equivalence relations: a map out of the common set factors through
both exactly when it is constant on that join. Since \(e_!,d_!,q_!\) are
surjective, \eqref{eq:join-equivalences} proves the required pushout property.
\end{proof}

The representation concerns the specified factorization classes and
homogeneous diagrams. It does not assert faithfulness or recovery of the
original form.

\subsection{Existence and universal targets}\label{sec:targets}\label{app:targets}

\begin{theorem}[Existence characterization]\label{thm:characterization}
Let $\C$ be a category with a specified proper factorization system
$(\E,\M)$, pullbacks of two $\M$-morphisms, and pushouts of two
$\E$-morphisms. Equip $\Set$ with its surjection--injection system.
The following are equivalent:
\begin{enumerate}[label=\textup{(\roman*)}]
\item $\C$ has a noetherian form inducing $(\E,\M)$;
\item There is a conservative functor $T:\C\to\Set$ sending $\E$
to surjections and $\M$ to injections, preserving pullbacks of two
$\M$-morphisms and pushouts of two $\E$-morphisms; equivalently,
$\C\rightsquigarrow\Set$.
\end{enumerate}
\end{theorem}
\begin{proof}
The cluster-set theorem gives (i)$\Rightarrow$(ii). Conversely, sets
admit a noetherian form inducing their usual factorization system
\cite{VN}; pulling it back along the transfer functor gives (i), by
\cref{thm:transfer}.
\end{proof}

The existence of the homogeneous diagrams can be included in each
representation condition instead of being assumed in advance. In the
noetherian-form condition it follows automatically. No particular original
form is recovered by this equivalence: it is an existence theorem for a
compatible form.

\begin{proposition}[Criterion for a universal target]\label[proposition]{prop:target}
Let $\D$ have a specified proper factorization system
$(\E_\D,\M_\D)$, pullbacks of two $\M_\D$-morphisms, and
pushouts of two $\E_\D$-morphisms. Equip $\Set$ with its
surjection--injection system, and write $\A\rightsquigarrow\mathcal B$
when a transfer functor exists between the specified systems on
$\A$ and $\mathcal B$. The following are equivalent:
\begin{enumerate}[label=\textup{(\roman*)}]
\item For every category $\C$ with a proper factorization system
$(\E_\C,\M_\C)$ and its corresponding homogeneous pullbacks and
pushouts, $\C$ has a noetherian form inducing that system if and
only if $\C\rightsquigarrow\D$.
\item $\D$ has a noetherian form inducing $(\E_\D,\M_\D)$,
and $\Set\rightsquigarrow\D$.
\item $\Set\rightsquigarrow\D$ and $\D\rightsquigarrow\Set$.
\end{enumerate}
A category $\D$ with its specified system satisfying these conditions
is called a universal target.
\end{proposition}
\begin{proof}
Sufficiency follows by composing a cluster functor with
\(\Set\rightsquigarrow\D\), and by applying the transfer theorem in the
reverse direction. For necessity, apply the proposed characterization first
to \(\C=\D\), using the identity functor, and then to \(\C=\Set\).
The alternative formulation follows from \cref{thm:characterization}.
\end{proof}

\subsection{Grandis's transfer functors and modular fibres}\label{sec:grandis-transfer}

We recall Grandis's lattice category and transfer construction
\cite{Grandis1984,Grandis1992,Grandis2012}. The formulation of modular
connections below is also recalled in \cite[Section~3]{GJM2025}.

\begin{definition}\label{def:mlc}
Let $\Mlc$ denote the category whose objects are set-sized bounded
modular lattices. An arrow $f:L\to M$ is an adjunction
$f_!:L\rightleftarrows M:f^*$ satisfying
\begin{equation}\label{eq:modular-connection}
 f^*f_!(x)=x\vee f^*(0_M),\qquad
 f_!f^*(y)=y\wedge f_!(1_L)
 \quad(x\in L,\ y\in M).
\end{equation}
Such an adjunction is called a \emph{modular connection}. Composition
uses composition of left adjoints and reverse composition of right
adjoints. We write $\mathsf U$ for the form on $\Mlc$ with fibre $L$
over $L$ and transport given by the components of each connection.
\end{definition}

On modular lattices, \eqref{eq:modular-connection} is equivalent to
Frobenius reciprocity and its dual, which are stable under composition;
thus these arrows do form a category
\cite[Sections~3--4]{GJM2025}. The next lemma is the standard
interval description of Grandis's exact category $\Mlc$
\cite[Section~1.4]{Grandis1992}; we include a verification in our
noetherian-form language.

\begin{lemma}[The canonical form on modular lattices]\label{lem:mlc-form}
Let $\Mlc$ be the category of set-sized bounded modular lattices and
modular connections, and let $\mathsf U$ be its form whose fibre at
$L$ is $L$ itself and whose transports are the adjoints of each
connection. Then $\mathsf U$ is noetherian, and every element of
every fibre is normal and conormal. Its quotient and embedding classes,
denoted $(\E_{\Mlc},\M_{\Mlc})$, consist, up to isomorphism, of
the upper-interval quotients and lower-interval embeddings
\[
 q_a:L\longrightarrow[a,1_L],\qquad
 j_a:[0_L,a]\longrightarrow L \quad(a\in L).
\]
Their adjoints are
\[
 (q_a)_!(x)=x\vee a,\quad q_a^*(z)=z,
 \qquad (j_a)_!(z)=z,\quad j_a^*(x)=x\wedge a.
\]
This is a proper factorization system, and its homogeneous pullbacks
and pushouts exist.
\end{lemma}
\begin{proof}
The displayed pairs are adjunctions satisfying
\eqref{eq:modular-connection}. Their kernels and images are respectively
$a$ and $a$, so every cluster is normal and conormal. Axiom (N1) is
\eqref{eq:modular-connection}, and (N3) is therefore automatic.

For a connection $f:L\to M$, put $k=f^*(0_M)$ and $i=f_!(1_L)$.
Its adjoints restrict to mutually inverse order isomorphisms
$[k,1_L]\cong[0_M,i]$, giving
\[
 L\xrightarrow{q_k}[k,1_L]\xrightarrow{\cong}
 [0_M,i]\xrightarrow{j_i}M.
\]
The interval arrows have the required universal properties. If
$g:K\to L$ has $g_!(1_K)\le a$, restrict $g_!$ to $[0_L,a]$
as codomain and $g^*$ to that interval as domain. The resulting
adjunction is again a modular connection and is the unique factor
through $j_a$. Dually, if $h:L\to K$ has $a\le h^*(0_K)$,
its restrictions define the unique factor through $q_a$. This proves
(N2). The proper factorization system and the existence of the
homogeneous diagrams follow from
\cref{lem:basic,lem:homogeneous-exist}.
\end{proof}

A noetherian form with modular fibres is recovered from this canonical
form, whether or not all its clusters are normal and conormal.

\begin{proposition}[Recovering a form with modular fibres]\label{prop:modular-form-representation}
Let $\N$ be a noetherian form on a category $\C$, with set-sized
modular fibre lattices $\Lambda_X$, transports $f_!\dashv f^*$,
and induced factorization system $(\E,\M)$. Equip $\Mlc$, the
category of bounded modular lattices and modular connections, with its
canonical form $\mathsf U$ and the factorization system of that form.
Then
\[
 T_{\N}:\C\longrightarrow\Mlc,\qquad
 T_{\N}(X)=\Lambda_X,\qquad T_{\N}(f)=(f_!\dashv f^*)
\]
is a transfer functor and $T_{\N}^*\mathsf U=\N$.
Conversely, a form obtained by pulling back $\mathsf U$ along a
functor to $\Mlc$ has modular fibres.
\end{proposition}
\begin{proof}
Axiom (N1) makes each transport a modular connection, and functoriality
of transport gives $T_{\N}$. Its pullback is $\N$ by construction,
so \cref{thm:transfer} applies. The converse follows because each
pullback fibre is an object of $\Mlc$.
\end{proof}

\Needspace{13\baselineskip}
Grandis's normal-subobject transfer functor is constructed in
\cite[Section~1.5]{Grandis1992}; exact categories are modular by
\cite[Section~1.7]{Grandis1992}. Its canonical form is the standard
binormal noetherian form associated with a Grandis exact category
\cite[Section~4]{JW2016}; see also \cite[Sections~1--2]{DGJRV}.
Its transfer property follows by specializing
\cref{prop:modular-form-representation}.

\begin{corollary}[Grandis transfer functors]\label[corollary]{prop:grandis-transfer}
Let $(\A,\mathcal J)$ be a Grandis exact category with set-sized
lattices of normal subobjects. Here $\mathcal J$ is its closed ideal
of null arrows, kernels and cokernels are taken relative to
$\mathcal J$, and exactness means that every arrow factors as its
normal coimage followed by its normal image, up to isomorphism.
Equip $\A$ with the normal-epimorphism--normal-monomorphism system,
and equip the category $\Mlc$ of bounded modular lattices and
modular connections with its upper-interval-quotient--lower-interval-embedding
system. Then Grandis's canonical functor
\[
 S_{\A}:\A\longrightarrow\Mlc,\qquad
 A\longmapsto\Nsb_{\A}(A),\qquad
 f\longmapsto(f_!\dashv f^*)
\]
is a transfer functor, where $f_!$ and $f^*$ are normal direct and
inverse images. Its pullback of the canonical form on $\Mlc$ is the
normal-subobject form of $\A$.
\end{corollary}
\begin{proof}
The cited normal-subobject form on $\A$ is noetherian, has modular
fibres, and induces the normal-epimorphism--normal-monomorphism
system. Its functor $T_{\N}$ in
\cref{prop:modular-form-representation} is precisely Grandis's
functor $S_{\A}$. That proposition gives both the transfer property
and the stated equality of forms.
\end{proof}

The scope of Grandis's terminology is broader. For a semiexact category,
his functor takes values in $\Ltc$, the category of bounded lattices
and arbitrary Galois connections, and can fail our conservativity
condition even when it factors through $\Mlc$. The following
counterexample makes explicit the topological-vector-space example
in \cite[Section~1.7]{Grandis1992}.

\begin{counterexample}[A nonconservative Grandis transfer functor]\label{ex:grandis-semiexact}
Let $\mathbf{TVect}_{\mathbb R}$ be the category of real topological
vector spaces, without a Hausdorff requirement, and continuous linear
maps, with the usual ideal of zero maps. Grandis's normal-subobject
functor
\[
 \Nsb:\mathbf{TVect}_{\mathbb R}\longrightarrow\Mlc
\]
is not conservative, and hence cannot be a transfer functor in the
sense of \cref{def:transfer-functor}.
\end{counterexample}
\begin{proof}
Normal subobjects are all linear subspaces with their induced
topologies: a subspace is the kernel of the quotient by that subspace,
which is allowed to be non-Hausdorff. Normal direct and inverse images
are the ordinary linear image and preimage. In particular the continuous
linear identity
\[
 \mathbb R_{\mathrm{usual}}\longrightarrow
 \mathbb R_{\mathrm{indiscrete}}
\]
induces the identity connection between the two-element subspace
lattices. It is not a homeomorphism, so $\Nsb$ does not reflect
isomorphisms.
\end{proof}

Thus Grandis's canonical functors for exact categories are examples of
our class, but his semiexact transfer construction does not in general
define an example. Our definition concerns arbitrary functors between
categories with specified factorization systems, rather than only
canonical normal-subobject functors.

\subsection{Algebraic forgetful and free functors}

For a finitary variety, we use all set-based models, including empty
carriers when its operations and equations permit them. Its specified
factorization system consists of sortwise surjective and sortwise
injective homomorphisms. In the single-sorted case these are the usual
surjective and injective homomorphisms. The left class consists of
regular epimorphisms; it need not contain every categorical epimorphism,
as the example of rings shows.

Preservation of surjective pushouts by algebraic forgetful functors and
the resulting existence of noetherian forms are established in
\cite[Theorem~6.1 and Corollary~6.2]{VN}. The congruence-join fact behind
the preservation statement is \cite[Chapter~II, Theorem~5.3]{BS1981}.
We state the sortwise extension once, retaining the sort of every
element as a label when passing to a single set.
\begin{proposition}[Underlying-set functors]
\label[proposition]{prop:variety-forgetful}
\label[proposition]{top:lem:variety-pushout}
\label[proposition]{prop:many-sorted-transfer}
Let $I$ be a set of sorts and let $\V$ be a finitary $I$-sorted
variety of all set-based algebras, allowing empty carriers when its
operations and equations permit them. Equip $\V$ with the
sortwise-surjective--sortwise-injective factorization system and
$\Set$ with the surjection--injection system. The functor
\[
 U_\#: \V\longrightarrow\Set,\qquad
 U_\#A=\coprod_{i\in I}A_i,\qquad
 U_\#f(i,a)=(i,f_i(a))
\]
is a transfer functor. Consequently, $\V$ admits a noetherian form
inducing its specified factorization system. In particular, for a
single-sorted variety, the ordinary underlying-set functor is a
transfer functor and preserves pushouts of two surjective
homomorphisms.
\end{proposition}
\begin{proof}
The carrier functor $\V\to\Set^I$ is conservative, creates limits,
and preserves the specified classes. Given two sortwise surjections
from an algebra $A$, let $\rho_i$ and $\sigma_i$ be their kernel
equivalence relations on each carrier $A_i$. Their equivalence-relation
joins, formed sortwise, constitute a congruence: change the finitely
many inputs of an operation one at a time along finite
$\rho$--$\sigma$ chains. Quotienting by this congruence computes
the pushout both in $\V$ and on every carrier. Thus the carrier
functor preserves the required pushouts.

The disjoint-union functor $\Set^I\to\Set$ is conservative on
sort-preserving maps and preserves the two classes, pullbacks, and
pushouts. For pullbacks, equality of two tagged elements forces
equality of their sorts; colimits are computed by identifying elements
within each sort. The composite is therefore a transfer functor.
Pull back a noetherian form on $\Set$ and apply \cref{thm:transfer}.
Taking $I$ to be a singleton gives the single-sorted assertion.
\end{proof}

The nonempty-intersection property used here is classical for set
functors; see \cite[Section~2]{Gumm2020} for Trnkov\'a's theorem.
The next statement combines it with the standard free-algebra
construction \cite[Chapter~II, Sections~10--11]{BS1981}, while making
the empty-intersection issue explicit.
\begin{theorem}[Free-algebra functors]\label{thm:free-variety}
Let $\V$ be a nontrivial single-sorted finitary variety of all
set-based algebras, including empty algebras when permitted by its
signature, and let $F_\V:\Set\to\V$ be its free-algebra functor.
Equip $\Set$ with the surjection--injection system and $\V$ with
the system of surjective and injective homomorphisms.
\begin{enumerate}[label=\textup{(\roman*)}]
\item If $\V$ has a closed term, then $F_\V$ is a transfer functor.
\item Fix a singleton $\{*\}$. Without any assumption on closed terms,
the functor $H:\Set\to\V$ defined on sets $X$ and functions $f$ by
\[
 H(X)=F_\V(X\sqcup\{*\}),\qquad
 H(f)=F_\V(f\sqcup\id_{\{*\}})
\]
is a transfer functor.
\end{enumerate}
Both functors in these assertions preserve all pushouts and preserve and
reflect injections and surjections.
\end{theorem}
\begin{proof}
Nontriviality means that some algebra has at least two elements. We first
prove (ii). Surjections of generating sets induce surjections of free
algebras. An injection of the pointed generating sets has a retraction
obtained by sending unused generators to $*$, so its image under $F_\V$
is injective.

If $f$ identifies two elements, the corresponding distinct free generators
of $H(X)$ are identified by $H(f)$. They are distinct because assignments
to a nontrivial algebra distinguish them. If $f$ misses $y\in Y$, two
assignments of the generators of $H(Y)$, agreeing everywhere except at
$y$, induce distinct homomorphisms that agree after $H(f)$. Thus $H(f)$
is not epic, and in particular is not surjective. These arguments give
reflection of both classes and conservativity.

The operation $X\mapsto X\sqcup\{*\}$ preserves pushouts: the added
points on the two sides are identified through the added point of their
common domain. Since $F_\V$ is a left adjoint, $H$ preserves pushouts.
For injections represented by subsets $A,B\subseteq X$, regard $H(A)$
and $H(B)$ as subalgebras of $H(X)$. The endomorphism that fixes the
generators of $A\sqcup\{*\}$ and sends all other generators to $*$
fixes $H(A)$ pointwise and sends $H(B)$ into $H(A\cap B)$. Therefore
\[
 H(A)\cap H(B)=H(A\cap B),
\]
which proves preservation of pullbacks of two injections.

For (i), replace every use of the extra generator by a fixed closed term.
This gives retractions for injections of free algebras, including those
whose generating domain is empty, and the same intersection argument.
The preceding generator-separation arguments apply unchanged, while
preservation of pushouts follows directly from the free-algebra adjunction.
\end{proof}

\begin{remark}[The empty-intersection issue]\label{rem:free-empty-intersection}
A closed term is sufficient, but not necessary: the ordinary free-semigroup
functor, for example, preserves intersections because a word belongs to
both generated subsemigroups exactly when all its letters belong to both
sets of generators. The condition cannot simply be omitted for arbitrary
varieties. Consider the variety with one unary operation $p$ and identity
$p(x)=p(y)$. Its free algebra on a nonempty set $X$ is $X\sqcup\{c\}$,
with $p$ constant at $c$, whereas its free algebra on the empty set is
empty. In the free algebra on $\{a,b\}$, the free subalgebras on
$\{a\}$ and $\{b\}$ intersect in $\{c\}$. Thus their intersection is
not the free algebra on their empty intersection. The additional generator
in part~\textup{(ii)} removes this obstruction uniformly.
\end{remark}

The existence of noetherian forms for varieties is already
\cite[Corollary~6.2]{VN}. The additional assertion here is their
universal-target property.
\begin{theorem}[Varieties as universal targets]\label{thm:varieties}
Let $\V$ be a nontrivial single-sorted finitary variety of all
set-based algebras, including empty algebras when its signature permits
them, with the factorization system of surjective and injective
homomorphisms. Then $\V$ is a universal target for existence of a
noetherian form.
\end{theorem}
\begin{proof}
Proposition~\ref{prop:variety-forgetful} and
Theorem~\ref{thm:free-variety} give transfer functors in both directions
$\V\rightsquigarrow\Set\rightsquigarrow\V$.
Apply \cref{prop:target}.
\end{proof}

Examples include groups, monoids, semigroups, Boolean algebras, bounded
lattices, unital rings, and modules over every nonzero unital ring. The
ordinary free functor works for all these examples; for semigroups this
follows from the word description in
\cref{rem:free-empty-intersection}.
The transferred form need not be the ordinary subobject form, and
semi-abelianity is not required. Finitarity is used in the congruence-join
argument for the underlying-set functor and cannot be dropped from that
argument without a separate proof.

\subsection{Modular-fibre replacement and the lattice target}\label{sec:modular-replacement}

The free-functor theorem has a further consequence for the possible
fibre lattices. It strengthens the existence assertion
without imposing modularity on the given form. We use the classical
completeness and algebraicity of subalgebra lattices \cite{BS1981},
together with the modular law for subgroups of an abelian group.

\begin{theorem}[Modular-fibre replacement]\label{thm:modular-replacement}
Let $\N$ be a noetherian form on a category $\C$, with set-sized
fibre lattices $\Lambda_X$, transports $f_!\dashv f^*$, and induced
factorization system $(\E,\M)$. Define $H:\C\to\Ab$ by
\[
 H(X)=\mathbb Z^{(\Lambda_X)},\qquad
 H(f)[S]=[f_!S]\quad(f:X\to Y,\ S\in\Lambda_X),
\]
where $\mathbb Z^{(\Lambda_X)}$ is the free abelian group with basis
the underlying set of $\Lambda_X$. Then the form $\N^{\mathrm{mod}}$
with fibres and transports
\begin{equation}\label{eq:modular-replacement}
 \N^{\mathrm{mod}}(X)=\Sub(HX),\qquad
 f^{\mathrm{mod}}_!(A)=H(f)(A),\qquad
 (f^{\mathrm{mod}})^*(B)=H(f)^{-1}(B)
\end{equation}
for $A\le HX$ and $B\le HY$ is noetherian and induces precisely
$(\E,\M)$. Every fibre is a complete algebraic modular lattice.
Here a complete lattice is algebraic when each of its elements is a
join of compact elements, and $c$ is compact when
$c\le\bigvee\mathcal S$ implies $c\le\bigvee\mathcal S_0$ for
some finite subset $\mathcal S_0\subseteq\mathcal S$.
\end{theorem}
\begin{proof}
The functor $H$ is the composite of the cluster-set transfer functor
of \cref{thm:cluster} and the free abelian group transfer functor
of \cref{thm:free-variety}, applied to $\Ab$. Their composite is a
transfer functor.
Pulling back the usual subgroup form on $\Ab$ therefore gives
\eqref{eq:modular-replacement}, with the prescribed factorization
system, by \cref{thm:transfer}.

For any abelian group $V$, the lattice $\Sub(V)$ is complete:
meets are intersections and joins are subgroup sums. For subgroups
$A,B,C\le V$ with $A\le C$, the elementary identity
\[
 (A+B)\cap C=A+(B\cap C)
\]
proves modularity. Every subgroup is the join of its finitely generated
subgroups. These are compact, since expressing finitely many generators
in a sum of subgroups involves only finitely many summands. Hence
$\Sub(V)$ is algebraic as well.
\end{proof}

\begin{corollary}[Modularity imposes no additional existence condition]\label{cor:modular-existence}
Let $\C$ be a category with a specified proper factorization system
$(\E,\M)$. The following conditions are equivalent:
\begin{enumerate}[label=\textup{(\roman*)}]
\item $\C$ admits a noetherian form inducing $(\E,\M)$;
\item $\C$ admits such a form with modular fibre lattices;
\item $\C$ admits such a form with complete algebraic modular fibre
lattices.
\end{enumerate}
No separate existence assumption on homogeneous diagrams is needed.
\end{corollary}
\begin{proof}
The implication (i)$\Rightarrow$(iii) is
\cref{thm:modular-replacement}, and the other implications are immediate.
Homogeneous diagrams exist whenever (i) holds, by
\cref{lem:homogeneous-exist}.
\end{proof}

The replacement preserves the represented subobject and quotient orders,
because it preserves their defining factorization classes. It need not
recover the original fibre lattices. Indeed, on the terminal category,
any set-sized bounded lattice with identity transport defines a
noetherian form: only bottom is normal and only top is conormal.
Taking a nonmodular lattice shows why modularity is a property of the
replacement, rather than a necessary property of a specified form.
Nor must every cluster of the replacement be normal or conormal:
the witnessing morphisms must belong to the source category.

For the target-category formulation, the functor
$S_{\Ab}:\Ab\to\Mlc$ below is Grandis's classical subobject
transfer functor \cite[Introduction, Section~0.3]{Grandis2012}, rather
than a new construction. Its transfer property in our terminology is
\cref{prop:grandis-transfer}; the universal-target conclusion follows
by composing it with our set and abelian representations.

\begin{corollary}[Modular lattices as a universal target]\label{cor:mlc-target}
Let $\Mlc$ be the category of set-sized bounded modular lattices and
modular connections, with its upper-interval-quotient--lower-interval-embedding
factorization system. Let $\C$ have a specified proper factorization
system $(\E,\M)$, pullbacks of two $\M$-arrows with common
codomain, and pushouts of two $\E$-arrows with common domain.
Then $\C$ has a noetherian form inducing $(\E,\M)$ if and only
if there is a transfer functor $\C\to\Mlc$ for these systems.
Thus $\Mlc$ is a universal target. An explicit transfer functor from
sets, equipped with their surjection--injection system, is
\[
 \Set\xrightarrow{\mathbb Z^{(-)}}\Ab
 \xrightarrow{S}\Mlc,
 \qquad X\longmapsto\Sub\bigl(\mathbb Z^{(X)}\bigr),
\]
where $\Ab$ has its surjection--injection system, $S=S_{\Ab}$ is
Grandis's subobject functor, and $S(h)$ is subgroup direct image paired with subgroup inverse
image along the homomorphism $h$.
\end{corollary}
\begin{proof}
The two displayed functors are transfer functors by
\cref{thm:free-variety,prop:grandis-transfer}. The target has its canonical
noetherian form by \cref{lem:mlc-form}, so \cref{prop:target} applies.
Equivalently, starting from a form $\N$ on $\C$, the composite
$S_{\Ab}H$, for $H$ in \cref{thm:modular-replacement}, is a
transfer functor whose pullback of the canonical form on $\Mlc$
is exactly $\N^{\mathrm{mod}}$.
\end{proof}

When the given form already has modular fibres,
\cref{prop:modular-form-representation} recovers that very form from
$\Mlc$. For an arbitrary form, the composite through $\Ab$ instead
recovers the replacement of \cref{thm:modular-replacement}. This is
the distinction between representing a specified form and characterizing
the existence of some compatible form.

\subsection{Essentially algebraic categories and algebraic envelopes}
\label{sec:essentially-algebraic}

Essentially algebraic theories allow partial operations whose domains
are specified by equations. We use the usual small, many-sorted
theories, with a set bound on operation arities. Their categories of
models are precisely the locally presentable categories; finitary
essentially algebraic theories correspond to locally finitely
presentable categories \cite[Theorem~3.36]{AR1994}.
Every locally presentable category has all small limits and colimits
and the proper factorization system of strong epimorphisms and
monomorphisms \cite[Proposition~1.61]{AR1994}. These facts supply the
source structure for the following application of our transfer theorem.

\subsubsection{A conservative left-adjoint criterion}

By \cref{prop:many-sorted-transfer}, every finitary many-sorted variety
has a noetherian form for its sortwise-surjective--sortwise-injective
system. We now transfer this form along a suitable left adjoint.

The next criterion combines the cited factorization theorem with
\cref{thm:transfer}. Its left-adjoint hypothesis supplies both
preservation of strong epimorphisms and preservation of quotient
pushouts, so these need not be verified separately in applications.

\begin{theorem}[Essentially algebraic existence criterion]
\label{thm:essentially-algebraic}
Let $\C$ be an essentially algebraic category, equipped with its
strong-epimorphism--monomorphism factorization system, and let $\V$
be a finitary many-sorted variety, equipped with its
sortwise-surjective--sortwise-injective factorization system.
Suppose that a functor $L:\C\to\V$ has a right adjoint and that
\begin{enumerate}[label=\textup{(\roman*)}]
\item $L$ reflects isomorphisms;
\item $L$ preserves monomorphisms;
\item $L$ preserves pullbacks of two monomorphisms with a common
codomain.
\end{enumerate}
Then $L$ is a transfer functor. The category $\C$ admits a
noetherian form whose quotients are exactly its strong epimorphisms
and whose embeddings are exactly its monomorphisms. Such a form can
be chosen with complete algebraic modular fibres.
\end{theorem}
\begin{proof}
Local presentability supplies the stated proper factorization system
and all the required pullbacks and pushouts. If $R$ is right adjoint
to $L$, then $R$ preserves monomorphisms. Under the adjunction, a
lifting square from $Le$ to a monomorphism in $\V$ corresponds to
a lifting square from $e$ to its image under $R$. Thus $L$ takes
strong epimorphisms to strong epimorphisms. In a variety the latter
are precisely the sortwise surjections, by the usual image
factorization. Together with (ii), this proves preservation of the
specified factorization systems.

Now \cref{prop:transfer-elementary}, together with (i) and (iii),
makes $L$ a transfer functor. Apply
\cref{prop:many-sorted-transfer,thm:transfer} to obtain a noetherian
form on $\C$, and then \cref{thm:modular-replacement} to obtain
complete algebraic modular fibres without changing the two classes.
\end{proof}

In applications, $L$ can be a universal algebraic envelope of the
partial operations. A normal-form description is useful when it shows
that envelopes of subobjects intersect in the envelope of their
intersection. Existence and uniqueness of normal forms alone do not
give the theorem: their compatibility with subobjects, and reflection
of isomorphisms, must also be checked. The following example carries
out these checks.

\subsubsection{Small categories and their semigroup envelopes}

The Karoubi envelope of a semigroup is a standard construction;
we use the description in \cite[Section~3.2]{CostaSteinberg2015}.
We give the presentation of its left adjoint and the normal-form
argument explicitly, since retaining each category identity as a
separate generator is essential for conservativity.

\begin{proposition}[The semigroup transfer for small categories]
\label{prop:cat-semigroup-transfer}
Let $\mathbf{Cat}$ be the category of small categories and functors,
with its strong-epimorphism--monomorphism factorization system. Let
$\mathbf{Sgp}$ be the variety of semigroups, including the empty
semigroup, with its surjection--injection factorization system.
For a small category $C$, define
\[
 S(C)=\left\langle [f]\ (f\in\operatorname{Mor}C)
 \ \middle|\ [g][f]=[g\circ f]
 \text{ whenever }g\circ f\text{ is defined}\right\rangle,
\]
and for a functor $F:C\to D$ let $S(F)[f]=[Ff]$.
Then $S:\mathbf{Cat}\to\mathbf{Sgp}$ is a conservative left
adjoint preserving monomorphisms and pullbacks of two monomorphisms
with a common codomain, and hence is a transfer functor. Here the monomorphisms in $\mathbf{Cat}$ are
the functors injective on objects and faithful; the strong
epimorphisms are the functors surjective on objects whose image
arrows generate the codomain under composition.
\end{proposition}
\begin{proof}
\emph{Normal forms.} The presentation uses nonempty words and does
not identify identity arrows with an empty word. Replace any adjacent
composable pair by its composite. Every reduction shortens the word.
Disjoint reductions commute, and overlapping reductions concern a
composable triple, where associativity gives the same composite.
Equivalently, a word has a canonical reduction obtained by composing
its maximal composable blocks: composing within a block preserves
its outer source and target, so the noncomposable boundaries remain.
Thus every element of $S(C)$ has a unique reduced word, with no
adjacent composable pair. In particular, distinct arrows of $C$
give distinct elements, including its distinct identity arrows.

\emph{The adjunction.} For a semigroup $T$, let $K(T)$ have as
objects its idempotents. An arrow $e\to d$ is a tagged element
$(d,t,e)$ with $dt=t=te$; multiplication gives composition and
$(e,e,e)$ is the identity at $e$. A homomorphism $\varphi:S(C)\to T$
determines a functor $C\to K(T)$ by
\[
 x\longmapsto\varphi([1_x]),\qquad
 f:x\to y\longmapsto
 \bigl(\varphi([1_y]),\varphi([f]),\varphi([1_x])\bigr).
\]
Conversely, a functor to $K(T)$ gives a homomorphism by assigning
to $[f]$ the middle element of its image. These assignments are
mutually inverse and natural, proving $S\dashv K$.

\emph{Monomorphisms and intersections.} A functor injective on
objects and faithful is, up to isomorphism, the inclusion of a
subcategory. Such an inclusion preserves and reflects composability
of its arrows. Its map under $S$ therefore includes precisely the
reduced words all of whose letters belong to that subcategory. If
$A,B$ are subcategories of $C$, then, inside $S(C)$,
\[
 S(A)\cap S(B)=S(A\cap B).
\]
This proves preservation of monomorphisms and pullbacks of two
monomorphisms with a common codomain, including when $A\cap B$
is empty.

\emph{Conservativity.} Suppose $S(F)$ is an isomorphism. Its
injectivity on one-letter words makes $F$ injective on arrows and,
by considering identities, on objects. Object injectivity makes
$F$ reflect composability, so $S(F)$ preserves lengths of reduced
words. Every target arrow is a one-letter word. Surjectivity of
$S(F)$ therefore makes it the image of a one-letter word from $C$.
Thus $F$ is bijective on arrows and surjective on objects as well.
It is consequently an isomorphism of categories.

For completeness, every functor factors through the subcategory
generated by its image, followed by the inclusion. The first factor
is surjective on objects and its image arrows generate. It has the
unique lifting property against subcategory inclusions: lifts are
forced on image arrows and hence on their composites. Object
injectivity ensures matching endpoints for the lifted composites,
and equality of composites can be checked after the inclusion. This proves the
stated description of the strong-epimorphism--monomorphism system.
The category $\mathbf{Cat}$ is finitary essentially algebraic
(composition is defined on the pullback of source and target).
\Cref{thm:essentially-algebraic} now applies.
\end{proof}

\begin{corollary}[Categories, groupoids, and semicategories]
\label{cor:cat-gpd-semicat}
Let $\C$ be any one of the following categories: small categories
and functors; small groupoids and functors; or small semicategories
and semifunctors. A semicategory here has a set of objects, a set
of arrows with source and target, and associative composition of
composable arrows, with no identity requirement; semifunctors
preserve this data. Then $\C$ admits a noetherian form inducing its
strong-epimorphism--monomorphism factorization system and having
complete algebraic modular fibres.
\end{corollary}
\begin{proof}
All three categories are finitary essentially algebraic. The case
of categories follows from \cref{prop:cat-semigroup-transfer} and
\cref{thm:essentially-algebraic}.

The inclusion of groupoids into categories is left adjoint to the
functor taking the maximal subgroupoid: every functor from a groupoid
sends arrows to invertible arrows. It is conservative and preserves
monomorphisms and their pullbacks. Composing this inclusion with
$S$ gives the required left adjoint into $\mathbf{Sgp}$.

For semicategories, freely adjoin a new identity at each object.
This defines a functor $J$ to $\mathbf{Cat}$, left adjoint to
forgetting the identity requirement. Its arrow set is the disjoint
union of the old arrows and the newly adjoined identities; these
remain distinct even if an old arrow already acts as an identity.
This description shows that $J$ reflects isomorphisms and preserves
monomorphisms and their pullbacks: in a pullback, old arrows can
match only old arrows and new identities only new identities.
The composite $SJ$ again satisfies
\cref{thm:essentially-algebraic}.
\end{proof}

\subsubsection{Diagrams, slices, and coslices}

The closure of locally presentable categories under small diagram
categories, slices, and coslices is standard
\cite[Corollary~1.54 and Proposition~1.57]{AR1994}. The next result
shows that the additional left-adjoint criterion is retained by
these constructions.

\begin{proposition}[Closure of the left-adjoint criterion]
\label{prop:essentially-algebraic-closure}
Let $\C$ be an essentially algebraic category, let $\V$ be a
finitary many-sorted variety, and let $L:\C\to\V$ be a
conservative left adjoint which preserves monomorphisms and
pullbacks of two monomorphisms with a common codomain. Let
$\mathcal D$ be a small category and $B$ an object of $\C$.
Each of
\[
 [\mathcal D,\C],\qquad \C/B,\qquad B/\C
\]
is essentially algebraic and admits a conservative left adjoint
into a finitary many-sorted variety preserving monomorphisms and
pullbacks of two monomorphisms with a common codomain.
Consequently, each has a noetherian form with
complete algebraic modular fibres inducing its
strong-epimorphism--monomorphism factorization system.
\end{proposition}
\begin{proof}
Write $L\dashv R$, with unit $\eta$. For diagrams, use
\[
 [\mathcal D,L]:[\mathcal D,\C]\longrightarrow[\mathcal D,\V].
\]
It is left adjoint to $[\mathcal D,R]$, and conservativity and
pullback preservation are pointwise. Monomorphisms in these diagram
categories are pointwise monomorphisms: each evaluation functor has
a left adjoint formed using coproducts, so it preserves
monomorphisms. The target is a finitary many-sorted variety. Its
sorts are pairs consisting of an object of $\mathcal D$ and a
sort of $\V$; unary operations encode arrows of $\mathcal D$,
and equations impose functoriality and compatibility with the
operations of $\V$.

For slices, the domain functor $P:\C/B\to\C$ is left adjoint
to $X\mapsto(B\times X\to B)$. It is conservative and preserves
monomorphisms and pullbacks. Thus $LP$ has the required properties.

For coslices, use
\[
 L_B:B/\C\longrightarrow L(B)/\V,
 \qquad (B\xrightarrow{b}X)\longmapsto
 (L(B)\xrightarrow{L(b)}L(X)).
\]
Its right adjoint sends $L(B)\xrightarrow{v}Y$ to
$B\xrightarrow{\eta_B}RL(B)\xrightarrow{R(v)}R(Y)$.
Conservativity follows from that of $L$. The forgetful functors
from the coslices create pullbacks. They also preserve and reflect
monomorphisms: to test the underlying map against two maps out of
$Z$, use the object $B\to B\sqcup Z$ of the coslice.
Hence $L_B$ preserves monomorphisms and pullbacks of two
monomorphisms with a common codomain.
The category $L(B)/\V$ is a finitary many-sorted variety: adjoin
a constant for every element of every carrier of $L(B)$ and
equations recording all its operations. These constants specify
exactly a homomorphism from $L(B)$.

The cited closure results give essential algebraicity, and
\cref{thm:essentially-algebraic} gives the asserted forms.
\end{proof}

The criterion includes every finitary many-sorted variety, by taking
its identity functor. With $\C=\Set$, the diagram construction
therefore gives all small presheaf categories, including simplicial
sets $[\Delta^{\op},\Set]$. Simplicial sets are distinct from the
abstract simplicial complexes of \cref{tc:cor:simplicial}.
With $\C=\mathbf{Cat}$ or $\mathbf{Gpd}$, it gives simplicial
categories and simplicial groupoids. Slices give categories over a
fixed category, while $B/\mathbf{Cat}$ consists of categories
equipped with a specified functor from $B$. The constructions may
be iterated, and in every case the two classes are strong
epimorphisms and monomorphisms.

This is a sufficient criterion for essentially algebraic categories,
not a conclusion from essential algebraicity alone. The argument
requires an algebraic envelope with the stated properties; it does
not establish that every essentially algebraic category has such
an envelope, or settle existence of a noetherian form in full
generality.

\subsection{Geometric theories and inherited forms}
\label{sec:geometric-theories}

A full subcategory need not be closed under all limits and colimits
to inherit a noetherian form. Closure under the subobjects and
quotients of the given form suffices. We use the standard stability
properties of orthogonal factorization systems \cite{AHS} together
with \cref{lem:homogeneous-exist,thm:transfer}.

\begin{proposition}[Inheritance by subobjects and quotients]
\label{prop:closed-subcategory}
Let $\F$ be a noetherian form on a category $\D$, with induced
proper factorization system $(\E,\M)$, and let $\C$ be a full,
replete subcategory of $\D$. Suppose that
\begin{enumerate}[label=\textup{(\roman*)}]
\item if $m:A\to X$ belongs to $\M$ and $X$ belongs to $\C$,
then $A$ belongs to $\C$;
\item if $e:X\to B$ belongs to $\E$ and $X$ belongs to $\C$,
then $B$ belongs to $\C$.
\end{enumerate}
Then the restrictions of $\E$ and $\M$ to $\C$ form a proper
factorization system, the inclusion $J:\C\hookrightarrow\D$
is a transfer functor, and $J^*\F$ is a noetherian form inducing
that system. If the fibres of $\F$ are complete, algebraic, or
modular, the restricted form has the corresponding properties.
\end{proposition}
\begin{proof}
For a morphism of $\C$, its $(\E,\M)$-factorization in $\D$
has intermediate object in $\C$, by either closure assumption.
Fullness retains orthogonality, so these factorizations define the
asserted proper factorization system.

A pullback of two $\M$-morphisms between objects of $\C$
exists in $\D$ by \cref{lem:homogeneous-exist}. Its projections
belong to $\M$, so its vertex belongs to $\C$ by (i).
Dually, a pushout of two $\E$-morphisms has coprojections in
$\E$, so its vertex belongs to $\C$ by (ii). Fullness makes
these diagrams pullbacks and pushouts in $\C$ as well.
The inclusion preserves the two classes and these diagrams, and
reflects isomorphisms. It is therefore a transfer functor.
Apply \cref{thm:transfer}; the fibres themselves are unchanged.
\end{proof}

Geometric formulas are built from atomic formulas using finite
conjunctions, arbitrary set-indexed disjunctions, and existential
quantification; their free variables lie in a finite context
\cite[Section~2.1.1]{Caramello2018}. We use set-based models, permitting
empty carriers when the axioms permit them, and the usual
homomorphisms. The following sufficient condition uses the fragment
without existential quantifiers or nontrivial antecedents.
The ambient existence result is \cite[Corollary~6.2]{VN}, in the
many-sorted form of \cref{prop:many-sorted-transfer}.

\begin{theorem}[A sufficient geometric condition]
\label{thm:geometric-equational}
Let $\mathbb T_0$ be a finitary algebraic theory with a set of
sorts and a set-sized signature. Let $\mathbb T$ be obtained by
adding to $\mathbb T_0$, in the same signature, a set of axioms
of the form
\begin{equation}\label{eq:geometric-equational}
 \top\ \vdash_{\vec x}\quad
 \bigvee_{i\in I}\ \bigwedge_{j=1}^{n_i}
 t_{ij}(\vec x)=s_{ij}(\vec x),
\end{equation}
where $\vec x$ is a finite context of sorted variables, $I$ is
a set, each $n_i$ is finite, and each displayed equation is
well-sorted. Then the category
$\operatorname{Mod}_{\Set}(\mathbb T)$ of set-based models
and homomorphisms admits a noetherian form whose quotients are
exactly the sortwise-surjective homomorphisms and whose embeddings
are exactly the sortwise-injective homomorphisms. Its fibres can
be chosen complete, algebraic and modular. Moreover, the full
inclusion into $\operatorname{Mod}_{\Set}(\mathbb T_0)$ is a
transfer functor for these factorization systems.
\end{theorem}
\begin{proof}
Put $\V=\operatorname{Mod}_{\Set}(\mathbb T_0)$ and
$\C=\operatorname{Mod}_{\Set}(\mathbb T)$.
The category $\V$ has a noetherian form for the stated system by
\cref{prop:many-sorted-transfer}; by
\cref{thm:modular-replacement} we may choose one with complete
algebraic modular fibres.

The full, replete subcategory $\C$ is closed under subalgebras:
given a tuple in a subalgebra, a disjunct of
\eqref{eq:geometric-equational} holds in the ambient algebra,
and its equations hold in the subalgebra because the inclusion
reflects equality. It is also closed under sortwise-surjective
homomorphic images. Lift the finitely many components of a tuple
to the domain, choose a disjunct satisfied there, and apply the
homomorphism to its equations. Thus
\cref{prop:closed-subcategory} applies and proves all assertions.
\end{proof}

This criterion does not require closure under products, and the
resulting model categories need not be varieties. The exponent or
finite system of equations selected by a disjunction may depend on
the tuple. The following examples are all direct specializations,
with their ordinary homomorphisms.

\begin{example}\label{ex:geometric-algebraic-examples}
Each of the following categories has a noetherian form with
surjective homomorphisms as quotients, injective homomorphisms as
embeddings, and complete algebraic modular fibres:
\begin{enumerate}[label=\textup{(\roman*)}]
\item Torsion groups, by adding to the group axioms
\[
 \top\vdash_x\bigvee_{n\geq1}x^n=1.
\]
For a fixed prime $p$, groups in which every element has
$p$-power order are obtained instead from
$\top\vdash_x\bigvee_{r\geq0}x^{p^r}=1$.
\item Periodic monoids, meaning monoids in which each element
generates a finite submonoid, by adding
\[
 \top\vdash_x\bigvee_{1\leq m<n}x^m=x^n.
\]
\item Nil associative rings without a prescribed multiplicative
identity, by adding to the corresponding ring axioms
\[
 \top\vdash_x\bigvee_{n\geq1}x^n=0.
\]
\item Locally nilpotent associative rings without a prescribed
multiplicative identity. For each $k\geq1$, impose
\[
 \top\vdash_{x_1,\ldots,x_k}
 \bigvee_{N\geq1}\,
 \bigwedge_{(i_1,\ldots,i_N)\in\{1,\ldots,k\}^{N}}
 x_{i_1}\cdots x_{i_N}=0.
\]
Each conjunction is finite. Its equations say precisely that the
subring generated by the tuple is nilpotent: longer products also
vanish, and every element of that subring is an integer linear
combination of nonempty words in the generators.
\end{enumerate}
In every case the assertion follows from
\cref{thm:geometric-equational}.
\end{example}

The requirement on the axioms cannot be replaced merely by
geometricity, or even regularity. We use here the standard fact
that a split monomorphism has the right lifting property against
every epimorphism \cite{AHS}, together with
\cref{lem:homogeneous-exist}.

\begin{counterexample}[The theory of inhabited sets]
\label{ex:inhabited-obstruction}
Let $\mathbb T$ be the one-sorted theory with no operations or
relations and with the regular, hence geometric, axiom
\[
 \top\vdash\exists x\,\top.
\]
Its category of set-based models and homomorphisms is
$\Set_{\ne\varnothing}$, the category of nonempty sets and all
functions. This category admits no noetherian form, for any
factorization system.
\end{counterexample}
\begin{proof}
If a noetherian form existed, every split monomorphism would
belong to its embedding class: this is the right class of a
proper factorization system, whose left class consists of
epimorphisms. Thus every pair of split monomorphisms with a common
codomain would have a pullback by \cref{lem:homogeneous-exist}.
However, the two inclusions
\[
 \{*\}\xrightarrow{\ i_0\ }\{0,1\}
 \xleftarrow{\ i_1\ }\{*\},
 \qquad i_0(*)=0,\quad i_1(*)=1,
\]
are split monomorphisms with no cone in $\Set_{\ne\varnothing}$:
a cone would equate the two distinct constant functions from a
nonempty set to $\{0,1\}$. In particular, they have no pullback.
\end{proof}

Specifying a constant instead gives pointed sets, a variety to
which \cref{prop:many-sorted-transfer} applies. The distinction
between an existential witness and a specified witness therefore
matters for existence. Likewise, the classifying topos of a
geometric theory and its category of set-based models are different
categories \cite[Section~2.1.2]{Caramello2018}; the existence theorem for topoi
\cite{JVN} concerns the former and does not imply existence for
the latter.

\subsection{Monadic functors and ideally exact categories}
\label{sec:ideally-exact}

For right adjoints, the pullback problem was characterized by van Niekerk
\cite[Theorem~5.5]{VN}. We recall his theorem to make explicit its relation
to \cref{thm:transfer}; here, as throughout this paper, a pushout of two
quotients means that the two quotients have a common domain.

\begin{theorem}[Monadic pullback criterion]\label{thm:monadic-pullback}
Let $\C$ and $\D$ be categories, let $U:\C\to\D$ have a left
adjoint $L:\D\to\C$, and let $\F$ be a noetherian form on $\D$.
Give the pullback $U^*\F$ the fibres $\F(UX)$ and the transports
along $Uf$ for each morphism $f$ of $\C$. Then $U^*\F$ is
noetherian if and only if $U$ is monadic and the induced monad
$UL:\D\to\D$ preserves pushouts of two $\F$-quotients with
common domain.
\end{theorem}

The following preservation statement uses the regular-pushout
characterization of exact Mal'tsev categories
\cite[Theorem~5.7]{CKP1993}; see also
\cite[Section~2.2 and Lemma~3.16]{SVL2018}.
A \emph{regular pushout} is a commutative square of regular epimorphisms
whose comparison from its upper-left vertex to the pullback of the other
three vertices is a regular epimorphism.

\begin{lemma}\label{lem:maltsev-pushouts}
A functor between Barr-exact Mal'tsev categories that preserves pullbacks
and regular epimorphisms preserves pushouts of two regular epimorphisms.
Consequently, every conservative regular functor between such categories
is a transfer functor for their regular-epimorphism--monomorphism systems.
\end{lemma}
\begin{proof}
In an exact Mal'tsev category, pushouts of two regular epimorphisms exist,
and a commutative square of regular epimorphisms is a pushout if and only
if it is a regular pushout. Preservation of pullbacks and regular
epimorphisms preserves the comparison morphism and its regular-epimorphism
property, so preserves these pushouts. For the last assertion, a regular
functor preserves finite limits and regular epimorphisms, hence also
monomorphisms and the required pullbacks. Conservativity supplies the
remaining transfer axiom.
\end{proof}

An \emph{ideally exact category} is Barr exact, is protomodular, has finite
coproducts, and has a regular epimorphism $0\to1$ from its initial to its
terminal object \cite[Theorem~3.1 and Definition~3.2]{GJIdeallyExact}.
The next result is a consequence of Janelidze's monadic characterization
\cite[Theorems~3.1 and~3.3]{GJIdeallyExact} and van Niekerk's monadic
pullback criterion \cite[Theorem~5.5]{VN}, using the usual subobject
noetherian form of a semi-abelian category \cite[Example~50]{JVN}.

\begin{theorem}[Ideally exact categories]\label{thm:ideally-exact}
Let $\C$ be a well-powered ideally exact category: it is Barr exact
and protomodular, has finite coproducts, and the unique morphism
$p:0\to1$ from its initial object $0$ to its terminal object $1$
is a regular epimorphism. Then $\C$ has a noetherian form inducing
its regular-epimorphism--monomorphism factorization system.
A canonical choice has, for each object $X\in\C$, the fibre
\[
 \Lambda_X=\Sub_{\C/0}\bigl(0\times X\longrightarrow0\bigr),
\]
where $0\times X\to0$ is the first projection and $\C/0$ is
the slice category. For a morphism $f:X\to Y$ of $\C$, direct
and inverse images are the subobject transports in $\C/0$ along
$p^*f=1_0\times f:(0\times X\to0)\to(0\times Y\to0)$.
\end{theorem}
\begin{proof}
By \cite[Theorem~3.1]{GJIdeallyExact}, the category
$\mathcal B=\C/0$ is semi-abelian and the base-change functor
\[
 U=p^*:\C\simeq\C/1\longrightarrow\mathcal B,
 \qquad X\longmapsto(0\times X\to0),
\]
is monadic, with left adjoint $L=p_!$ given by composition with $p$.
The canonical monad $UL$ preserves pullbacks and regular epimorphisms
\cite[Theorem~2.4(b) and Theorem~3.3(b)]{GJIdeallyExact}.
By \cref{lem:maltsev-pushouts}, it therefore preserves pushouts of two
regular epimorphisms. These are precisely the quotients of the usual
subobject noetherian form on $\mathcal B$. Thus
\cref{thm:monadic-pullback} makes the pullback of that form along $U$
noetherian, with the displayed fibres and transports.

The functor $U$ preserves monomorphisms and regular epimorphisms by
base change. Since its pullback form satisfies (N2), \cref{thm:N2}
identifies the induced system with the specified
regular-epimorphism--monomorphism system. Finally, well-poweredness of
$\C$ implies well-poweredness of $\C/0$, so these fibres are sets in
the fixed universe.
\end{proof}

Only preservation of pullbacks and regular epimorphisms is needed from
the monad $UL$; it need not preserve a terminal object. Monadicity by
itself would not supply the pushout-preservation hypothesis of
\cref{thm:monadic-pullback}.

Janelidze's quotient--ideal correspondence
\cite[Theorem~4.3 and Definition~4.4]{GJIdeallyExact} identifies the normal
clusters of this form with all ideals of $X$: every normal subobject of
$U(X)$ is the kernel of $U(q)$ for a regular quotient $q$ of $X$.
For unital rings, the equivalence $\C/0\simeq\mathbf{Rng}$ with rings
without a prescribed identity identifies the fibre with the lattice of
subrings that need not contain the identity. Its normal clusters are the
usual two-sided ideals, while its conormal clusters are the subrings
containing the identity. This illustrates why the resulting form need
not be the ordinary subobject form on $\C$.

\subsection{Copower and constant-object functors}

\subsubsection{Copowers in abelian categories}

We use the standard noetherian subobject form of an abelian category
\cite[Example~50]{JVN}. The proposition concerns the additional
transfer and universal-target properties of its copower functors.
\begin{proposition}\label[proposition]{prop:abelian-targets}
Let $\A$ be a well-powered abelian category with all small coproducts,
equipped with its epi--mono system, and let $G\in\A$ be nonzero.
Equip $\Set$ with its surjection--injection system. Define the
copower functor $H_G:\Set\to\A$ by
$H_G(X)=\coprod_{x\in X}G$ and
$H_G(f)i_x=i_{f(x)}$ for each function $f:X\to Y$, where $i_x$
and $i_{f(x)}$ are the corresponding coproduct injections.
Then $H_G$ is a transfer functor.
Consequently every nonzero such category is a universal target;
in particular, this holds for every nonzero Grothendieck abelian category.
\end{proposition}
\begin{proof}
The ordinary subobject form is noetherian. Choose any nonzero object \(G\),
and put \(H_G(X)=\coprod_{x\in X}G\). The adjunction
\[
 \Hom_\A(H_G(X),A)\cong\Hom_\Set(X,\Hom_\A(G,A))
\]
shows that \(H_G\) preserves colimits; in particular it preserves pushouts
and surjections, which are regular epimorphisms in \(\Set\).
Injections give split coordinate-summand inclusions. Partitioning \(X\)
into \(A\cap B\), \(A\setminus B\), \(B\setminus A\), and the remaining
part gives a finite biproduct decomposition and proves
\[
 H_G(A)\cap H_G(B)=H_G(A\cap B).
\]

If \(f(x)=f(x')\) for \(x\ne x'\), the nonzero map
\(i_x-i_{x'}:G\to H_G(X)\) is killed by \(H_G(f)\). It is nonzero since
the coordinate projection at \(x\) composes with it to \(\id_G\).
If \(f\) misses \(y\), the nonzero coordinate projection
\(p_y:H_G(Y)\to G\) kills \(H_G(f)\). Thus a noninjective or nonsurjective
\(f\) cannot be sent to an isomorphism. Apply \cref{prop:target}.
\end{proof}

The object \(G\) need be neither a generator nor projective. The use of
arbitrary coproducts here does not require a separate exactness assumption on
infinite coproducts: the intersection argument uses a finite partition into
summands.

\subsubsection{Constant objects in Grothendieck toposes}

The existence of a noetherian form on a topos is
\cite[Theorem~172 and Example~173]{JVN}. We combine this with the
constant-object geometric morphism \cite[Chapter~VII]{MM} to obtain
the following universal-target statement.
\begin{proposition}\label[proposition]{prop:topos-targets}
Let $\mathcal T$ be a nondegenerate Grothendieck topos with terminal
object $1$ and its epi--mono factorization system. Equip $\Set$ with
the surjection--injection system. Define the constant-object functor
$\Delta:\Set\to\mathcal T$ by $\Delta(X)=\coprod_{x\in X}1$
and $\Delta(f)i_x=i_{f(x)}$ for a function $f:X\to Y$, where the
$i_x$ are coproduct injections. Then $\Delta$ is a transfer functor,
and $\mathcal T$ is a universal target.
\end{proposition}
\begin{proof}
The constant-object functor is
\[
 \Delta:\Set\to\mathcal T,\qquad \Delta(X)=\coprod_{x\in X}1.
\]
It is the inverse-image part of the global-sections geometric morphism, so
preserves finite limits and all small colimits; see \cite{MM} for these
standard topos constructions. Thus it preserves the classes and homogeneous
diagrams.

For clarity, conservativity does not require enough points. If \(f\)
identifies two elements, their distinct coproduct injections are identified
by \(\Delta(f)\), so \(\Delta(f)\) is not monic. Their distinctness follows
from disjointness of coproducts and \(0\not\cong1\). If \(f\) misses
\(y\), pullback of \(\Delta(f)\) over the \(y\)-summand is \(0\to1\),
which is not an isomorphism. Thus \(\Delta(f)\) cannot be invertible.
\end{proof}

The claim concerns Grothendieck toposes. An arbitrary elementary topos need
not have the required set-indexed coproducts. Nondegeneracy cannot be omitted
from this particular universal-target assertion.

\subsection{Duality and the scope of universal targets}\label{sec:target-scope}

\subsubsection{Opposites and Stone spaces}

Duality of noetherian forms is recalled in \cite[Section~2.5]{VN};
the following is its consequence for universal targets.
\begin{proposition}\label[proposition]{prop:opposite}
Let $\D$, with proper factorization system $(\E_\D,\M_\D)$,
be a universal target for existence of a noetherian form.
Then $\D^{\op}$ is a universal target for the system
$(\M_\D^{\op},\E_\D^{\op})$. If $\mathcal B$ is equivalent
to $\D$, it too is a universal target when equipped with the
factorization system transported along the equivalence.
\end{proposition}
\begin{proof}
Equivalence transports all the required structure. For opposites, a
noetherian form on \(\C\) dualizes to one on \(\C^{\op}\). A suitable
functor \(\C^{\op}\to\D\) then gives a suitable functor
\(\C\to\D^{\op}\), with the two types of diagrams interchanged. Conversely,
\(\D^{\op}\) has a compatible noetherian form by dualizing one on \(\D\),
so pullback gives the reverse implication in the characterization.
\end{proof}

By classical Stone duality \cite{Stone},
\(\Bool^{\op}\simeq\Stone\), where \(\Stone\) is the category of compact
zero-dimensional Hausdorff spaces. Under this duality the transported system
is
\[
 (\text{surjective continuous maps},\text{topological embeddings}).
\]
Indeed, surjective Boolean homomorphisms dualize to embeddings, and injective
Boolean homomorphisms dualize to surjections. Therefore \(\Stone\) is a
universal target. Its pushouts are taken in \(\Stone\), not in \(\Top\).

A concrete representation obtained from a noetherian form is
\begin{equation}\label{eq:Stone-representation}
 X\longmapsto\{0,1\}^{\Lambda_X},\qquad
 f\longmapsto\bigl(u\longmapsto u\circ f^*\bigr).
\end{equation}
This is the cluster functor for the dual form, followed by the free Boolean
algebra functor and Stone duality. The Stone space of the free Boolean algebra
on a set \(S\) is the Cantor cube \(\{0,1\}^S\). The points in
\eqref{eq:Stone-representation} are all set-functions from \(\Lambda_X\) to
\(\{0,1\}\), not necessarily lattice homomorphisms.

\subsubsection{Having a form is not enough to be a universal target}

\begin{counterexample}\label[counterexample]{ex:finite-dimensional}
The category \(\FinVect_k\) has its usual noetherian subobject form, but it
is not a universal target for all sets. Suppose that a conservative functor
\(T:\Set\to\FinVect_k\) preserved injections. The chain of finite subsets
\[
 \{0\}\subsetneq\{0,1\}\subsetneq\{0,1,2\}\subsetneq\cdots\subseteq\mathbb N
\]
would give an infinite chain of subspaces of the fixed finite-dimensional
space \(T(\mathbb N)\). Every successive inclusion would be strict, since
an isomorphism after \(T\) would contradict conservativity. This is impossible.
\end{counterexample}

Thus the vector-space target above consists of all-dimensional vector spaces.
The same argument excludes any proposed target in which every object has the
corresponding ascending-chain bound on the distinguished subobjects.

\section{Noetherian forms for topological categories}\label{sec:top}

We construct transfer functors to sets for topological categories with
injective initial tests. The construction has two stages: encoding the
tests by equivalence relations, and passing from these relations to
reduced words. Concrete examples and the scope of the test hypothesis
follow the general theorem. The specialization to topological spaces
is developed in \cref{sec:top-spaces}.

\subsection{Initial tests and their Boolean encoding}\label{sec:top-tests}

\subsubsection{The natural factorization system}
Let $U:\C\to\Set$ be a faithful functor. A source of morphisms
$(a_j:X\to X_j)_{j\in J}$ is \emph{$U$-initial} if a function
$h:UZ\to UX$ underlies a morphism $Z\to X$ precisely when every
$(Ua_j)h$ underlies a morphism $Z\to X_j$. The functor $U$ is
\emph{topological} if every possibly large or empty source of functions
$(s_j:S\to UX_j)_{j\in J}$ has an initial lift to an object with
underlying set $S$. We use the usual amnestic normalization, in which
this lift is unique; working up to concrete isomorphism gives the same
results. Thus the index $J$ may be a class; the test family introduced
below, by contrast, is required to be set-indexed.

Write $\E$ for the morphisms with surjective underlying function and
$\M$ for the \emph{initial embeddings}: the morphisms whose underlying
functions are injective and whose one-map sources are $U$-initial.

The following standard lifting and factorization facts are recalled from
\cite[Propositions~21.14--21.15 and Corollary~21.17]{AHS}.

\begin{proposition}\label{tc:prop:fs}
Let $U:\C\to\Set$ be a topological functor. Let $\E$ consist of
the morphisms whose underlying functions are surjective, and let $\M$
consist of the $U$-initial morphisms whose underlying functions are
injective. Then $(\E,\M)$ is a proper orthogonal factorization system
on $\C$. The category $\C$ has small limits and colimits, and $U$
preserves them.
\end{proposition}
\begin{proof}
Factor $Uf$ through its image in $UY$ and equip this image with the
initial structure induced by its inclusion into $Y$. Initiality makes
the resulting surjection from $X$ a morphism, giving $f=me$.
In a square with $e\in\E$ on the left and $m\in\M$ on the right,
the unique set-theoretic diagonal is a morphism because its composite
with $m$ is a morphism. Thus $e\perp m$. These factorizations and
orthogonality give the factorization system. Faithfulness of $U$ makes
the indicated surjections epimorphisms and injections monomorphisms.

The assertions about limits and colimits are the standard lifting
theorem for topological functors: they uniquely lift limits and colimits
of the base category and preserve the resulting diagrams
\cite[Proposition~21.15 and Corollary~21.17]{AHS}.
\end{proof}

\subsubsection{Initially determining injective test objects}
\begin{definition}\label{tc:def:tests}
A \emph{family of injective initial tests} for $U$ is a set-indexed
family $(Q_i)_{i\in I}$ of objects such that:
\begin{enumerate}[label=\textup{(\roman*)}]
\item each $UQ_i$ is nonempty;
\item for every $X$, the source of all morphisms $a:X\to Q_i$,
with $i\in I$, is $U$-initial;
\item every $Q_i$ is $\M$-injective: for $m:A\to B$ in $\M$,
every $a:A\to Q_i$ has an extension $b:B\to Q_i$ with $bm=a$.
\end{enumerate}
\end{definition}

We also arrange that some $UQ_i$ has at least two elements. This imposes
no additional restriction: if necessary, adjoin the indiscrete object
on $\{0,1\}$, obtained by lifting the empty source. Every function
into it is a morphism, and a function on a subset extends by zero, so
it is $\M$-injective. Adding a test preserves initial determination.
No object needs to be adjoined in the single-test examples below.

The smallness of $I$ is part of the hypothesis. In fact it implies
fibre-smallness up to concrete isomorphism: a structure on a fixed set
$S$ is determined by which functions $S\to UQ_i$ are morphisms. We
make no claim that a small family of such tests exists for every
topological category; \cref{sec:top-scope} examines this issue.

For the encoding and word constructions fix such a family, and put
\begin{equation}\label{tc:eq:testsets}
P(S)=\coprod_{i\in I}\Set(S,UQ_i),\qquad
\Omega(X)=\coprod_{i\in I}\C(X,Q_i)\ \subseteq\ P(UX).
\end{equation}
Here the inclusion uses faithfulness of $U$ and retains the label $i$.
Both constructions are contravariant, with maps given by precomposition.

\subsubsection{Boolean powers and labelled test sets}
Put $\two=\{0,1\}$. For a set $S$, let $\BP(S)=\two^S$, and for
$u:S\to T$ let $\BP(u):\two^T\to\two^S$ be precomposition with $u$.
All functions into $\two$ in this construction are arbitrary functions.

Representability and injectivity of nonempty sets give the familiar
first two assertions below; see \cite[Corollary~13.9 and
Example~9.3(1)]{AHS}. We include the intersection calculation as well.

\begin{lemma}\label{top:lem:boolean}
Let $Q$ be a nonempty set and define the contravariant power functor
$E_Q:\Set^{\op}\to\Set$ by $E_Q(S)=Q^S$ and
$E_Q(u)(a)=a\circ u$ for functions $u:S\to T$ and $a:T\to Q$.
Then:
\begin{enumerate}
\item $E_Q$ takes colimits of sets to limits of sets;
\item it takes surjections to injections and injections to surjections;
if $|Q|\geq2$, it also reflects bijectivity;
\item it takes pullbacks of two injections, including those with empty
intersection, to pushouts.
\end{enumerate}
In particular, the Boolean-power functor $\BP=E_{\{0,1\}}$ has all
these properties and reflects bijectivity.
\end{lemma}
\begin{proof}
The first assertion is the universal property of a colimit, applied to
maps into $Q$. Fix $q_0\in Q$. Precomposition with a surjection is
injective, and a function from a subset to $Q$ extends by assigning
$q_0$ outside the subset. If $|Q|\geq2$, choose $q_1\ne q_0$.
A missed point gives two functions with the same precomposition,
and two identified points give a function taking the values $q_0$
and $q_1$ that cannot be obtained by precomposition. This proves
the second assertion, including reflection of bijectivity when stated.

For the third, represent the injections by subsets $A,B\subseteq C$.
The required assertion is
\begin{equation}\label{top:eq:restriction-pushout}
Q^A\amalg_{Q^C}Q^B\ \cong\ Q^{A\cap B}.
\end{equation}
Functions on $A$ and $B$ have the same restriction to $A\cap B$ if and
only if they glue on $A\cup B$, and then extend to $C$ using $q_0$.
Functions on one side with the same restriction are both identified
with any extension on the other side. All functions on the intersection
extend, and no identification changes their restriction. This proves
the assertion, including empty intersections. The Boolean-power
assertions are the specialization $Q=\{0,1\}$.
\end{proof}

\begin{lemma}\label{tc:lem:P}
Let $U:\C\to\Set$ be a functor, and let $(Q_i)_{i\in I}$ be a
set-indexed family of objects with nonempty underlying sets, with
$|UQ_i|\geq2$ for at least one $i$. Define
$P:\Set^{\op}\to\Set$ by
\[
P(S)=\coprod_{i\in I}\Set(S,UQ_i),\qquad
P(u)(i,a)=(i,a\circ u),
\]
for $u:S\to T$ and $a:T\to UQ_i$.
Then $P$ takes surjections to injections, injections to surjections,
pushouts to pullbacks, and pullbacks of two injections to pushouts.
It reflects bijectivity. Consequently, for $\two=\{0,1\}$, the
covariant functor $D:\Set\to\Set$ defined by
\[
D(S)=\two^{P(S)},\qquad D(u)(\phi)=\phi\circ P(u)
\]
preserves surjections and injections, reflects bijectivity, and
preserves pullbacks of two injections and pushouts of two surjections.
\end{lemma}
\begin{proof}
Apply \cref{top:lem:boolean} to each nonempty set $UQ_i$.
It supplies the class properties componentwise. Since $P(u)$
preserves labels, its bijectivity makes every component bijective;
the component with $|UQ_i|\geq2$ therefore reflects bijectivity of $u$.

Each component takes pushouts to pullbacks. In a pullback of labelled
sets, a compatible pair necessarily has the same label, so their
coproduct $P$ also takes pushouts to pullbacks. Each component takes
pullbacks of two injections to pushouts, and the coproduct of these
pushout squares is again a pushout square. Finally, composing with
the Boolean-power functor $\BP$ in \cref{top:lem:boolean} gives
all the assertions about $D$.
\end{proof}

The argument uses connected diagrams. A coproduct of representables
need not take arbitrary colimits to limits. In particular, no claim
about all colimits is being made for $P$ or $\Omega$.

\subsubsection{Equivalence relations as quotient maps}\label{top:sec:equiv}
An object of $\Equiv$ is a pair $(A,R)$, where $R$ is an equivalence
relation on $A$. Morphisms preserve the relation. A \emph{full embedding}
is an injective morphism that also reflects the relation. A
\emph{surjection} has a surjective underlying function.

The usual quotient--equivalence-relation correspondence
\cite[Example~7.86(1)]{AHS} has the following functorial form.

\begin{lemma}\label{top:lem:equiv-colimits}
Let $\Equiv$ be the category of pairs $(A,R)$ consisting of a set
and an equivalence relation, with functions that preserve the relations
as morphisms. The functor sending $(A,R)$ to its quotient map
$A\twoheadrightarrow A/R$, and a morphism to the induced commuting
square of quotient maps, is an equivalence from $\Equiv$ to the full
subcategory of the arrow category of $\Set$ consisting of surjections.
Colimits in that subcategory, including the empty colimit, are computed
componentwise in $\Set$.
\end{lemma}
\begin{proof}
A relation-preserving function induces a unique map on quotient sets.
Conversely, the upper map in a commuting square of quotient maps
preserves the kernel equivalence relations. Every surjection is
isomorphic to the quotient map of its kernel relation.

For a diagram of surjections $A_j\twoheadrightarrow B_j$, the map
$\operatorname{colim}A_j\to\operatorname{colim}B_j$ is surjective:
a representative in $B_j$ lifts to $A_j$. The empty diagram gives
$\varnothing\to\varnothing$. Thus componentwise colimits stay in the
full subcategory and have its universal property.
\end{proof}

Pullbacks in $\Equiv$ have the usual underlying set pullbacks, with the
coordinatewise equivalence relation. In particular, a pullback of two
full embeddings is an intersection with the inherited relation.

\subsubsection{Encoding the test structure}\label{top:sec:K}
For $X\in\C$, define
\begin{equation}\label{tc:eq:Kdef}
A_X=\two^{P(UX)},\qquad
\phi\mathrel{R_X}\psi\ \Longleftrightarrow\
\phi|_{\Omega(X)}=\psi|_{\Omega(X)},\qquad KX=(A_X,R_X).
\end{equation}
For $f:X\to Y$, put
\begin{equation}\label{tc:eq:Kmap}
Kf(\phi)=\phi\circ P(Uf).
\end{equation}
Precomposition takes tests to tests, so $Kf$ preserves equivalence.
The identity and composition laws follow at once. Under
Lemma~\ref{top:lem:equiv-colimits}, $K$ is represented by the natural
restriction surjections
\begin{equation}\label{tc:eq:restriction-arrow}
\two^{P(UX)}\twoheadrightarrow\two^{\Omega(X)}.
\end{equation}
They are surjective because functions on $\Omega(X)$ extend by zero.

\begin{proposition}\label{top:prop:Kconservative}
Let $U:\C\to\Set$ be topological, and let $(Q_i)_{i\in I}$ be a
set-indexed family of objects with nonempty underlying sets, at least
one of which has at least two elements. Suppose that all maps $X\to Q_i$
jointly form a $U$-initial source for every $X\in\C$, and that
every map $A\to Q_i$ extends along every $U$-initial injection
$A\to B$. Put $\two=\{0,1\}$ and
\[
P(S)=\coprod_{i\in I}\Set(S,UQ_i),\qquad
\Omega(X)=\coprod_{i\in I}\C(X,Q_i)\subseteq P(UX),
\]
with $P(u)$ given by precomposition. Let $\Equiv$ denote sets with
equivalence relations and relation-preserving functions. Define the
encoding functor $K:\C\to\Equiv$ on objects $X$ and morphisms
$f:X\to Y$ by
\[
KX=(\two^{P(UX)},R_X),\qquad
\phi R_X\psi\ \Longleftrightarrow\
\phi|_{\Omega(X)}=\psi|_{\Omega(X)},\qquad
Kf(\phi)=\phi\circ P(Uf).
\]

Then $K$ is conservative: if $Kf$ is an isomorphism, then $f$ is
an isomorphism.
\end{proposition}
\begin{proof}
If $Kf$ is invertible, its underlying map $D(Uf)$ is bijective, so
$Uf$ is bijective by Lemma~\ref{tc:lem:P}. Its map on quotient sets is
precomposition with $\Omega(f)$, so $\Omega(f)$ is also bijective by
Lemma~\ref{top:lem:boolean}. For each $a:X\to Q_i$, this gives
$b:Y\to Q_i$ such that $a=bf$. Therefore $a(Uf)^{-1}=b$ is a
morphism for every test $a$. Initial determination of $X$ implies that
$(Uf)^{-1}$ is a morphism. Thus $f$ is an isomorphism.
\end{proof}

\begin{proposition}\label{top:prop:Kclasses}
Let $U:\C\to\Set$ be topological, and let $(Q_i)_{i\in I}$ be a
set-indexed family of objects with nonempty underlying sets, at least
one of which has at least two elements. Suppose that all maps $X\to Q_i$
jointly form a $U$-initial source for every $X\in\C$, and that
every map $A\to Q_i$ extends along every $U$-initial injection
$A\to B$. Put $\two=\{0,1\}$ and
\[
P(S)=\coprod_{i\in I}\Set(S,UQ_i),\qquad
\Omega(X)=\coprod_{i\in I}\C(X,Q_i)\subseteq P(UX),
\]
with $P(u)$ given by precomposition. Let $\Equiv$ denote sets with
equivalence relations and relation-preserving functions. Define the
encoding functor $K:\C\to\Equiv$ on objects $X$ and morphisms
$f:X\to Y$ by
\[
KX=(\two^{P(UX)},R_X),\qquad
\phi R_X\psi\ \Longleftrightarrow\
\phi|_{\Omega(X)}=\psi|_{\Omega(X)},\qquad
Kf(\phi)=\phi\circ P(Uf).
\]

Let $\E$ consist of underlying surjections in $\C$, and $\M$ of
$U$-initial injections. Then $K$ sends $\E$-morphisms to underlying
surjections in $\Equiv$ and $\M$-morphisms to full embeddings,
meaning injective functions that preserve and reflect equivalence.
\end{proposition}
\begin{proof}
For $e\in\E$, the map $P(Ue)$ is injective, so $Ke$ is surjective.
For $m:A\to X$ in $\M$, both
\[
P(Um):P(UX)\twoheadrightarrow P(UA),\qquad
\Omega(m):\Omega(X)\twoheadrightarrow\Omega(A)
\]
are surjective: the first by nonemptiness of the tests, the second by
their $\M$-injectivity. The first makes $Km$ injective. The second gives
\[
Km(\phi)\mathrel{R_X}Km(\psi)
\quad\Longleftrightarrow\quad\phi\mathrel{R_A}\psi,
\]
so $Km$ is a full embedding.
\end{proof}

\begin{proposition}\label{top:prop:Kpullbacks}
Let $U:\C\to\Set$ be topological, and let $(Q_i)_{i\in I}$ be a
set-indexed family of objects with nonempty underlying sets, at least
one of which has at least two elements. Suppose that all maps $X\to Q_i$
jointly form a $U$-initial source for every $X\in\C$, and that
every map $A\to Q_i$ extends along every $U$-initial injection
$A\to B$. Put $\two=\{0,1\}$ and
\[
P(S)=\coprod_{i\in I}\Set(S,UQ_i),\qquad
\Omega(X)=\coprod_{i\in I}\C(X,Q_i)\subseteq P(UX),
\]
with $P(u)$ given by precomposition. Let $\Equiv$ denote sets with
equivalence relations and relation-preserving functions. Define the
encoding functor $K:\C\to\Equiv$ on objects $X$ and morphisms
$f:X\to Y$ by
\[
KX=(\two^{P(UX)},R_X),\qquad
\phi R_X\psi\ \Longleftrightarrow\
\phi|_{\Omega(X)}=\psi|_{\Omega(X)},\qquad
Kf(\phi)=\phi\circ P(Uf).
\]

Let $\M$ consist of the $U$-initial injections in $\C$. Then $K$
preserves pullbacks of two $\M$-morphisms with common codomain.
\end{proposition}
\begin{proof}
In such a pullback all four maps belong to $\M$. The underlying-set
component $DU$ preserves the square by Proposition~\ref{tc:prop:fs}
and Lemma~\ref{tc:lem:P}. By Proposition~\ref{top:prop:Kclasses}, all
four image maps are full embeddings. Their relations are inherited
from the common codomain, so the underlying pullback carries precisely
the relation required for a pullback in $\Equiv$.
\end{proof}

\begin{proposition}\label{top:prop:Kpushouts}
Let $U:\C\to\Set$ be topological, and let $(Q_i)_{i\in I}$ be a
set-indexed family of objects with nonempty underlying sets, at least
one of which has at least two elements. Suppose that all maps $X\to Q_i$
jointly form a $U$-initial source for every $X\in\C$, and that
every map $A\to Q_i$ extends along every $U$-initial injection
$A\to B$. Put $\two=\{0,1\}$ and
\[
P(S)=\coprod_{i\in I}\Set(S,UQ_i),\qquad
\Omega(X)=\coprod_{i\in I}\C(X,Q_i)\subseteq P(UX),
\]
with $P(u)$ given by precomposition. Let $\Equiv$ denote sets with
equivalence relations and relation-preserving functions. Define the
encoding functor $K:\C\to\Equiv$ on objects $X$ and morphisms
$f:X\to Y$ by
\[
KX=(\two^{P(UX)},R_X),\qquad
\phi R_X\psi\ \Longleftrightarrow\
\phi|_{\Omega(X)}=\psi|_{\Omega(X)},\qquad
Kf(\phi)=\phi\circ P(Uf).
\]

Let $\E$ consist of the morphisms of $\C$ whose underlying functions
are surjective. Then $K$ preserves pushouts of two $\E$-morphisms
with common domain.
\end{proposition}
\begin{proof}
Consider $B\xleftarrow{e}A\xrightarrow{d}C$ with $e,d\in\E$, and
let $Z=B\amalg_A C$. The underlying component $DU$ preserves this
pushout by Proposition~\ref{tc:prop:fs} and Lemma~\ref{tc:lem:P}.
Representability in each summand and compatibility of labels give
\[
\Omega(Z)\ \cong\ \Omega(B)\times_{\Omega(A)}\Omega(C).
\]
Both maps into $\Omega(A)$ are injective, since $e$ and $d$ have
surjective underlying functions. Lemma~\ref{top:lem:boolean}(iii)
therefore gives
\[
\two^{\Omega(Z)}\ \cong\
\two^{\Omega(B)}\amalg_{\two^{\Omega(A)}}\two^{\Omega(C)}.
\]
Both components of \eqref{tc:eq:restriction-arrow} thus preserve the
pushout. Lemma~\ref{top:lem:equiv-colimits} proves the assertion.
\end{proof}

\subsection{Reduced words and a conservative functor to sets}\label{top:sec:words}
\subsubsection{Reduced words and functoriality}
For the equality relation, the construction below is the classical
word model of the free left regular band
\cite[Definition~2.1]{BCR2023}. We adapt its reduction rule to an
arbitrary equivalence relation and verify the properties needed here.

Given $(A,R)\in\Equiv$, call a finite word $(a_1,\ldots,a_n)$
\emph{$R$-reduced} if its letters belong to pairwise distinct $R$-classes.
The empty word $\eps$ is included. Define
\begin{equation}\label{top:eq:Wdef}
W(A,R)=\{(a_1,\ldots,a_n):n\geq0,\ a_i\not\mathrel R a_j\text{ for }i\neq j\}.
\end{equation}
For an arbitrary finite word $w$ over $A$, let $\red_R(w)$ be obtained
by retaining only the first occurrence of each $R$-class.
For $h:(A,R)\to(B,S)$, set
\begin{equation}\label{top:eq:Wmap}
Wh(a_1,\ldots,a_n)=\red_S(h(a_1),\ldots,h(a_n)).
\end{equation}

\begin{lemma}\label{top:lem:Wfunctor}
Let $\Equiv$ be the category of sets with equivalence relations and
relation-preserving functions. For $(A,R)\in\Equiv$, let $W(A,R)$
be the set of finite words in $A$ with letters in pairwise distinct
$R$-classes, including the empty word. For a relation-preserving
function $h:(A,R)\to(B,S)$, let $Wh$ apply $h$ to each letter and
then retain only the first occurrence of each $S$-class.

These assignments define a functor $W:\Equiv\to\Set$.
\end{lemma}
\begin{proof}
Identity maps do not delete letters from a reduced word. If $h$ preserves
equivalence, a letter deleted by $\red_R$ has an earlier $R$-equivalent
letter whose image is $S$-equivalent. Consequently,
\[
\red_S\bigl(h(\red_R(w))\bigr)=\red_S(h(w)).
\]
Applying this observation to successive maps proves $W(kh)=Wk\circ Wh$.
\end{proof}

Single-letter words retain the original elements of $A$, even when those
elements are $R$-equivalent. The relation affects which longer words are
allowed and which letters disappear under a morphism. This distinction
is essential for conservativity.

\subsubsection{Conservativity, classes, and pullbacks}
\begin{proposition}\label{top:prop:Wclasses}
Let $\Equiv$ be the category of sets with equivalence relations and
relation-preserving functions. For $(A,R)\in\Equiv$, let $W(A,R)$
be the set of finite words in $A$ with letters in pairwise distinct
$R$-classes, including the empty word. For a relation-preserving
function $h:(A,R)\to(B,S)$, let $Wh$ apply $h$ to each letter and
then retain only the first occurrence of each $S$-class.

For every such morphism $h$,
\begin{align*}
Wh\text{ is surjective}&\quad\Longleftrightarrow\quad h\text{ is surjective},\\
Wh\text{ is injective}&\quad\Longleftrightarrow\quad h\text{ is a full embedding},
\end{align*}
where a full embedding is an injective function that preserves and
reflects equivalence. In particular, $W$ is conservative.
\end{proposition}
\begin{proof}
If $h$ is surjective, lift the letters of a target reduced word
individually. Their lifts cannot be $R$-equivalent, since their images
are pairwise $S$-inequivalent. This lifts the target word without
deletion. Conversely, if $Wh$ is surjective, lift a one-letter word
$(b)$. Its preimage is nonempty, and the image of its first letter is
$b$. Thus $h$ is surjective.

If $h$ is a full embedding, no letters are deleted by $Wh$, and its
injectivity follows from that of $h$. Conversely, injectivity of $Wh$
on one-letter words forces injectivity of $h$. If $a\not\mathrel R a'$
but $h(a)\mathrel S h(a')$, the distinct words $(a,a')$ and $(a)$ have
the same image. Therefore $h$ reflects equivalence and is a full
embedding. A bijective full embedding is an isomorphism in $\Equiv$,
which proves conservativity.
\end{proof}

\begin{proposition}\label{top:prop:Wpullbacks}
Let $\Equiv$ be the category of sets with equivalence relations and
relation-preserving functions. For $(A,R)\in\Equiv$, let $W(A,R)$
be the set of finite words in $A$ with letters in pairwise distinct
$R$-classes, including the empty word. For a relation-preserving
function $h:(A,R)\to(B,S)$, let $Wh$ apply $h$ to each letter and
then retain only the first occurrence of each $S$-class.

Then $W$ preserves pullbacks of two full embeddings, where a full
embedding is an injective function that preserves and reflects
equivalence.
\end{proposition}
\begin{proof}
Represent the full embeddings by subsets $A,B\subseteq C$, with the
relations induced from $(C,T)$. The image of $W(A,T|_A)$ in $W(C,T)$
is exactly the set of reduced words whose letters lie in $A$. Its
intersection with the analogous image for $B$ is exactly the set of
reduced words whose letters lie in $A\cap B$. This is the required
set-theoretic pullback.
\end{proof}

\subsubsection{The left regular band adjunction}
A \emph{left regular band monoid} is a monoid satisfying
\begin{equation}\label{top:eq:lrb-identities}
x^2=x,\qquad xyx=xy.
\end{equation}
Let $\LRB$ denote their category. These identities and the classical
free model by words without repeated letters are standard; see
\cite[Section~2.1]{BCR2023}. The construction below also allows a prescribed
equivalence relation among the letters.

Define a multiplication on $W(A,R)$ by concatenation followed by reduction:
\begin{equation}\label{top:eq:wordproduct}
u\cdot v=\red_R(uv).
\end{equation}
The reduction identities
\[
\red_R(\red_R(u)v)=\red_R(uv)
=\red_R(u\,\red_R(v))
\]
imply associativity. The empty word is an identity. Every letter in a
second copy of $u$ is deleted, as is every letter in the final copy of
$u$ in $uvu$. Hence \eqref{top:eq:lrb-identities} holds. Formula
\eqref{top:eq:Wmap} respects this multiplication, defining a functor
\[
\mathsf L:\Equiv\to\LRB
\]
whose underlying-set functor is $W$.

For $M\in\LRB$, define a relation on its underlying set by
\begin{equation}\label{top:eq:band-relation}
b\approx_M c\quad\Longleftrightarrow\quad bc=b\text{ and }cb=c.
\end{equation}
This is an equivalence relation. Reflexivity follows from idempotence,
and symmetry is built into the definition. For transitivity, if
$b\approx_M c$ and $c\approx_M d$, then
\[
bd=(bc)d=b(cd)=bc=b,
\qquad
db=(dc)b=d(cb)=dc=d.
\]
Monoid homomorphisms preserve this relation. Thus there is a functor
$\mathcal R:\LRB\to\Equiv$, with $\mathcal R(M)=(UM,\approx_M)$.

\begin{proposition}\label{top:prop:wordadjunction}
Let $\Equiv$ denote sets with equivalence relations and
relation-preserving functions, and let $\LRB$ denote monoids satisfying
$x^2=x$ and $xyx=xy$, with monoid homomorphisms. Define
$\mathsf L:\Equiv\to\LRB$ by taking $\mathsf L(A,R)$ to consist
of finite words with letters in pairwise distinct $R$-classes,
including the empty word. Its product concatenates words and retains
only the first occurrence of each $R$-class, and its identity is the
empty word. On morphisms, $\mathsf L$ applies the function to letters
and performs the same reduction in the target relation. Define
$\mathcal R:\LRB\to\Equiv$ by
\[
\mathcal R(M)=(UM,\approx_M),\qquad
b\approx_M c\ \Longleftrightarrow\ bc=b\text{ and }cb=c,
\]
where $UM$ is the underlying set of $M$, and let $\mathcal R$ act on
homomorphisms by their underlying functions. Then
$\mathsf L\dashv\mathcal R$.
\end{proposition}
\begin{proof}
An equivalence-preserving function $u:(A,R)\to\mathcal R(M)$ extends by
\begin{equation}\label{top:eq:wordextension}
\overline u(a_1,\ldots,a_n)=u(a_1)\cdots u(a_n),\qquad
\overline u(\eps)=1_M.
\end{equation}
To verify that reductions do not affect this product, suppose $b\approx_M c$.
For any $z\in M$,
\begin{equation}\label{top:eq:deletion}
bzc=(bc)zc=b(czc)=b(cz)=(bc)z=bz.
\end{equation}
Thus a later letter equivalent to an earlier letter can be deleted even
when other letters intervene. Repeated use of \eqref{top:eq:deletion} proves
that \eqref{top:eq:wordextension} is a monoid homomorphism.

Conversely, if $v:\mathsf L(A,R)\to M$ is a monoid homomorphism, its restriction
to one-letter words preserves equivalence: when $a\mathrel R a'$, the
identities $(a)\cdot(a')=(a)$ and $(a')\cdot(a)=(a')$ imply that the
images satisfy \eqref{top:eq:band-relation}. Since one-letter words generate
$\mathsf L(A,R)$, the extension is unique. These correspondences are natural
and give the required adjunction.
\end{proof}

\subsubsection{Why the word functor preserves the required pushouts}
\begin{proposition}\label{top:prop:Wpushouts}
Let $\Equiv$ be the category of sets with equivalence relations and
relation-preserving functions. For $(A,R)\in\Equiv$, let $W(A,R)$
be the set of finite words in $A$ with letters in pairwise distinct
$R$-classes, including the empty word. For a relation-preserving
function $h:(A,R)\to(B,S)$, let $Wh$ apply $h$ to each letter and
then retain only the first occurrence of each $S$-class.

Then $W:\Equiv\to\Set$ preserves pushouts of two morphisms with
surjective underlying functions.
\end{proposition}
\begin{proof}
Write $V:\LRB\to\Set$ for the underlying-set functor, so that $W=V\mathsf L$.
The functor $\mathsf L$ preserves pushouts because it is a left adjoint.
By Proposition~\ref{top:prop:Wclasses}, it sends surjections in $\Equiv$ to
surjective monoid homomorphisms. Proposition~\ref{top:lem:variety-pushout}, applied
to the finitary variety $\LRB$, therefore shows that $V$ preserves the
resulting pushout. Hence $W$ preserves the original pushout.
\end{proof}

\subsubsection{The representation theorem for topological categories}\label{top:sec:GH}
We can now combine the encoding with the word construction.

\begin{theorem}\label{tc:thm:form}
Let $U:\C\to\Set$ be topological. Suppose there is a set-indexed
family $(Q_i)_{i\in I}$ of objects with nonempty underlying sets such
that all morphisms $X\to Q_i$ jointly form a $U$-initial source for
every $X\in\C$, and every map $A\to Q_i$ extends along every
$U$-initial injection $A\to B$. Let $\E$ consist of morphisms with
surjective underlying functions, and $\M$ of $U$-initial injections.
Then there is a conservative functor $G:\C\to\Set$ such that $Gf$
is surjective, respectively injective, if and only if $f\in\E$,
respectively $f\in\M$, and $G$ preserves pullbacks of two
$\M$-morphisms and pushouts of two $\E$-morphisms. Let $\Ab$ denote
the category of abelian groups, and let $\FZ:\Set\to\Ab$ send a set to its free
abelian group and a function to its linear extension, and put
$H=\FZ G$. The form with subgroup fibres and transport
\begin{equation}\label{tc:eq:form}
\N_X=\Sub_{\Ab}(HX),\qquad
f_!(S)=Hf(S),\qquad f^*(T)=(Hf)^{-1}(T)
\end{equation}
for $X\in\C$, $f:X\to Y$, $S\leq HX$, and $T\leq HY$ is
noetherian, has complete algebraic modular fibres, and induces $(\E,\M)$.
\end{theorem}
\begin{proof}
Adjoin the indiscrete two-point test if necessary, and take $G=WK$.
Conservativity follows from Propositions~\ref{top:prop:Kconservative}
and \ref{top:prop:Wclasses}. Class preservation follows from
Propositions~\ref{top:prop:Kclasses} and \ref{top:prop:Wclasses};
reflection follows from Lemma~\ref{lem:classreflection}. The pullback
assertion follows from Propositions~\ref{top:prop:Kpullbacks} and
\ref{top:prop:Wpullbacks}, and the pushout assertion from
Propositions~\ref{top:prop:Kpushouts} and \ref{top:prop:Wpushouts}.

By Theorem~\ref{thm:free-variety}, applied to abelian groups, $H$ has the same conservativity,
class, and preservation properties. The subgroup form on $\Ab$ is
noetherian, and the required diagrams exist by
Proposition~\ref{tc:prop:fs}. The transfer theorem
(\cref{thm:transfer}) gives \eqref{tc:eq:form} and its stated
factorization system. Its fibres are subgroup lattices of abelian groups,
so they are complete algebraic modular lattices, as also verified in
the proof of \cref{thm:modular-replacement}.
\end{proof}

Explicitly, $GX$ consists of the finite words in functions
$P(UX)\to\two$ whose restrictions to $\Omega(X)$ are pairwise
distinct, including the empty word. The map $Gf$ applies
precomposition with $P(Uf)$ to each letter and then deletes later
letters with a previously occurring test restriction. Thus the theorem
supplies a set-sized construction at every object, rather than only
an abstract existence argument.

\subsection{Quotients, induced subobjects, and their concrete interpretation}
\label{sec:concrete-dictionary}\label{sec:top-examples}

The quotient and embedding orders of a noetherian form are determined by
its proper factorization system; see \cite[Remark~42 and Lemma~44]{JVN}
and \cref{lem:basic,lem:intervals}. We now express these orders in the
concrete language of topological categories, using initial structures
and transport along bijections as in \cite{AHS}. This gives a common
derivation of the isomorphism and correspondence statements in the
examples below.

\begin{proposition}[The concrete quotient dictionary]
\label{tc:prop:quotient-dictionary}
Let $U:\C\to\Set$ be a topological functor, and let $\F$ be a
noetherian form on $\C$ whose quotients are the morphisms with
surjective underlying functions and whose embeddings are the
$U$-initial injections. For $X\in\C$, write $\Nrm_\F(X)$ and
$\Cnm_\F(X)$ for its normal and conormal cluster orders.
A \emph{weak quotient datum} on $X$ is a pair $\alpha=(R,Q)$,
where $R$ is an equivalence relation on $UX$, $Q$ is a
$\C$-structure on $UX/R$, and the class map
$q_\alpha:X\to Q$ is a morphism. Put $X/\alpha=Q$, identify
structures up to isomorphism over their underlying sets, and order
these data by
\[
 \alpha\le\beta
 \quad\Longleftrightarrow\quad
 q_\beta=rq_\alpha\text{ for a morphism }r:X/\alpha\to X/\beta.
\]
Write $\mathcal Q_U(X)$ for this order, and let $X|_A$ denote the
initial structure on a subset $A\subseteq UX$. Then:
\begin{enumerate}[label=\textup{(\roman*)}]
\item There are order isomorphisms
\[
 \mathcal Q_U(X)\cong\Nrm_\F(X),\quad
 \alpha\longmapsto\Ker_\F q_\alpha,
 \qquad
 \Pow(UX)\cong\Cnm_\F(X),\quad
 A\longmapsto\Img_\F(X|_A\hookrightarrow X).
\]
\item For a morphism $f:X\to Y$, let $R_f$ be equality of its
underlying values. Give $UX/R_f$ the structure transported from
$Y|_{f(UX)}$ along $[x]\mapsto f(x)$, and call the resulting weak
quotient datum $\rho_f$. The canonical map is an isomorphism
\[
 X/\rho_f\xrightarrow{\ \cong\ }Y|_{f(UX)}.
\]
This is the quotient--embedding factorization of $f$.
\item For a weak quotient datum $\alpha$ on $X$ and a subset
$A\subseteq UX$, let $\alpha|_A=\rho_{q_\alpha i_A}$, where
$i_A:X|_A\hookrightarrow X$. Its equivalence relation is the
restriction of $R_\alpha$ to $A$. There is a canonical isomorphism
\[
 (X|_A)/(\alpha|_A)\cong(X/\alpha)|_{q_\alpha(A)}.
\]
\item If $\alpha\le\beta$, let $r:X/\alpha\to X/\beta$ be the
induced surjection and put $\beta/\alpha=\rho_r$. Then
\[
 (X/\alpha)/(\beta/\alpha)\cong X/\beta,\qquad
 \mathcal Q_U(X/\alpha)\cong
 \{\beta\in\mathcal Q_U(X):\alpha\le\beta\}.
\]
The order isomorphism sends a quotient of $X/\alpha$ to its
composite with $q_\alpha$. On normal clusters it is inverse image
along $q_\alpha$, with inverse given by direct image.
\item For a set-indexed family of weak quotient data
$(\alpha_i)_{i\in I}$ on $X$, the joint morphism
$X\to\prod_i X/\alpha_i$ is an initial embedding exactly when
\[
 \bigcap_i R_{\alpha_i}=\Delta_{UX}
 \quad\text{and}\quad
 (q_{\alpha_i}:X\to X/\alpha_i)_{i\in I}
 \text{ is a }U\text{-initial source}.
\]
In this case it is a subdirect embedding, since every component
is surjective.
\end{enumerate}
\end{proposition}
\begin{proof}
Every underlying surjection is, up to a unique isomorphism commuting
with the quotient map, the class map for a weak quotient datum.
The universal property of a quotient gives
$\Ker e\le\Ker d$ exactly when $d$ factors through $e$.
Conversely every normal cluster is the kernel of such a quotient,
by (N2). Every initial injection represents the initial structure
on its image, and the dual embedding property gives the second
order isomorphism in (i).

The initial lift of the inclusion $f(UX)\subseteq UY$ gives the
usual surjection--initial-embedding factorization. This proves (ii),
and applying it to $q_\alpha i_A$ and to $r$ gives (iii) and the
first assertion of (iv). Quotients of $X$ factoring through
$q_\alpha$ correspond exactly to quotients of its codomain, by
cancellation of the underlying surjection. The kernel identity
$\Ker(sq_\alpha)=q_\alpha^*\Ker s$ and
\cref{lem:intervals} identify this correspondence with transport.
Finally, the underlying joint function is injective exactly when
the displayed intersection is the diagonal, and the initial
structure of a product makes its initiality equivalent to that
of the component source. This proves (v).
\end{proof}

A weak quotient retains a specified structure on the set of classes;
the quotient map need not be final. This is necessary because our
left class consists of \emph{all} underlying surjections. The
examples below spell out the admissible quotient structures and the
initial-source condition, rather than repeating the general proof.
They impose no separation, countability, or finiteness assumptions
except those explicitly stated. All forms supplied by
\cref{tc:thm:form} have complete algebraic modular fibres. The
dictionary describes their normal and conormal parts; it does not
identify all clusters with weak quotients or with subsets, nor does
it identify intrinsic meets of quotient data with ambient cluster
meets.

\subsection{Topological spaces and their congruences}\label{sec:top-spaces}
\label{top:sec:form}

For topological spaces, the general construction yields a form in which
normal clusters record Veldsman's congruences and conormal clusters record
subspaces. The familiar homeomorphism theorems then become instances of
the quotient--image calculus, without further transfer arguments.

\paragraph{The test and the explicit form.}
The Sierpi\'nski space $\mathbb S=\{0,1\}$, with open sets
$\varnothing,\{1\},\mathbb S$, is the classical initial test for
topology. It is injective with respect to topological embeddings
\cite[Examples~9.3(4)(c),(d) and Proposition~9.5]{AHS}; this is also
the two-element-frame case of \cref{tc:cor:fuzzy}.
Indeed, continuous maps $X\to\mathbb S$ are characteristic functions
of open subsets of $X$. Every open subset of a subspace is the trace
of an ambient open subset, so each such characteristic function extends.
The initial embeddings are the usual topological embeddings, and the
associated factorization system is that of continuous surjections and
topological embeddings \cite[Section~2]{ELA2015}.

\begin{corollary}\label{top:cor:topnoetherian}
Let $\Top$ be the category of all topological spaces and continuous
maps. For a space $X$, write $UX$ for its underlying set and
$\Omega X$ for its topology, and put $A_X=\{0,1\}^{\Pow(UX)}$.
Let $G(X)$ consist of all finite words in $A_X$, including the empty
word, whose letters have pairwise distinct restrictions to $\Omega X$.
For a continuous map $f:X\to Y$, define $Gf$ by replacing each
letter $\phi$ by $B\mapsto\phi(f^{-1}B)$, for $B\subseteq UY$,
and retaining only the first letter with each restriction to $\Omega Y$.
Let $H:\Top\to\Ab$ be its free abelian linearization:
\begin{equation}\label{top:eq:Hexplicit}
HX=\mathbb Z^{(G(X))},\qquad Hf([w])=[Gf(w)].
\end{equation}
Equip $\Set$ and $\Ab$ with their surjection--injection systems.
Then $G$ and $H$ are transfer functors from the continuous-surjection--
topological-embedding system to these systems. For $f:X\to Y$,
$S\le HX$, and $T\le HY$, the form $\N$ given by
\begin{equation}\label{top:eq:form-fibres}
\N_X=\Sub_{\Ab}(HX),\qquad
f_!(S)=Hf(S),\qquad f^*(T)=(Hf)^{-1}(T)
\end{equation}
is noetherian, has complete algebraic modular fibres, and induces
precisely that factorization system. Empty spaces are included.
\end{corollary}
\begin{proof}
Apply \cref{tc:thm:form} to the single injective initial test
$\mathbb S$. The Boolean-function and reduced-word constructions
specialize to the displayed formulas.
\end{proof}

\begin{corollary}\label{cor:top-target}
Let $\C$ be a category with a proper factorization system
$(\E,\M)$, pullbacks of pairs of $\M$-morphisms with common
codomain, and pushouts of pairs of $\E$-morphisms with common
domain. Equip $\Top$ with continuous surjections and topological
embeddings. Then $\C$ admits a noetherian form inducing
$(\E,\M)$ if and only if there is a transfer functor
$\C\to\Top$ for these systems.
\end{corollary}
\begin{proof}
The discrete-space functor $\Set\to\Top$ is conservative and
left adjoint to the underlying-set functor, and preserves the
designated classes and pullbacks of two injections. It is therefore
a transfer functor by \cref{prop:transfer-elementary}. Together
with the preceding transfer $G:\Top\to\Set$, this gives the
claim by \cref{prop:target}.
\end{proof}

\paragraph{Congruences and their cluster interpretation.}
We use Veldsman's congruences and weak quotients
\cite{Veldsman2019}\cite[Section~3.1]{Veldsman2022}.
A \emph{congruence} on a space $X$ is a pair
$\rho=(R_\rho,\sigma_\rho)$, where $R_\rho$ is an equivalence
relation on $UX$ and $\sigma_\rho\subseteq\Omega X$ is a topology
of $R_\rho$-saturated sets. Its \emph{weak quotient} is the set
$X/R_\rho$, written $X/\rho$, with topology
\begin{equation}\label{vl:eq:weak-topology}
\Omega(X/\rho)=\{V\subseteq X/R_\rho:q_\rho^{-1}(V)\in\sigma_\rho\}
=\{q_\rho(U):U\in\sigma_\rho\},
\end{equation}
where $q_\rho(x)=[x]_\rho$. It is a continuous surjection.
It is an ordinary topological quotient map exactly when $\rho$ is
\emph{strong}, meaning that $\sigma_\rho$ contains every
$R_\rho$-saturated open subset of $X$.
Write $\VCon(X)$ for these congruences, with order
\begin{equation}\label{vl:eq:con-order}
\rho\leq\eta\quad\Longleftrightarrow\quad
R_\rho\subseteq R_\eta\ \text{ and }\ \sigma_\eta\subseteq\sigma_\rho.
\end{equation}
Its least and greatest elements are
$\iota_X=(\Delta_{UX},\Omega X)$ and
$\upsilon_X=(UX\times UX,\{\varnothing,UX\})$.

The quotient classification used below is Veldsman's First
Homeomorphism Theorem \cite[Theorem~3.2]{Veldsman2022}:
continuous surjections out of $X$, up to isomorphism of quotient
arrows, correspond to congruences. We combine this classification
with the general cluster dictionary of
\cref{tc:prop:quotient-dictionary}.

\begin{proposition}[Topological dictionary]\label{vl:thm:comparison}
\label{vl:prop:dictionary}
Let $\F$ be a noetherian form on $\Top$ whose quotients are
continuous surjections and whose embeddings are topological
embeddings. For each space $X$, let $\VCon(X)$ be the poset of
pairs $(R,\sigma)$ consisting of an equivalence relation on $UX$
and a subtopology of $\Omega X$ of $R$-saturated sets, ordered by
inclusion of relations and reverse inclusion of topologies. For
$\rho=(R,\sigma)$ let $q_\rho:X\to X/R$ have codomain topology
$\{q_\rho(U):U\in\sigma\}$. For $S\subseteq UX$, let
$i_S:S_{\mathrm{sub}}\hookrightarrow X$ be the subspace inclusion.
Then the assignments
\begin{equation}\label{vl:eq:abstract-kappa}
\rho\longmapsto\kappa_X(\rho)=\Ker_{\F}q_\rho,
\qquad S\longmapsto c_X(S)=\Img_{\F}i_S
\end{equation}
are order isomorphisms from $\VCon(X)$ and $\Pow(UX)$ to the
normal- and conormal-cluster orders of $\F$ over $X$, respectively.
For a continuous map $f:X\to Y$, put
\[
\rho_f=(R_f,f^{-1}(\Omega Y)),\qquad
xR_fx'\ \Longleftrightarrow\ f(x)=f(x').
\]
Then $\kappa_X(\rho_f)=\Ker_{\F}f$ and
$c_Y(f(X))=\Img_{\F}f$.
\end{proposition}
\begin{proof}
Veldsman's quotient classification identifies $\rho$ with $q_\rho$;
its order is the factorization order because
$\rho\leq\eta$ exactly when $q_\eta$ factors continuously through
$q_\rho$. Topological embeddings into $X$, up to isomorphism over
$X$, are precisely its subspaces, ordered by inclusion. Apply
\cref{tc:prop:quotient-dictionary}. The factorization of $f$
through $f(X)_{\mathrm{sub}}$ has the same congruence $\rho_f$,
which gives the final assertions.
\end{proof}

Thus all continuous surjections, including those which are not ordinary
quotient maps, are projections in $\F$. The topology component is
essential: the identity function from the discrete two-point space
$D$ to the indiscrete two-point space has congruence
$(\Delta_{UD},\{\varnothing,UD\})\ne\iota_D$. Its cluster kernel
records loss of topology without any point identifications.

\paragraph{Homeomorphisms, correspondence, and subdirect products.}
The following packages Veldsman's homeomorphism, correspondence,
and subdirect-product theorems
\cite[Theorems~3.2--3.3 and 3.5, Lemma~3.4, Corollary~3.6, and
Theorem~3.7]{Veldsman2022}. Here they follow from the preceding
dictionary and the general noetherian-form calculus.

\begin{corollary}\label{vl:thm:first}\label{vl:thm:second}
\label{vl:thm:third}\label{vl:thm:correspondence}\label{vl:cor:subdirect}
For a space $Z$, let $\VCon(Z)$ consist of pairs
$\gamma=(R_\gamma,\sigma_\gamma)$, where $R_\gamma$ is an
equivalence relation on $UZ$ and $\sigma_\gamma\subseteq\Omega Z$
is a topology of $R_\gamma$-saturated sets. Order these pairs by
relation inclusion and reverse topology inclusion, and give
$Z/\gamma=UZ/R_\gamma$ the topology
$\{q_\gamma(U):U\in\sigma_\gamma\}$, where
$q_\gamma(z)=[z]_\gamma$. Then the following hold.
\begin{enumerate}
\item For a continuous $f:X\to Y$, define
$\rho_f=(R_f,f^{-1}(\Omega Y))$ by
$xR_fx'\Longleftrightarrow f(x)=f(x')$. The canonical map
\[
X/\rho_f\xrightarrow{\ \cong\ }f(X)_{\mathrm{sub}},
\qquad [x]\longmapsto f(x),
\]
is a homeomorphism, where the target has the subspace topology.
\item For $\rho\in\VCon(X)$ and $S\subseteq UX$ with its
subspace topology, put
$\rho|_S=(R_\rho\cap(S\times S),\{U\cap S:U\in\sigma_\rho\})$.
There is a canonical homeomorphism
\[
S/(\rho|_S)\xrightarrow{\ \cong\ }q_\rho(S)_{\mathrm{sub}},
\qquad [s]_{\rho|_S}\longmapsto[s]_\rho.
\]
\item For $\alpha\leq\beta$ in $\VCon(X)$, define
$\beta/\alpha$ on $X/\alpha$ by
$[x]_\alpha R_{\beta/\alpha}[y]_\alpha\Longleftrightarrow xR_\beta y$
and $\sigma_{\beta/\alpha}=\{q_\alpha(U):U\in\sigma_\beta\}$.
There is a canonical homeomorphism
\[
(X/\alpha)/(\beta/\alpha)\xrightarrow{\ \cong\ }X/\beta,
\qquad [[x]_\alpha]\longmapsto[x]_\beta.
\]
Moreover, $\beta\mapsto\beta/\alpha$ is an order isomorphism
$\{\beta\in\VCon(X):\alpha\leq\beta\}\cong\VCon(X/\alpha)$.
\item For a set-indexed family
$(\rho_i=(R_i,\sigma_i))_{i\in I}$ in $\VCon(X)$, the joint map
$j=(q_{\rho_i})_{i\in I}:X\to\prod_{i\in I}X/\rho_i$, with
the product topology on the target, is a topological embedding
if and only if
\begin{equation}\label{vl:eq:subdirect-criterion}
\bigcap_{i\in I}R_i=\Delta_{UX},\qquad
\topgen{\bigcup_{i\in I}\sigma_i}=\Omega X.
\end{equation}
Its coordinate maps are surjective, so this is exactly a subdirect
embedding. Empty spaces and the empty family are allowed.
\end{enumerate}
\end{corollary}
\begin{proof}
Choose the form of \cref{top:cor:topnoetherian}. By the dictionary,
its quotient--image factorization gives (1). Apply the same
factorization to $q_\rho i_S$ for (2), and to
$r:X/\alpha\to X/\beta$, $r([x]_\alpha)=[x]_\beta$, for (3).
Their congruences are respectively $\rho|_S$ and $\beta/\alpha$.
Under the dictionary these operations read
\[
\kappa_S(\rho|_S)=i_S^*\kappa_X(\rho),\qquad
\kappa_{X/\alpha}(\beta/\alpha)=(q_\alpha)_!\kappa_X(\beta).
\]
The correspondence assertion is therefore the normal-cluster
correspondence along $q_\alpha$ from
\cref{tc:prop:quotient-dictionary}. Finally, the relation and
inverse-image topology of $j$ are
$\bigcap_iR_i$ and $\topgen{\bigcup_i\sigma_i}$. Thus (4)
expresses that its congruence is $\iota_X$, equivalently that its
cluster kernel is bottom and $j$ is an embedding.
\end{proof}

For clarity, the meets used in the subdirect criterion are intrinsic
congruence meets. Veldsman's lattice formulas are
\cite[Section~3.1, ``Ordering of congruences'']{Veldsman2022}
\begin{equation}\label{vl:eq:con-meet}
\bigwedge_i\rho_i=
\left(\bigcap_iR_i,\topgen{\bigcup_i\sigma_i}\right),\qquad
\bigvee_i\rho_i=
\left(\operatorname{Eq}(\bigcup_iR_i),\bigcap_i\sigma_i\right),
\end{equation}
where $\operatorname{Eq}$ denotes generated equivalence relation and
the topology intersection is taken inside $\Omega X$. Empty meets
and joins are $\upsilon_X$ and $\iota_X$, respectively. The dictionary
identifies these with the intrinsic normal-cluster order; it does
not identify intrinsic meets with arbitrary ambient cluster meets.
In particular, modularity of the full subgroup fibres does not
assert modularity of the congruence lattice. The canonical maps
above are realized in $\Top$ by the form, so their construction
requires no lifting of arbitrary abelian-group isomorphisms.

\subsection{Measurable spaces}\label{tc:sec:measurable}
Here $\mathbf{Meas}$ denotes all measurable spaces and measurable maps;
no measure, separation condition, or countability condition is imposed.
The surjection--trace-embedding factorization system used below is
described by Jacobs and Sokolova \cite[Section~3.1]{JS2010}.

\begin{corollary}\label{tc:cor:measurable}
Let $\mathbf{Meas}$ be the category of sets equipped with
$\sigma$-algebras and functions whose inverse images preserve
measurable subsets. Then $\mathbf{Meas}$ has a noetherian form with
complete algebraic modular fibres, inducing surjective measurable maps
and injective measurable maps whose domains carry the trace
$\sigma$-algebras from their codomains.
\end{corollary}
\begin{proof}
Initial $\sigma$-algebras are generated by inverse images, so the
underlying-set functor is topological. The discrete measurable space
$Q=\{0,1\}$ is an initial test: measurable maps to $Q$ are exactly
characteristic functions of measurable subsets. If $S\subseteq X$
has the trace $\sigma$-algebra, each measurable $B\subseteq S$ is
$S\cap A$ for some measurable $A\subseteq X$; hence $\chi_B$
extends to $\chi_A$. Thus $Q$ is injective with respect to initial
embeddings, and \cref{tc:thm:form} applies.
\end{proof}

We now make the normal and conormal clusters explicit. Saturated
measurable sets and their relation to sub-$\sigma$-algebras are classical;
see Battenfeld \cite[Section~2.2]{Battenfeld2010}. The following paired
description retains both the equivalence relation and the measurable
structure chosen on its quotient.

\begin{definition}\label{meas:def:congruence}
For a measurable space $X=(X,\Sigma_X)$, a \emph{measurable congruence}
is a pair $\rho=(R_\rho,\mathcal A_\rho)$, where $R_\rho$ is an
equivalence relation on $X$ and $\mathcal A_\rho\subseteq\Sigma_X$
is a sub-$\sigma$-algebra consisting of $R_\rho$-saturated sets.
Write $\Con_{\sigma}(X)$ for these pairs, ordered by
\[
\rho\leq\eta\quad\Longleftrightarrow\quad
R_\rho\subseteq R_\eta\quad\text{and}\quad
\mathcal A_\eta\subseteq\mathcal A_\rho.
\]
Its \emph{weak quotient} is the measurable space
\begin{equation}\label{meas:eq:weak-quotient}
X/\rho=\left(X/R_\rho,
 \{B\subseteq X/R_\rho:q_\rho^{-1}(B)\in\mathcal A_\rho\}\right),
\qquad q_\rho(x)=[x]_{R_\rho}.
\end{equation}
\end{definition}

Thus $q_\rho$ is measurable and surjective, and the inverse images of
the measurable subsets of $X/\rho$ are exactly $\mathcal A_\rho$.
It has the usual final quotient $\sigma$-algebra precisely when
$\mathcal A_\rho$ contains every $R_\rho$-saturated member of $\Sigma_X$.
Allowing a smaller $\mathcal A_\rho$ is necessary because the quotient
class of the form consists of all surjective measurable maps.

The next description combines the factorization of
\cite[Section~3.1]{JS2010} with the general quotient dictionary
of \cref{tc:prop:quotient-dictionary}.
\begin{proposition}\label{meas:prop:dictionary}
Let $\F$ be a noetherian form on $\mathbf{Meas}$ whose quotient maps
are the surjective measurable maps and whose embeddings are the
injective measurable maps carrying the trace $\sigma$-algebra.
For $X=(X,\Sigma_X)$, write $\Nrm_{\F}(X)$ and $\Cnm_{\F}(X)$
for its ordered sets of normal and conormal clusters. Let
$\Con_{\sigma}(X)$ consist of pairs $(R,\mathcal A)$ with $R$ an
equivalence relation and $\mathcal A\subseteq\Sigma_X$ a
sub-$\sigma$-algebra of $R$-saturated sets, ordered by increasing
$R$ and decreasing $\mathcal A$; equip $X/R$ with
$\{B:q^{-1}(B)\in\mathcal A\}$ and write $q:X\to X/R$ for
the class map. Then there are order isomorphisms
\begin{equation}\label{meas:eq:cluster-dictionary}
\begin{aligned}
\Con_{\sigma}(X)&\longrightarrow\Nrm_{\F}(X),&
 (R,\mathcal A)&\longmapsto\Ker_{\F}q,\\
\Pow(X)&\longrightarrow\Cnm_{\F}(X),&
 S&\longmapsto\Img_{\F}(S\hookrightarrow X),
\end{aligned}
\end{equation}
where $S$ has $\Sigma_X|_S=\{S\cap A:A\in\Sigma_X\}$.
In particular, $S$ need not be measurable in $X$.
\end{proposition}
\begin{proof}
For a measurable map $f:X\to Y$, put
\begin{equation}\label{meas:eq:map-congruence}
\rho_f=(R_f,f^{-1}\Sigma_Y),\qquad
xR_fx'\ \Longleftrightarrow\ f(x)=f(x').
\end{equation}
The canonical bijection $X/\rho_f\to f(X)$ identifies the quotient
$\sigma$-algebra with the trace from $Y$. Consequently the weak
quotients classify all surjective measurable maps, up to isomorphism
under $X$. Moreover, $f$ factors measurably through $q_\rho$ exactly
when $R_\rho\subseteq R_f$ and
$f^{-1}\Sigma_Y\subseteq\mathcal A_\rho$. This identifies quotient
factorization order with the displayed congruence order. Embeddings
over $X$ are classified by arbitrary trace subspaces. Apply
\cref{tc:prop:quotient-dictionary}.
\end{proof}

The following isomorphisms are elementary consequences of that same
surjection--trace-embedding factorization
\cite[Section~3.1]{JS2010}; their role here is to express its
isomorphism and correspondence principles in concrete terms.
\begin{corollary}[Measurable isomorphism and correspondence rules]
\label{meas:cor:isomorphisms}
For each measurable space $Z=(Z,\Sigma_Z)$, let
$\Con_{\sigma}(Z)$ be the pairs $(R,\mathcal A)$ with $R$ an
equivalence relation and $\mathcal A\subseteq\Sigma_Z$ a
sub-$\sigma$-algebra of $R$-saturated sets; order them by
$(R,\mathcal A)\leq(S,\mathcal B)$ when
$R\subseteq S$ and $\mathcal B\subseteq\mathcal A$.
For such a pair $\gamma$, give $Z/\gamma=Z/R_\gamma$ the
$\sigma$-algebra $\{B:q_\gamma^{-1}(B)\in\mathcal A_\gamma\}$,
where $q_\gamma(z)=[z]_{R_\gamma}$.
The following canonical maps are isomorphisms of measurable spaces.
\begin{enumerate}[label=\textup{(\roman*)}]
\item If $f:X\to Y$ is measurable and
$\rho_f=(\{(x,x'):f(x)=f(x')\},f^{-1}\Sigma_Y)$, then
\[
 X/\rho_f\xrightarrow{\sim} f(X),\qquad [x]\longmapsto f(x),
\]
where $f(X)$ has the trace $\sigma$-algebra from $Y$.
\item If $\rho=(R,\mathcal A)\in\Con_{\sigma}(X)$ and
$S\subseteq X$ has the trace $\sigma$-algebra, put
$\rho|_S=(R\cap(S\times S),\mathcal A|_S)$, where
$\mathcal A|_S=\{S\cap A:A\in\mathcal A\}$. Then
\[
 S/(\rho|_S)\xrightarrow{\sim} q_\rho(S),\qquad
 [s]\longmapsto q_\rho(s),
\]
where $q_\rho(S)$ has the trace $\sigma$-algebra from $X/\rho$.
\item If $\alpha=(R,\mathcal A)\leq\beta=(S,\mathcal B)$
in $\Con_{\sigma}(X)$, define $\beta/\alpha$ on $X/\alpha$ by
\[
 [x]_R\mathrel{(S/R)}[y]_R\ \Longleftrightarrow\ xSy,
 \qquad \mathcal B/\alpha=\{q_\alpha(B):B\in\mathcal B\}.
\]
Then
\[
 (X/\alpha)/(\beta/\alpha)\xrightarrow{\sim}X/\beta,
 \qquad [[x]_R]_{S/R}\longmapsto[x]_S.
\]
For fixed $\alpha$, the assignment $\beta\mapsto\beta/\alpha$
is an order isomorphism
\[
 \{\beta\in\Con_{\sigma}(X):\alpha\leq\beta\}
 \xrightarrow{\sim}\Con_{\sigma}(X/\alpha).
\]
\end{enumerate}
\end{corollary}
\begin{proof}
These are the first isomorphism, restriction, iteration, and interval
correspondence of \cref{tc:prop:quotient-dictionary}, expressed through
\cref{meas:prop:dictionary}. In (ii), the congruence of
$q_\rho|_S$ is $\rho|_S$; in (iii), that of
$X/\alpha\to X/\beta$ is $\beta/\alpha$.
The inverse of the correspondence sends
$(T,\mathcal C)$ on $X/\alpha$ to
$((q_\alpha\times q_\alpha)^{-1}(T),q_\alpha^{-1}\mathcal C)$.
\end{proof}

For completeness, the concrete order also displays the subdirect
representation criterion. This uses only the usual initial
$\sigma$-algebra on a product and the trace-embedding class above.
\begin{proposition}\label{meas:prop:subdirect}
Let $X=(X,\Sigma_X)$ be a measurable space and let
$\rho_i=(R_i,\mathcal A_i)$, $i\in I$, be a set-indexed family
with $R_i$ an equivalence relation on $X$ and
$\mathcal A_i\subseteq\Sigma_X$ a sub-$\sigma$-algebra of
$R_i$-saturated sets. Give $X/R_i$ the $\sigma$-algebra
$\{B:q_i^{-1}(B)\in\mathcal A_i\}$, where $q_i$ is the class map,
and give $\prod_iX/R_i$ its product $\sigma$-algebra. Then
$\langle q_i\rangle:X\to\prod_iX/R_i$ is a measurable embedding
with surjective coordinate maps if and only if
\[
 \bigcap_{i\in I}R_i=\Delta_X,
 \qquad \sigma\!\left(\bigcup_{i\in I}\mathcal A_i\right)=\Sigma_X.
\]
Here $\sigma(\mathcal H)$ denotes the $\sigma$-algebra generated
by a family $\mathcal H$ of subsets of $X$.
\end{proposition}
\begin{proof}
Apply the subdirect criterion of \cref{tc:prop:quotient-dictionary}.
The first equality is joint injectivity. The inverse images of the
product generators generate $\sigma(\bigcup_i\mathcal A_i)$, so
the second equality is initiality. Each $q_i$ is already surjective.
\end{proof}

\subsection{Quantale-enriched categories}\label{tc:sec:quantales}
Preorders, equivalence relations, and several metric examples belong
to a common enriched construction. Let
$V=(V,\leq,\otimes,k)$ be a set-sized commutative unital quantale:
$V$ is a complete lattice and $\otimes$ preserves arbitrary joins in
each variable. Write $u\Rightarrow v$ for its residual, characterized by
$u\otimes w\leq v$ if and only if $w\leq(u\Rightarrow v)$.
A $V$-category is a set $X$ with a function $a:X\times X\to V$ such that
\[
k\leq a(x,x),\qquad a(x,y)\otimes a(y,z)\leq a(x,z).
\]
A $V$-functor $f:(X,a)\to(Y,b)$ satisfies
$a(x,y)\leq b(fx,fy)$. No separation condition is imposed.

Hofmann and Nora recall that $V\text{-}\mathbf{Cat}$ is topological
and that the residual $V$-category is an injective initial test
\cite[Theorem~2.5 and Proposition~2.6]{HN2023}. Combining these
facts with our transfer construction gives the following result;
we also spell out the symmetric variant.

\begin{theorem}\label{tc:thm:quantale}
Let $V=(V,\leq,\otimes,k)$ be a set-sized commutative unital
quantale: $V$ is a complete lattice, and $\otimes$ is associative,
commutative, has unit $k$, and preserves arbitrary joins in each
variable. Let $V\text{-}\mathbf{Cat}$ have as objects sets $X$ with
functions $a:X\times X\to V$ satisfying
\[
k\leq a(x,x),\qquad a(x,y)\otimes a(y,z)\leq a(x,z),
\]
and as morphisms functions $f:(X,a)\to(Y,b)$ satisfying
$a(x,y)\leq b(fx,fy)$. No separation condition is imposed.
Then both $V\text{-}\mathbf{Cat}$ and its full subcategory of
symmetric objects, for which $a(x,y)=a(y,x)$, admit noetherian forms
inducing surjective morphisms and injective fully faithful morphisms,
where fully faithful means $a(x,y)=b(fx,fy)$ for every $x,y$.
\end{theorem}
\begin{proof}
Initial structures are given by
$a(x,y)=\bigwedge_i a_i(f_i x,f_i y)$; hence the forgetful functor is
topological, also in the symmetric case. Its initial embeddings are
the injective maps for which the $V$-functor inequality is an equality.

The single test for $V\text{-}\mathbf{Cat}$ is
$Q=(V,q)$, where $q(u,v)=u\Rightarrow v$. The following formulas
give a direct verification of its two required properties. The $V$-functors
$a(z,-):X\to Q$ satisfy
\begin{equation}\label{tc:eq:quantale-initial}
a(x,y)=\bigwedge_{z\in X}\bigl(a(z,x)\Rightarrow a(z,y)\bigr).
\end{equation}
Transitivity gives one inequality. For the other, take $z=x$ and use
$k\leq a(x,x)$. Thus these maps initially determine $X$.

For a full induced subcategory $A\subseteq(X,a)$ and a $V$-functor
$f:A\to Q$, define
\begin{equation}\label{tc:eq:quantale-extension}
\overline f(x)=\bigvee_{u\in A} f(u)\otimes a(u,x).
\end{equation}
The functor inequality for $f$ shows that $\overline f|_A=f$.
Transitivity gives
$\overline f(x)\otimes a(x,y)\leq\overline f(y)$, so this is the
required extension. The formula includes $A=\varnothing$.

For symmetric $V$-categories use instead the symmetric test
\[
q_{\mathrm s}(u,v)=(u\Rightarrow v)\wedge(v\Rightarrow u).
\]
A $V$-functor from a symmetric domain to $Q$ automatically takes
values in this symmetric structure: interchange $x,y$ in its functor
inequality. Consequently \eqref{tc:eq:quantale-initial} still detects
the structure, and \eqref{tc:eq:quantale-extension} still extends every
test. Apply Theorem~\ref{tc:thm:form} in both cases.
\end{proof}

The metric interpretations in the next table are recalled in
\cite[Example~2.3]{HN2023}; the probabilistic enrichment below is
developed in \cite{HR2013}.

\begin{corollary}\label{tc:cor:quantale-examples}
For each quantale $V=(V,\leq,\otimes,k)$ in the table below, let
$V\text{-}\mathbf{Cat}$ consist of sets $X$ with functions
$a:X\times X\to V$ satisfying $k\leq a(x,x)$ and
$a(x,y)\otimes a(y,z)\leq a(x,z)$, and functions
$f:(X,a)\to(Y,b)$ with $a(x,y)\leq b(fx,fy)$.
Each such category, and its full subcategory defined by
$a(x,y)=a(y,x)$, has a noetherian form inducing underlying
surjections and injective maps with $a(x,y)=b(fx,fy)$.
\begin{center}
\begin{tabularx}{\textwidth}{@{}>{\raggedright\arraybackslash}p{.43\textwidth}Y@{}}
\toprule
Quantale $(V,\leq,\otimes,k)$ & Category of $V$-categories\\
\midrule
$([0,\infty],\geq,+,0)$ & Extended quasi-pseudometric spaces\\
$([0,\infty],\geq,\max,0)$ & Extended quasi-ultrametric spaces\\
$([0,1],\geq,\min(1,u+v),0)$ & Quasi-pseudometric spaces bounded by $1$\\
$([0,1],\leq,T,1)$, $T$ a left-continuous t-norm & Fuzzy preorders\\
\bottomrule
\end{tabularx}
\end{center}
Here a t-norm is an associative, commutative, monotone operation on
$[0,1]$ with unit $1$. Morphisms are nonexpansive in the metric cases
and preserve degrees in the fuzzy case. Symmetric examples include
extended pseudo-ultrametric spaces and fuzzy similarity spaces.
Distinct points may have zero distance, or similarity degree $1$.
\end{corollary}
\begin{proof}
Apply \cref{tc:thm:quantale} to each displayed quantale, using its
symmetric case for the symmetric examples.
\end{proof}

There is also a probabilistic family. For a fixed left-continuous
t-norm $T$, let $\Delta$ consist of functions
$\varphi:[0,\infty]\to[0,1]$ satisfying
$\varphi(t)=\sup_{s<t}\varphi(s)$, ordered pointwise, and put
\[
(\varphi\otimes_T\psi)(t)
 =\sup_{r+s\leq t}T(\varphi(r),\psi(s)),\qquad
\kappa(0)=0,\quad\kappa(t)=1\ (t>0).
\]
These operations give a commutative unital quantale; see
\cite{HR2013}. Theorem~\ref{tc:thm:quantale} therefore gives forms
on generalized probabilistic quasi-pseudometric spaces and on their
symmetric counterparts. Precisely, these are $\Delta$-categories,
with $\kappa\leq d(x,x)$ and
$d(x,y)\otimes_T d(y,z)\leq d(x,z)$, and maps satisfying
$d(x,y)\leq d'(fx,fy)$. This convention allows defective distributions
and does not require $d(x,y)(\infty)=1$ or separation. Thus no
normalization or separation subcategory is being asserted to be
topological by this argument.

The quotient dictionary is explicit in this setting as well. For a
$V$-category $(X,a)$, its weak quotient data are an equivalence
relation $R$ on $X$ and a $V$-category structure $b$ on $X/R$
such that
\[
 a(x,y)\leq b([x]_R,[y]_R)\qquad(x,y\in X).
\]
In the symmetric case require $b$ to be symmetric. These data
describe normal clusters by \cref{tc:prop:quotient-dictionary};
conormal clusters are arbitrary subsets with $a$ restricted to
them. For a $V$-functor $f:(X,a)\to(Y,c)$, the first isomorphism
uses equality of $f$-values and the structure
$b([x],[y])=c(fx,fy)$. Restrictions and iterated quotients use the
corresponding restricted and transported structures. A family of
surjective $V$-functors $f_i:(X,a)\to(Y_i,b_i)$, indexed by a set,
is subdirect
exactly when it jointly separates points and
\[
 a(x,y)=\bigwedge_i b_i(f_i x,f_i y),
\]
where the meet is taken in $V$. These formulas specialize to the
order and distance calculi below.

\subsection{Preorders}
The two-element chain is the standard injective order test; this is
also the Boolean-quantale case of \cite[Proposition~2.6]{HN2023}.
The following is its noetherian-form consequence.

\begin{corollary}\label{tc:cor:preord}
Let $\mathbf{Preord}$ be the category of sets equipped with reflexive,
transitive relations and monotone maps; antisymmetry is not required.
Then $\mathbf{Preord}$ has a noetherian form inducing surjective
monotone maps and order-reflecting monotone injections.
\end{corollary}
\begin{proof}
Apply \cref{tc:thm:quantale} to the Boolean quantale
$V=(\{0,1\},\leq,\wedge,1)$. A $V$-category is precisely a
preorder, with $a(x,y)=1$ if and only if $x\leq y$;
$V$-functors are monotone maps and injective fully faithful maps
are order-reflecting injections. The residual test is the chain
$0<1$.
\end{proof}

This assertion concerns preorders. The underlying-set functor on
partially ordered sets is not topological: an arbitrary initial
preorder need not be antisymmetric. This distinction, and the
topologicity of $\mathbf{Preord}$, are recorded in
\cite[Examples~21.2 and 21.8]{AHS}.

The following calculation specializes the quotient dictionary to the
usual induced preorders and surjective monotone maps. In particular,
the target preorder is part of the quotient data: it can contain
comparisons not generated by the source preorder.

\begin{corollary}\label{tc:cor:preorder-calculus}
Let $\mathbf{Preord}$ be the category of preordered sets and monotone
maps, and let $\mathcal F$ be a noetherian form on it inducing
surjective monotone maps and order-reflecting monotone injections.
For a preordered set $(X,\leq_X)$, define a quotient datum to be
$\rho=(R_\rho,\preceq_\rho)$, where $R_\rho$ is an equivalence
relation on the underlying set and $\preceq_\rho$ is a preorder on
$X/R_\rho$ such that
\[
 x\leq_X y\quad\Longrightarrow\quad
 [x]_\rho\preceq_\rho[y]_\rho.
\]
Write $q_\rho:X\to X/\rho=(X/R_\rho,\preceq_\rho)$ for the
canonical surjection. Order these data by
\[
 \rho\leq\sigma\ \Longleftrightarrow\
 R_\rho\subseteq R_\sigma\ \text{ and }\
 \bigl([x]_\rho\preceq_\rho[y]_\rho
 \Longrightarrow[x]_\sigma\preceq_\sigma[y]_\sigma\bigr)
 \quad(x,y\in X).
\]
Then $\rho\mapsto\Ker_{\mathcal F}q_\rho$ identifies this ordered
set with the normal clusters on $X$. The conormal clusters are
identified with all subsets $A\subseteq X$, equipped with the
preorder induced from $X$.

For a monotone map $f:X\to Y$, put
$xR_fy\Longleftrightarrow f(x)=f(y)$ and
$[x]_{\rho_f}\preceq_{\rho_f}[y]_{\rho_f}
\Longleftrightarrow f(x)\leq_Y f(y)$. Then there is a canonical
isomorphism of preordered sets
\[
 X/\rho_f\ \cong\ f(X),
\]
where $f(X)$ has the preorder induced from $Y$.
For a quotient datum $\rho$ and a subset $A\subseteq X$, restrict
$R_\rho$ to $A$ and transport the induced preorder on $q_\rho(A)$
to $A/(R_\rho\cap A^2)$; call the resulting datum $\rho|_A$.
Then
\[
 A/(\rho|_A)\cong q_\rho(A).
\]
For $\rho\leq\sigma$, the quotient datum $\sigma/\rho$ on
$X/\rho$ identifies $R_\rho$-classes having the same
$R_\sigma$-class and equips the resulting set with
$\preceq_\sigma$. It satisfies
\[
 (X/\rho)/(\sigma/\rho)\cong X/\sigma.
\]
Moreover, $\sigma\mapsto\sigma/\rho$ identifies quotient data on
$X$ above $\rho$ with all quotient data on $X/\rho$.

For a set-indexed family of surjective monotone maps $f_i:X\to Y_i$, the
joint map $X\to\prod_iY_i$ is an embedding precisely when the
$f_i$ jointly separate points and
\[
 x\leq_X y\quad\Longleftrightarrow\quad
 f_i(x)\leq_{Y_i}f_i(y)\text{ for every }i.
\]
\end{corollary}
\begin{proof}
Every surjective monotone map determines exactly the displayed
quotient data, up to its unique identification with the quotient of
its underlying set. Factorization through another such map is
equivalent to the displayed order condition. Apply
\cref{tc:prop:quotient-dictionary}; the three isomorphisms carry
each class to its named image. The final assertion follows from
the coordinatewise preorder on the product.
\end{proof}

\subsection{Equivalence relations, graphs, and relational structures}
\label{tc:sec:relations}

\paragraph{Equivalence relations and graphs.}
\begin{corollary}\label{tc:cor:equiv}
Let $\Equiv$ be the category of sets equipped with equivalence
relations and functions preserving those relations. Then $\Equiv$
has a noetherian form inducing underlying surjections and full
embeddings, meaning injective functions that preserve and reflect
equivalence.
\end{corollary}
\begin{proof}
Apply the symmetric case of \cref{tc:thm:quantale} to the Boolean
quantale $V=(\{0,1\},\leq,\wedge,1)$. Symmetric preorders are
equivalence relations, and injective fully faithful maps are exactly
the full embeddings in the statement.
\end{proof}

Alternatively, the transfer functor $W$ constructed above and
\cref{thm:characterization} give this conclusion directly.

\begin{corollary}\label{tc:cor:graphs}
Consider either the category of pairs $(X,R)$ with $X$ a set and
$R\subseteq X\times X$ a reflexive relation, or its full subcategory
where $R$ is also symmetric. In both cases morphisms are functions
that preserve the relations. These categories of reflexive directed
graphs and reflexive symmetric graphs have noetherian forms inducing
underlying surjections and induced-subgraph embeddings, namely
injective functions that preserve and reflect the relations.
\end{corollary}
\begin{proof}
Apply \cref{tc:thm:horn-tests} to the axiom
$\top\vdash_x R(x,x)$, adding
$R(x,y)\vdash_{x,y}R(y,x)$ in the symmetric case.
The first axiom has no premise variables, and the second retains both
premise variables in its conclusion. The factorization classes in that
theorem are exactly those stated here.
\end{proof}

\paragraph{Relational structures without additional axioms.}
Topologicity of unrestricted relational structures is a standard
instance of functor-structured categories
\cite[Example~21.8(2)]{AHS}. The finite-test theorem for relational
Horn axioms below gives the noetherian-form conclusion as a special case.

\begin{proposition}\label{tc:prop:relational}
Let $\Sigma$ be a set-sized finitary relational signature whose
symbols have positive arity; the signature may be empty. Let
$\mathbf{Str}(\Sigma)$ have as objects sets with an arbitrary
interpretation of each symbol as a relation of its specified arity,
and as morphisms functions that preserve every interpreted relation.
No additional axioms on these relations are imposed. Then
$\mathbf{Str}(\Sigma)$ has a noetherian form inducing surjective
homomorphisms and induced-substructure embeddings, namely injective
homomorphisms that reflect every relation.
\end{proposition}
\begin{proof}
This is \cref{tc:thm:horn-tests} for the empty set of axioms.
Its hypotheses hold vacuously, including when the signature is empty.
The resulting initial embeddings are precisely the injective
homomorphisms that reflect every relation.
\end{proof}

Additional axioms are treated in \cref{tc:sec:matrix-conditions}.
The example in \cref{sec:top-scope} shows why arbitrary universal
Horn axioms cannot automatically be included in the finite-test argument.

For relational structures, induced substructures and their product
relations give an equally explicit quotient dictionary. Here, too,
relations on a quotient may be larger than the direct images of
the source relations.

\begin{corollary}\label{tc:cor:relational-calculus}
Let $\Sigma$ be a set-sized finitary relational signature with
positive arities, let $\mathbf{Str}(\Sigma)$ be the category of
arbitrary $\Sigma$-structures and relation-preserving functions,
and let $\mathcal F$ be a noetherian form inducing surjective
homomorphisms and induced-substructure embeddings. For a
$\Sigma$-structure $X$, a quotient datum is
$\rho=(R_\rho,(P_\rho)_{P\in\Sigma})$, where $R_\rho$ is an
equivalence relation on its underlying set and, for each symbol
$P$ of arity $n$, the relation $P_\rho$ on $X/R_\rho$ satisfies
\[
 q_\rho^{\,n}(P^X)\subseteq P_\rho,
 \qquad q_\rho:X\to X/R_\rho.
\]
Write $X/\rho$ for this quotient structure. Order quotient data
by $\rho\leq\sigma$ when $R_\rho\subseteq R_\sigma$ and the
canonical map $X/R_\rho\to X/R_\sigma$ preserves all relations.
Then $\rho\mapsto\Ker_{\mathcal F}q_\rho$ identifies this ordered
set with the normal clusters on $X$. The conormal clusters are
all subsets of $X$ with their induced relations.

For a homomorphism $f:X\to Y$, let $R_f$ be its equality kernel
and define $\rho_f$ by
\[
 ([x_1],\ldots,[x_n])\in P_{\rho_f}
 \quad\Longleftrightarrow\quad
 (f(x_1),\ldots,f(x_n))\in P^Y.
\]
For a quotient datum $\rho$ and a subset $A\subseteq X$, let
$\rho|_A$ have equivalence relation $R_\rho\cap A^2$ and the
relations transported from the induced substructure $q_\rho(A)$.
For $\rho\leq\sigma$, let $\sigma/\rho$ on $X/\rho$ identify
the classes with the same $R_\sigma$-class and have the relations
transported from $X/\sigma$. The canonical bijections give
isomorphisms of relational structures
\[
 X/\rho_f\cong f(X),\qquad
 A/(\rho|_A)\cong q_\rho(A),\qquad
 (X/\rho)/(\sigma/\rho)\cong X/\sigma,
\]
where images carry induced relations. The assignment
$\sigma\mapsto\sigma/\rho$ is an order isomorphism from quotient
data on $X$ above $\rho$ to quotient data on $X/\rho$.

For a set-indexed family of surjective homomorphisms $f_i:X\to Y_i$, the
joint map into $\prod_iY_i$ is an induced-substructure embedding
if and only if it is injective and, for every symbol $P$ of arity
$n$ and tuple $(x_1,\ldots,x_n)$,
\[
 (x_1,\ldots,x_n)\in P^X
 \quad\Longleftrightarrow\quad
 (f_i(x_1),\ldots,f_i(x_n))\in P^{Y_i}\text{ for every }i.
\]
\end{corollary}
\begin{proof}
The containment defining a quotient datum says exactly that its
canonical surjection is a homomorphism. Its order expresses
factorization of these surjections. Thus
\cref{tc:prop:quotient-dictionary} applies, and the displayed
isomorphisms follow by the specified transport of relations.
The last assertion uses the coordinatewise relations on products.
\end{proof}

For the reflexive graph categories of \cref{tc:cor:graphs}, the
same formulas hold with the additional requirement that the
quotient relation be reflexive and, in the symmetric case,
symmetric. For the category $\Equiv$ of \cref{tc:cor:equiv},
write $E_X$ for the given equivalence relation on $X$. Quotient
data consist of an arbitrary equivalence relation $R$ on the
underlying set and an equivalence relation $D$ on $X/R$ satisfying
$(q\times q)(E_X)\subseteq D$. In particular, $R$ and $E_X$
serve different roles; no containment between them is required.
Induced subsets, transported quotient relations, and products
stay in these categories, so the proof and all the formulas above
apply without change.

\subsubsection{Matrix conditions and relational Horn theories}
\label{tc:sec:matrix-conditions}
We use the matrix-closure terminology introduced in
\cite{Janelidze2006}; the distinction between common and row-wise
interpretations is recalled in \cite[Section~2.1]{HJ2024}.
For a single designated $n$-ary relation $R$, an unpointed,
variable-only matrix condition has the form
\[
 R(\mathbf c_1)\land\cdots\land R(\mathbf c_m)
 \ \vdash_{\vec x}\ R(\mathbf c_0),
\]
where the $\mathbf c_i$ are columns of variables. Ordinary matrix
closure uses one common assignment of the variables. Strict matrix
closure allows independent assignments in different rows: to express
it by this sequent, first rename variables in distinct rows
independently. The results below concern structures with designated
relations satisfying these sequents, rather than the categorical
condition that every internal relation be matrix-closed.

We allow more generally a set-sized, single-sorted relational
signature $\Sigma$ of positive finite arities and a set $\mathbb H$
of rules
\begin{equation}\label{tc:eq:relational-horn}
 \bigwedge_{j=1}^{m}R_j(\mathbf t_j)
 \ \vdash_{\vec x}\ S(\mathbf t_0),
 \qquad m\geq0,
\end{equation}
where every tuple consists of variables; there are no constants,
operations, or equality atoms. The empty conjunction is $\top$.
Let $\mathbf{Mod}(\mathbb H)$ be the full subcategory of
$\mathbf{Str}(\Sigma)$ on the models, with all relation-preserving
functions as morphisms. The standard relational initial-lift
construction \cite[Example~21.8(2)]{AHS} applies here: initial
structures are obtained by
intersecting inverse images of relations. Such structures still
satisfy every rule: an assignment satisfying all premises satisfies
them in every target, where the conclusion follows. The empty
source gives the full relations. Thus this category is topological
over $\Set$, and its natural factorization system consists of
surjective homomorphisms and induced-substructure embeddings,
by \cref{tc:prop:fs}.

\begin{proposition}\label{tc:prop:horn-inclusion}
Let $\Sigma$ be a set-sized relational signature of positive finite
arities, and let $\mathbb H$ be a set of universally quantified
implications from a finite conjunction of relation atoms to one
relation atom, all arguments being variables. Equip
$\mathbf{Mod}(\mathbb H)$ and $\mathbf{Str}(\Sigma)$ with the
factorization systems of surjective homomorphisms and
induced-substructure embeddings. The inclusion
\[
 J:\mathbf{Mod}(\mathbb H)\longrightarrow\mathbf{Str}(\Sigma)
\]
is a transfer functor if and only if every pushout of two surjective
homomorphisms between models, computed in
$\mathbf{Str}(\Sigma)$, is again a model. When this holds,
$\mathbf{Mod}(\mathbb H)$ has a noetherian form inducing the
displayed factorization system, with complete algebraic modular
fibres.
\end{proposition}
\begin{proof}
Both categories are topological over sets, so the required diagrams
exist. The inclusion is full and hence conservative. It preserves
surjections and induced embeddings. Pullbacks of induced embeddings
are computed by taking the set-theoretic pullback with its induced
relations, which satisfy $\mathbb H$ by the initial-structure
calculation above. Thus only preservation of pushouts of two
surjections remains. An ambient pushout whose vertex is a model
is also a pushout in the full subcategory. Conversely, preservation
of such a pushout identifies its ambient vertex with a model.
The final assertion follows from \cref{tc:prop:relational,thm:transfer}
and \cref{thm:modular-replacement}.
\end{proof}

Here are two syntactic ways of obtaining noetherian forms. They
apply to a whole family of rules; rules justified by different
criteria cannot be combined without checking one of the criteria
for the combined family.

\begin{theorem}[Rules with at most one premise]\label{tc:thm:single-premise}
Let $\Sigma$ be a set-sized relational signature of positive finite
arities. Let $\mathbb H$ be a set of universally quantified rules,
each of one of the forms
\[
 \top\ \vdash_{\vec x}\ S(\mathbf t),
 \qquad
 R(x_1,\ldots,x_n)\ \vdash_{\vec x}\ S(\mathbf t),
\]
where all arguments are variables and, in the second form,
$x_1,\ldots,x_n$ are pairwise distinct. Variables in $\mathbf t$
may repeat, omit premise variables, or include additional variables.
Then the inclusion of the category of $\mathbb H$-models and
relation-preserving functions into $\mathbf{Str}(\Sigma)$ is a
transfer functor for the surjection--induced-embedding systems.
Consequently this model category has a noetherian form inducing
that system, with complete algebraic modular fibres.
\end{theorem}
\begin{proof}
Consider an ambient pushout of surjections $A\to B$ and $A\to C$.
Its underlying set $P$ is the pushout of sets; both maps $B\to P$
and $C\to P$ are surjective. For each symbol, its relation on $P$
is the union of the direct images of the relations on $B$ and $C$.
Suppose an assignment in $P$ satisfies a premise
$R(x_1,\ldots,x_n)$. This tuple comes from, say, an $R$-tuple
in $B$. Since the premise variables are distinct, those entries
specify a lift of their assignment to $B$. Lift all remaining
variables using the surjection $B\to P$. The conclusion holds
in $B$, hence in $P$. For a premise-free rule, lift the entire
assignment through either surjection. Thus $P$ satisfies every
rule. Apply \cref{tc:prop:horn-inclusion}.
\end{proof}

In particular, this covers coordinate-permutation rules such as
symmetry and cyclic invariance, together with premise-free rules
such as reflexivity. Repeated premise variables are excluded
because a tuple in a pushout can acquire equal coordinates without
having a lift with the same equalities.

\begin{theorem}[Premise variables retained in the conclusion]\label{tc:thm:horn-tests}
Let $\Sigma$ be a set-sized relational signature of positive finite
arities, and let $\mathbb H$ be a set of universally quantified
rules
\[
 \bigwedge_{j=1}^{m}R_j(\mathbf t_j)
 \ \vdash_{\vec x}\ S(\mathbf t_0),\qquad m\geq0,
\]
whose arguments are variables. Suppose that, in each rule, every
variable occurring in a premise occurs in the conclusion tuple
$\mathbf t_0$. Then the category $\mathbf{Mod}(\mathbb H)$ of
set-based models and relation-preserving functions has a noetherian
form inducing surjective homomorphisms and induced-substructure
embeddings, with complete algebraic modular fibres.
\end{theorem}
\begin{proof}
The category is topological by the initial-structure calculation
above. For every finite model $A$, form $A^+=A\amalg\{*\}$,
interpreting each $k$-ary symbol $R$ by
\[
 R^{A^+}=R^A\ \cup\
 \{(a_1,\ldots,a_k)\in(A^+)^k:
                         a_i=*\text{ for some }i\}.
\]
This is a model. Indeed, if the conclusion tuple of a rule contains
$*$, its conclusion is true by definition. Otherwise every variable
in a premise takes its value in $A$, by the stated variable
condition. The premises and conclusion can therefore be evaluated
in $A$, where the rule holds. Unused context variables may be
discarded.

Every $A^+$ is injective with respect to induced embeddings:
extend a homomorphism defined on an induced substructure by
assigning $*$ to every remaining element. A relation tuple lying
entirely in the substructure is preserved by the given map, and
any other tuple maps to one containing $*$.

Choose one representative of every finite model up to isomorphism;
these form a set since $\Sigma$ is set-sized. The corresponding
$A^+$ initially determine all models. To see this, let
$(x_1,\ldots,x_k)$ be a missing $R$-tuple in a model $X$, and
let $A$ be the induced substructure on its finitely many entries.
Induced substructures satisfy the rules. The map $X\to A^+$
that is the identity on $A$ and sends every other element to $*$
is a homomorphism and still detects that missing tuple. Thus every
relation is recovered from the inverse images along these maps.
All tests are nonempty, and the full two-point structure can be
adjoined as an additional injective test. Apply
\cref{tc:thm:form}.
\end{proof}

The second criterion allows several premises and repeated premise
variables, but it does not require the inclusion into unrestricted
structures to preserve the relevant pushouts. Its transfer functor
is supplied instead by the injective-test construction. Neither
criterion is necessary: transitivity satisfies neither, whereas
preorders have the form of \cref{tc:cor:preord}.

For a strict row-wise matrix with at least one premise column,
the variable condition in \cref{tc:thm:horn-tests}, after renaming
variables independently by row, forces every premise entry in a
row to equal that row's conclusion entry. Such a matrix rule is
tautological. Thus the nontrivial applications of this particular
criterion to a single matrix concern common-variable closure;
it does not establish the general strict row-wise case.

The quotient-closure criterion of \cref{prop:closed-subcategory}
cannot replace these arguments for general matrix rules. For a
single relation and a family of rules each having at least one
premise, $(X,\varnothing)$ is always a model, and the identity
function gives a surjective homomorphism
\[
 (X,\varnothing)\longrightarrow(X,R)
\]
for every relation $R$. Closure under all ambient surjective
homomorphic images would therefore force every relation to be a
model. The homogeneous-pushout condition in
\cref{tc:prop:horn-inclusion} is weaker.

\subsubsection{Positive disjunctive extensions}
Once a relational model category has a noetherian form with the
natural factorization system, it admits a further class of
geometric extensions, parallel to the algebraic case.

\begin{proposition}\label{tc:prop:positive-relational}
Let $\Sigma$ be a set-sized relational signature of positive finite
arities. Let $\mathcal C$ be a full replete subcategory of
$\mathbf{Str}(\Sigma)$ with a noetherian form whose quotients
are exactly the surjective homomorphisms and whose embeddings
are exactly the induced-substructure embeddings in $\mathcal C$.
Let $\mathbb G$ be a set of geometric axioms
\[
 \top\ \vdash_{\vec x}\
 \bigvee_{i\in I}\ \bigwedge_{j=1}^{n_i}\alpha_{ij}(\vec x),
\]
where each context is finite, $I$ is a set, each $n_i$ is finite,
and each $\alpha_{ij}$ is a relation atom or an equality between
variables. Then the full subcategory $\mathcal C_{\mathbb G}$
of objects satisfying $\mathbb G$ inherits the given noetherian
form, with the restricted quotient and embedding classes. In
particular, it also admits such a form with complete algebraic
modular fibres.
\end{proposition}
\begin{proof}
Atomic formulas are reflected by induced-substructure embeddings.
Thus an assignment in an induced subobject satisfies the same
disjunct that it satisfies in the ambient object. Along a
surjective homomorphism, lift a finite assignment, choose a
satisfied disjunct upstairs, and use preservation of atomic
formulas to obtain it downstairs. Hence
$\mathcal C_{\mathbb G}$ is closed under the specified subobjects
and quotients inside $\mathcal C$. Apply
\cref{prop:closed-subcategory,thm:modular-replacement}.
\end{proof}

\begin{corollary}\label{tc:cor:total-preorders}
Let $\mathbf{TotPreord}$ be the category of sets equipped with a
reflexive, transitive, total relation and monotone functions;
the empty set is allowed. It has a noetherian form whose quotients
are the surjective monotone maps and whose embeddings are the
injective maps preserving and reflecting the relation. The fibres
can be chosen complete, algebraic and modular.
\end{corollary}
\begin{proof}
Apply \cref{tc:prop:positive-relational} to
$\mathcal C=\mathbf{Preord}$, using \cref{tc:cor:preord},
and the single axiom
$\top\vdash_{x,y}(x\leq y)\lor(y\leq x)$.
\end{proof}

\subsection{Extended pseudometric spaces}\label{tc:sec:metrics}
Let $\mathbf{PMet}_\infty$ denote the category of sets equipped with
symmetric distances in $[0,\infty]$ satisfying the triangle inequality
and $d(x,x)=0$, with nonexpansive maps. Distinct points may have
distance zero. The initial-structure formula in the symmetric case of
\cref{tc:thm:quantale}, specialized to
$V=([0,\infty],\geq,+,0)$, gives for a source $(f_j:S\to UX_j)_j$
\begin{equation}\label{tc:eq:initial-metric}
d(x,y)=\sup_j d_j(f_j(x),f_j(y)),\qquad \sup\varnothing=0.
\end{equation}
The same specialization identifies the initial embeddings as isometric
injections. These are the pseudometric conventions of
\cite[Example~8.2(4), footnote~44, and Example~10.42(5)]{AHS}.

The extension formula below is the McShane formula
\cite{McShane1934}; its extended-valued version is a specialization
of \eqref{tc:eq:quantale-extension}. The existence result follows
from the same quantale theorem.

\begin{proposition}\label{tc:prop:pmetric}
Let $\mathbf{PMet}_\infty$ be the category of sets $X$ with symmetric
functions $d:X\times X\to[0,\infty]$ satisfying $d(x,x)=0$ and
$d(x,z)\leq d(x,y)+d(y,z)$. Morphisms $f:(X,d)\to(Y,e)$ are
nonexpansive functions, so $e(fx,fy)\leq d(x,y)$; distinct points
may have distance zero, and infinite distances are allowed. Then
$\mathbf{PMet}_\infty$ has a noetherian form inducing surjective
nonexpansive maps and isometric injections.
\end{proposition}
\begin{proof}
Apply the symmetric case of \cref{tc:thm:quantale} to
$V=([0,\infty],\geq,+,0)$. Its objects and morphisms are exactly
the stated extended pseudometric spaces and nonexpansive maps,
and injective fully faithful maps are isometric injections.
\end{proof}

For this specialization, the symmetric residual test in the proof of
\cref{tc:thm:quantale} is $Q=[0,\infty]$ with distance
\[
\delta(s,t)=
\begin{cases}
|s-t|,&s,t<\infty,\\
0,&s=t=\infty,\\
\infty,&\text{exactly one of }s,t\text{ is }\infty.
\end{cases}
\]
The initial-test formula becomes
\begin{equation}\label{tc:eq:metric-tests}
d(x,y)=\sup_{a:X\to Q}\delta(a(x),a(y)),
\end{equation}
where the supremum ranges over nonexpansive maps. For an isometric
inclusion $A\subseteq X$ and a nonexpansive map $h:A\to Q$, the
extension formula becomes
\begin{equation}\label{tc:eq:metric-extension}
\overline h(x)=\inf_{a\in A}\bigl(h(a)+d(x,a)\bigr),
\qquad\inf\varnothing=\infty.
\end{equation}
Thus both formulas, including their infinite-value and empty-domain
conventions, are supplied by the general construction.

Both the value $\infty$ and zero distances between distinct points
belong to this example. Restricting to separated or finite-valued
metric spaces does not give the same topological category over $\Set$.

The metric quotient data retain the entire target distance.
Consequently the relevant quotients include surjective contractions
that strictly decrease distances, as required by
\cref{tc:prop:pmetric}.

\begin{corollary}\label{tc:cor:pmetric-calculus}
Let $\mathbf{PMet}_\infty$ be the category of extended pseudometric
spaces and nonexpansive maps, with symmetric distances in
$[0,\infty]$, zero diagonal, and no separation condition. Let
$\mathcal F$ be a noetherian form inducing surjective nonexpansive
maps and isometric injections. For an object $(X,d_X)$, a quotient
datum is $\rho=(R_\rho,d_\rho)$, where $R_\rho$ is an equivalence
relation on $X$ and $d_\rho$ is an extended pseudometric on
$X/R_\rho$ such that
\[
 d_\rho([x]_\rho,[y]_\rho)\leq d_X(x,y)
 \qquad(x,y\in X).
\]
Write $q_\rho:X\to X/\rho=(X/R_\rho,d_\rho)$ for the
canonical surjection. Order these data by
\[
 \rho\leq\sigma\ \Longleftrightarrow\
 R_\rho\subseteq R_\sigma\ \text{ and }\
 d_\sigma([x]_\sigma,[y]_\sigma)
 \leq d_\rho([x]_\rho,[y]_\rho)\quad(x,y\in X).
\]
Then $\rho\mapsto\Ker_{\mathcal F}q_\rho$ identifies this ordered
set with the normal clusters on $X$, and the conormal clusters
are all subsets with the distances restricted from $X$.

For a nonexpansive map $f:X\to Y$, define $R_f$ by equality
of images and put
$d_{\rho_f}([x],[y])=d_Y(f(x),f(y))$.
For a quotient datum $\rho$ and a subset $A\subseteq X$, define
$\rho|_A$ by restricting $R_\rho$ and using the distance
$d_{\rho|_A}([a],[b])=d_\rho([a]_\rho,[b]_\rho)$.
For $\rho\leq\sigma$, define $\sigma/\rho$ on $X/\rho$ by
identifying $R_\rho$-classes with the same $R_\sigma$-class and
transporting $d_\sigma$ to the resulting quotient set. Then the
canonical bijections are isometries:
\[
 X/\rho_f\cong f(X),\qquad
 A/(\rho|_A)\cong q_\rho(A),\qquad
 (X/\rho)/(\sigma/\rho)\cong X/\sigma.
\]
The two images have their restricted target distances, and
$\sigma\mapsto\sigma/\rho$ identifies quotient data on $X$
above $\rho$ with quotient data on $X/\rho$.

For a set-indexed family of surjective nonexpansive maps $f_i:X\to Y_i$,
the joint map into $\prod_iY_i$ is an isometric injection if and
only if the family jointly separates points and
\[
 d_X(x,y)=\sup_i d_{Y_i}(f_i(x),f_i(y))
 \qquad(x,y\in X),\qquad\sup\varnothing=0.
\]
\end{corollary}
\begin{proof}
The defining inequalities express nonexpansiveness of $q_\rho$
and of the comparison $X/\rho\to X/\sigma$, respectively.
Apply \cref{tc:prop:quotient-dictionary}. The three isometries
preserve distances by their definitions, and the product distance
is the supremum of the coordinate distances.
\end{proof}

The equivalence relation $R_\rho$ records equality of quotient
points, whereas $d_\rho$ may assign zero distance to distinct
points. Thus these formulas neither impose metric separation nor
require the distance on a quotient to be the largest distance
making its quotient map nonexpansive.
For example, let $X=\{a,b\}$ have $d_X(a,b)=2$, and give the
same set the distance $d_\rho(a,b)=1$. The identity function is
a surjective nonexpansive map and has quotient datum
$(\Delta_X,d_\rho)$, distinct from $(\Delta_X,d_X)$. These give
distinct normal clusters although their underlying equivalence
relations are identical.
\subsection{Further concrete examples}\label{tc:sec:further-examples}

\subsubsection{Pretopological spaces}\label{tc:sec:pretopological}
A pretopological space, equivalently a \v{C}ech closure space, has an
extensive operator $c:\Pow(X)\to\Pow(X)$ with
$c(\varnothing)=\varnothing$ and
$c(A\cup B)=c(A)\cup c(B)$; idempotence is not required.
Continuity means $f(c_X A)\subseteq c_Y(fA)$. These spaces form a
topological category, as is also seen from their description by
vicinity filters. For an induced subspace $A\subseteq X$ one has
$c_A(B)=A\cap c_X(B)$ for $B\subseteq A$.

The three-point test used here is Bourdaud's space: its injectivity
and initial determination are proved in
\cite[Propositions~8.2--8.3]{Heckmann2006}. Thus our construction
has the following consequence.

\begin{corollary}\label{tc:cor:pretopological}
Let $\mathbf{PrTop}$ have as objects sets $X$ with extensive operators
$c_X:\Pow(X)\to\Pow(X)$ satisfying $c_X(\varnothing)=\varnothing$
and $c_X(A\cup B)=c_X(A)\cup c_X(B)$; idempotence is not required.
Its morphisms are continuous functions, meaning
$f(c_X(A))\subseteq c_Y(f(A))$ for every $A\subseteq X$.
Then this category of pretopological spaces admits a noetherian form
inducing surjective continuous maps and induced-subspace embeddings,
where a subset $A\subseteq X$ has closure
$c_A(B)=A\cap c_X(B)$ for $B\subseteq A$.
\end{corollary}
\begin{proof}
Use Bourdaud's test $\Lambda=\{0,1,*\}$, whose closure is
\[
c_\Lambda(S)=
\begin{cases}
\varnothing,&S=\varnothing,\\
\{0,*\},&S=\{0\},\\
\Lambda,&S\cap\{1,*\}\ne\varnothing.
\end{cases}
\]
We recall the test verification explicitly. A function
$f:X\to\Lambda$ is continuous exactly when
\begin{equation}\label{tc:eq:bourdaud-test}
f^{-1}\{1\}\cap c_X(f^{-1}\{0\})=\varnothing.
\end{equation}
The only nonempty set in $\Lambda$ with non-total closure is $\{0\}$,
which proves this criterion. If $x\notin c_X(B)$, send $B$ to $0$,
$x$ to $1$, and all remaining points to $*$. This continuous test
detects that failure of closure membership, so the tests initially
determine $X$.

Given a continuous map from an induced subspace $A$ to $\Lambda$,
extend it by $*$ outside $A$. Its $0$- and $1$-fibres are unchanged,
and \eqref{tc:eq:bourdaud-test} follows from
$c_A(B)=A\cap c_X(B)$. Thus $\Lambda$ is embedding-injective,
and Theorem~\ref{tc:thm:form} applies.
\end{proof}

For a pretopological space $(X,c_X)$, a weak quotient is specified by
an equivalence relation $R$ on $X$ and a pretopological closure
$\bar c$ on $X/R$ satisfying
\[
 q(c_X(B))\subseteq\bar c(q(B))\qquad(B\subseteq X),
 \quad q(x)=[x]_R.
\]
Thus normal clusters retain both $R$ and $\bar c$, while conormal
clusters are the induced subspaces. The first isomorphism for a
continuous $f:X\to Y$ equips its image with
$c_{f(X)}(B)=f(X)\cap c_Y(B)$. All restriction, iteration, and
correspondence rules follow from \cref{tc:prop:quotient-dictionary}.

\subsubsection{Moore closure spaces and abstract convexity}\label{tc:sec:closure}
Here a Moore closure space is a set $X$ with a family
$\mathcal C_X\subseteq\Pow(X)$ closed under arbitrary intersections,
including the empty intersection $X$. Its associated closure operator
is extensive, monotone, and idempotent. Unlike a \v{C}ech closure,
it need not preserve finite unions. We consider both the unrestricted
convention and the convention $\varnothing\in\mathcal C_X$.
An abstract convexity space additionally has
$\varnothing\in\mathcal C_X$ and closure under directed unions.
In all these categories, morphisms preserve distinguished subsets
under inverse image. These convexity conventions agree with
\cite[Section~2.1]{Kenney2023}.

\begin{corollary}\label{tc:cor:closure-convexity}
Consider each of the following categories of sets $X$ with
distinguished families $\mathcal C_X\subseteq\Pow(X)$:
\begin{enumerate}
\item Moore closure spaces, where $\mathcal C_X$ is closed under
arbitrary intersections, including the empty intersection $X$;
\item Moore closure spaces with the additional requirement
$\varnothing\in\mathcal C_X$;
\item abstract convexity spaces, where $\mathcal C_X$ contains
$\varnothing$ and $X$, and is closed under arbitrary intersections
and directed unions.
\end{enumerate}
In every case, morphisms are functions whose inverse images preserve
distinguished subsets. Each category admits a noetherian form inducing
surjective morphisms and embeddings with the trace structure
$\{A\cap C:C\in\mathcal C_X\}$ on a subset $A\subseteq X$.
\end{corollary}
\begin{proof}
Initial structures are generated by inverse images of distinguished
subsets under the required intersection and directed-union operations.
Thus the categories are topological. The test is the two-point set
$Q=\{0,1\}$ with distinguished subsets
\[
\mathcal C_Q=\{\{1\},Q\}
\quad\text{or}\quad
\mathcal C_Q=\{\varnothing,\{1\},Q\},
\]
respectively, according to whether the empty set is required to be
distinguished. Its morphisms are exactly characteristic functions of
distinguished subsets, so they initially determine every structure.

For an induced subspace $A\subseteq X$, the distinguished subsets
are the traces $A\cap C$, $C\in\mathcal C_X$. In the convexity case
these traces are indeed closed under directed unions: if $(D_i)$ is
a directed family of traces, its ambient convex hulls
$\operatorname{hull}_X(D_i)$ form a directed family and satisfy
$A\cap\operatorname{hull}_X(D_i)=D_i$. Their union is an ambient
convex set with trace $\bigcup_i D_i$. Every characteristic function
of a trace therefore extends to the characteristic function of an
ambient distinguished subset. The test is injective, and
Theorem~\ref{tc:thm:form} applies.
\end{proof}

In each of these closure and convexity categories, the weak quotient
data can be written $(R,\mathcal A)$, where $R$ is an equivalence
relation on $X$ and $\mathcal A\subseteq\mathcal C_X$ consists of
$R$-saturated sets and satisfies the same intersection, empty-set,
and directed-union requirements as the chosen category. The
distinguished subsets of $X/R$ are $\{q(A):A\in\mathcal A\}$.
The order is inclusion of relations and reverse inclusion of
families. Conormal clusters are trace subspaces, and restriction
of quotient data uses traces $\mathcal A|_S$. Consequently the
three isomorphisms and quotient correspondence have exactly the
form in \cref{tc:prop:quotient-dictionary}.

\subsubsection{Fuzzy and frame-valued topologies}\label{tc:sec:fuzzy}
Fix a set-sized frame $L$. An $L$-topology on $X$ is a subframe
$\tau_X\subseteq L^X$, with joins and finite meets taken pointwise;
the empty operations include the constant $0$ and $1$ functions.
A function $f:X\to Y$ is continuous when $u\circ f\in\tau_X$
for every $u\in\tau_Y$. Initial structures are the subframes
generated by these composites, so $L\text{-}\mathbf{Top}$ is
topological over sets.

The Sierpi\'nski $L$-space and its initial-test property are those of
\cite[Proposition~3.1 and Theorem~3.1]{NoorSrivastava2013}.
Their proof of Theorem~3.2 gives the needed extension property;
as recalled below, that argument does not require separation.

\begin{corollary}\label{tc:cor:fuzzy}
Let $L$ be a set-sized frame, and let $L\text{-}\mathbf{Top}$ be the
category whose objects are sets $X$ equipped with a subframe
$\tau_X\subseteq L^X$, using pointwise arbitrary joins and finite
meets, including the constant $0$ and $1$ functions. Its morphisms
$f:X\to Y$ satisfy $u\circ f\in\tau_X$ for every $u\in\tau_Y$.
No separation condition is imposed. Then $L\text{-}\mathbf{Top}$
has a noetherian form inducing surjective continuous maps and initial
embeddings, with trace topology $\{u|_A:u\in\tau_X\}$ on
$A\subseteq X$. For $L=[0,1]$ with its usual order, this includes
fuzzy topological spaces.
\end{corollary}
\begin{proof}
Let $L_{\mathrm S}$ be the set $L$ with the subframe of $L^L$
generated by $\id_L$. A function $u:X\to L_{\mathrm S}$ is continuous
exactly when $u\in\tau_X$. These tests initially determine $X$.
For an induced subspace $A\subseteq X$,
$\tau_A=\{u|_A:u\in\tau_X\}$, since restrictions preserve all the
frame operations. Hence every test on $A$ extends to $X$.
The argument requires no $T_0$ assumption. Apply
Theorem~\ref{tc:thm:form}.
\end{proof}

For an $L$-space $(X,\tau_X)$ the normal-cluster data are
$(R,\sigma)$, with $R$ an equivalence relation and
$\sigma\subseteq\tau_X$ a subframe of functions constant on
$R$-classes. The quotient $L$-topology is
\[
 \{v:X/R\to L:vq\in\sigma\},\qquad q(x)=[x]_R.
\]
Conormal clusters are arbitrary trace subspaces. In particular,
for a continuous $f:X\to Y$, its quotient data are equality of
$f$-values together with $\{vf:v\in\tau_Y\}$; its quotient is
canonically isomorphic to the induced $L$-subspace $f(X)$.
The remaining rules are those of \cref{tc:prop:quotient-dictionary}.

\subsubsection{Approach spaces}\label{tc:sec:approach}
An approach space has a distance $\delta:X\times\Pow(X)\to[0,\infty]$
satisfying
\[
\begin{gathered}
\delta(x,\{x\})=0,\qquad\delta(x,\varnothing)=\infty,\\
\delta(x,A\cup B)=\min\{\delta(x,A),\delta(x,B)\},\\
\delta(x,A)\leq\delta(x,A^{(\varepsilon)})+\varepsilon,
\qquad A^{(\varepsilon)}=\{y:\delta(y,A)\leq\varepsilon\}
\quad(0\leq\varepsilon<\infty).
\end{gathered}
\]
Its morphisms are contractions, namely functions with
$\delta_Y(fx,fA)\leq\delta_X(x,A)$. The category
$\mathbf{App}$ is topological; on a subspace one restricts $\delta$
to its points and subsets. We use the conventions of
\cite[Section~2.2]{GutierresHofmann2013}.

The Sierpi\'nski approach space and its embedding-injectivity are
recalled in \cite[Examples~2.8 and~3.14]{GutierresHofmann2013}.
Together with its initial-test property, verified below, they give
the following application.

\begin{corollary}\label{tc:cor:approach}
Let $\mathbf{App}$ be the category of sets $X$ with functions
$\delta_X:X\times\Pow(X)\to[0,\infty]$ satisfying
\[
\begin{gathered}
\delta_X(x,\{x\})=0,\qquad\delta_X(x,\varnothing)=\infty,\\
\delta_X(x,A\cup B)=\min\{\delta_X(x,A),\delta_X(x,B)\},\\
\delta_X(x,A)\leq\delta_X(x,A^{(\varepsilon)})+\varepsilon,
\qquad A^{(\varepsilon)}=\{y:\delta_X(y,A)\leq\varepsilon\}
\quad(0\leq\varepsilon<\infty).
\end{gathered}
\]
Morphisms are contractions $f:X\to Y$, meaning
$\delta_Y(fx,fA)\leq\delta_X(x,A)$ for all $x\in X$ and
$A\subseteq X$. This category of approach spaces, without a
separation assumption, admits a noetherian form inducing surjective
contractions and initial embeddings; the induced structure on a
subset restricts $\delta_X$ to that subset's points and subsets.
\end{corollary}
\begin{proof}
The test is the Sierpi\'nski approach space $P=[0,\infty]$ with
\begin{equation}\label{tc:eq:approach-test}
\delta_P(t,B)=
\begin{cases}
t\mathbin{\dot{-}}\sup B,&B\ne\varnothing,\\
\infty,&B=\varnothing,
\end{cases}
\end{equation}
where $s\mathbin{\dot{-}}t=\inf\{r\geq0:s\leq t+r\}$;
in particular, $\infty\mathbin{\dot{-}}\infty=0$.
This space is embedding-injective by the cited result.

For completeness, its initial determination follows directly from
the distance functions $r_A(x)=\delta_X(x,A)$. The approach axioms give
\[
\delta_X(x,A)\leq\delta_X(x,B)+\sup_{b\in B}\delta_X(b,A),
\]
so $r_A:X\to P$ is a contraction. Indeed, when the supremum is
finite, apply the last axiom to every larger $\varepsilon$ and use
$B\subseteq A^{(\varepsilon)}$; the other case is automatic.
For nonempty $A$, $r_A$ vanishes on $A$, whence
\[
\delta_P(r_A(x),r_A(A))=\delta_X(x,A).
\]
Consequently, if a function into $X$ becomes a contraction after
every test, apply the test $r_{f(B)}$ to see that it is itself a
contraction. The empty-set condition is automatic. Thus $P$ is an
injective initial test, and Theorem~\ref{tc:thm:form} applies.
\end{proof}

For an approach space $(X,\delta_X)$, weak quotient data consist of
an equivalence relation $R$ and an approach distance $\bar\delta$
on $X/R$ for which
\[
 \bar\delta(qx,qA)\leq\delta_X(x,A)
 \qquad(x\in X,\ A\subseteq X).
\]
These describe normal clusters; conormal clusters restrict the
point-to-subset distance to an arbitrary subset. For a contraction
$f:X\to Y$, the quotient associated with its factorization has
the approach structure of $f(X)$, obtained by restricting
$\delta_Y$. This also specifies restriction and iteration of
quotients through \cref{tc:prop:quotient-dictionary}.

\subsubsection{Totally bounded uniform spaces and proximities}\label{tc:sec:uniform}
Let $\mathbf{TBUnif}$ consist of totally bounded uniform spaces and
uniformly continuous maps, without a separation assumption. Initial
uniformities remain totally bounded: a basic entourage uses only
finitely many target coordinates, and their finite uniform covers
give a finite cover of the source. Thus this category is topological
over sets, with the usual induced uniform embeddings.

The following application uses classical uniform completion and
extension theory \cite{Isbell1964}, together with the correspondence
between proximities and totally bounded uniformities
\cite[Corollary~1.2.10.8]{Dimitrijevic2009}.

\begin{corollary}\label{tc:cor:uniform}
Let $\mathbf{TBUnif}$ be the category of all totally bounded uniform
spaces, without a separation assumption, and uniformly continuous
maps. Then $\mathbf{TBUnif}$ admits a noetherian form inducing
surjective uniformly continuous maps and uniform embeddings. The
category of Efremovi\v{c} proximity spaces, also allowing
non-separated spaces, and proximity-preserving functions likewise
admits a noetherian form inducing underlying surjections and
embeddings with the induced proximity.
\end{corollary}
\begin{proof}
The test is $I=[0,1]$ with its usual uniformity. For a totally bounded
space $X$, let $q:X\to X_0$ be its separated quotient and let
$K$ be the completion of $X_0$. Then $K$ is compact Hausdorff,
$X_0\hookrightarrow K$ is a uniform embedding, and the structure
of $X$ is initial along this composite. Continuous functions
$K\to I$ initially determine $K$ and are uniformly continuous.
Their restrictions therefore initially determine $X$.

If $A\hookrightarrow X$ is a uniform embedding, a uniformly
continuous map $A\to I$ factors through $A_0\subseteq X_0$.
Completeness of $I$ extends it to the closure of $A_0$ in $K$.
Tietze's theorem extends this function to $K$, and compactness makes
that extension uniformly continuous. Restriction along $X\to K$
gives the required extension. These standard uniform completion
facts are treated in \cite{Isbell1964}. Apply
Theorem~\ref{tc:thm:form}.

The concrete equivalence between totally bounded uniform spaces
and Efremovi\v{c} proximity spaces transports the result. Under it,
$A\mathrel\delta B$ means that $(A\times B)\cap V\ne\varnothing$
for every entourage $V$. The equivalence preserves underlying
sets, maps, and initial structures, hence also the specified classes.
\end{proof}

Total boundedness is part of this example. The interval tests used
here are not asserted to initially determine arbitrary uniform
spaces.

A weak quotient of a totally bounded uniform space $(X,\mathcal U_X)$
is an equivalence relation $R$ together with a totally bounded
uniformity $\mathcal V$ on $X/R$ such that
$(q\times q)^{-1}(V)\in\mathcal U_X$ for every $V\in\mathcal V$.
Normal clusters classify these quotient data, and conormal clusters
classify subsets with induced uniformities. The first isomorphism
identifies the quotient associated with a uniformly continuous map
with its image carrying the induced uniformity. The proximity
version is obtained through the concrete equivalence above, and
\cref{tc:prop:quotient-dictionary} gives restriction, iteration,
and correspondence in both descriptions.

\subsubsection{Bornologies generated by an \texorpdfstring{$\ell^\infty$}{ell-infinity}-structure}\label{tc:sec:bornological}
A bornology $\mathcal B_X$ contains all finite subsets of $X$ and
is closed under subsets and finite unions. Morphisms are bounded
maps, sending bounded sets to bounded sets. An $\ell^\infty$-bornology
has the further property that every unbounded subset contains a
countably infinite subset whose bounded subsets are exactly its
finite subsets. Write $\mathbf{Bor}_\infty$ for these spaces.
This is a topological subcategory of bornological spaces and
includes metric bornologies; see
\cite[Section~4]{ColebundersLowen2009}.

Colebunders and Lowen establish the initial determination of these
bornologies by $(\mathbb N,\mathrm{Fin})$
\cite[Theorem~4.7]{ColebundersLowen2009}. The elementary extension
argument below supplies the remaining hypothesis for our theorem.

\begin{corollary}\label{tc:cor:bornological}
Let $\mathbf{Bor}_\infty$ have as objects sets $X$ with a bornology
$\mathcal B_X\subseteq\Pow(X)$ containing all finite subsets and
closed under subsets and finite unions, subject to the further
condition that every unbounded subset of $X$ contains a countably
infinite subset whose bounded subsets are exactly its finite subsets.
Morphisms are bounded functions, which send bounded sets to bounded
sets. Then $\mathbf{Bor}_\infty$ admits a noetherian form inducing
surjective bounded maps and embeddings with the induced bornology
$\{B\subseteq A:B\in\mathcal B_X\}$ on a subset $A\subseteq X$.
\end{corollary}
\begin{proof}
Take the test $Q=(\mathbb N,\mathrm{Fin})$. By the cited detection
result, bounded maps to $Q$ initially determine every
$\mathbf{Bor}_\infty$ object. If $A\subseteq X$ has the induced bornology, extend a
bounded $f:A\to Q$ by $0$ outside $A$. For every bounded $B\subseteq X$,
the set $f(B\cap A)$ is finite, and the extended image of $B$ is
contained in $f(B\cap A)\cup\{0\}$. The extension is therefore
bounded. Thus $Q$ is also embedding-injective, and
Theorem~\ref{tc:thm:form} applies.
\end{proof}

This statement concerns $\mathbf{Bor}_\infty$, rather than all
bornological spaces or a category of bornological vector spaces.

For $(X,\mathcal B_X)$ in $\mathbf{Bor}_\infty$, a weak quotient
is $(R,\mathcal B)$, where $R$ is an equivalence relation and
$\mathcal B$ is an $\ell^\infty$-bornology on $X/R$ containing
every $q(B)$ for $B\in\mathcal B_X$. The target bornology must
remain in the stated category. These are the normal-cluster data;
conormal clusters are induced subsets. For a bounded $f:X\to Y$,
its image has bornology
$\{B\subseteq f(X):B\in\mathcal B_Y\}$. This identifies its
first isomorphism, and \cref{tc:prop:quotient-dictionary} supplies
the remaining quotient rules.

\subsubsection{Abstract simplicial complexes}\label{tc:sec:simplicial}
An abstract simplicial complex on a set $X$ is a family of finite
subsets, called faces, containing $\varnothing$ and all singletons
and closed under taking subsets. A simplicial map is a function
sending faces to faces. For a source of functions $(f_i:X\to X_i)_i$,
declare a finite subset $S\subseteq X$ to be a face exactly when
every $f_i(S)$ is a face. This supplies initial lifts; initial
embeddings are embeddings of full induced subcomplexes.

This standard topological category is described in
\cite[Example~22.2(3)]{AHS}. Its noetherian form is also a
consequence of the single-premise relational criterion.

\begin{corollary}\label{tc:cor:simplicial}
Let $\mathbf{Simp}$ be the category whose objects are sets $X$
equipped with families of finite subsets, called faces, containing
$\varnothing$ and all singletons and closed under taking subsets.
Morphisms are functions sending faces to faces. Then
$\mathbf{Simp}$ admits a noetherian form inducing surjective
simplicial maps and full induced-subcomplex embeddings, where the
faces on a subset $A\subseteq X$ are precisely the faces of $X$
contained in $A$.
\end{corollary}
\begin{proof}
Use a relation symbol $R_n$ of arity $n$ for every $n\geq1$,
with $R_n(x_1,\ldots,x_n)$ meaning that
$\{x_1,\ldots,x_n\}$ is a face. The axioms are
\[
 \top\vdash_x R_1(x),\qquad
 R_n(x_1,\ldots,x_n)\vdash_{x_1,\ldots,x_n}
 R_k(x_{u(1)},\ldots,x_{u(k)})
\]
for all $n,k\geq1$ and functions
$u:\{1,\ldots,k\}\to\{1,\ldots,n\}$.
They express closure under taking nonempty subsets, permutations,
and repetitions of coordinates. Their models correspond exactly
to abstract simplicial complexes by adjoining the empty face,
and their homomorphisms correspond to simplicial maps.
Each nonempty premise has pairwise distinct variables, so
\cref{tc:thm:single-premise} applies. Induced relational substructures
correspond precisely to full induced subcomplexes.
\end{proof}

For an abstract complex $(X,\mathcal K_X)$, a weak quotient is an
equivalence relation $R$ together with a complex $\mathcal K$
on $X/R$ containing $q(F)$ for every face $F\in\mathcal K_X$.
Additional faces are allowed. These are its normal-cluster data,
and conormal clusters are full induced subcomplexes. For a
simplicial map $f:X\to Y$, the image complex consists of all
faces of $Y$ contained in $f(X)$, including any such faces that
are not images of faces of $X$. With this structure the canonical
first isomorphism, restriction and iterated-quotient isomorphisms,
and quotient correspondence are precisely those of
\cref{tc:prop:quotient-dictionary}.

\subsubsection{Independent structures on one set}\label{tc:sec:independent}
The test hypothesis is closed under combining independent structures.
This gives further examples without changing the construction of the form.

\begin{proposition}\label{tc:prop:independent}
Let $(U_j:\C_j\to\Set)_{j\in J}$ be a set-indexed family of
topological categories. For each $j$, suppose that there is a
set-indexed family $(Q_{j,i})_{i\in I_j}$ of objects with nonempty
underlying sets such that all maps from any $X\in\C_j$ to these
objects jointly form a $U_j$-initial source, and every map
$A\to Q_{j,i}$ extends along every $U_j$-initial injection $A\to B$.
Let $\C$ be the category of sets equipped independently with one
$\C_j$-structure for each $j$, with functions that are morphisms in
every component, and let $U:\C\to\Set$ forget all structures.
Then $U$ is topological and has a set-indexed family of nonempty-valued
tests whose maps initially determine every object and extend along
every $U$-initial injection. Consequently $\C$ admits a noetherian
form inducing underlying surjections and injections initial in every
component. The assertion includes $J=\varnothing$, when $\C=\Set$.
\end{proposition}
\begin{proof}
Initial lifts are formed componentwise. For a test $Q$ in component
$j$, equip its underlying set with the indiscrete structure in
every other component. Every function into any of those other
components is a morphism. Hence maps into this combined test are
exactly the test maps for component $j$.

The family of all these tests, indexed by the disjoint union of
the component test families, is set-sized and nonempty-valued.
It initially determines every component, hence the combined
structure. An initial embedding is initial in every component;
extension of a map to one of the tests requires only the extension
in its distinguished component. All remaining continuity conditions
are automatic. Theorem~\ref{tc:thm:form} applies. If $J$ is empty,
the category is $\Set$, with the indiscrete two-point set as a test.
\end{proof}

For the topological-convexity convention in the last example below,
see \cite[Section~2.1]{Kenney2023}.

\begin{corollary}\label{tc:cor:combined}
Each of the following categories admits a noetherian form inducing
underlying surjections and injections with all component structures
induced from their codomains:
\begin{enumerate}
\item for a fixed set $J$, sets equipped with a topology for each
$j\in J$, with functions continuous for each corresponding pair of
topologies, including bitopological spaces when $|J|=2$;
\item sets with a topology and a $\sigma$-algebra, with functions
that are continuous and measurable;
\item sets with a topology and a preorder, with continuous monotone
functions;
\item sets with a topology and an abstract convexity, namely a family
of subsets containing $\varnothing$ and the whole set and closed
under arbitrary intersections and directed unions, with continuous
functions whose inverse images preserve convex subsets.
\end{enumerate}
In every case, the structures are independent: no compatibility
axiom between the components is imposed.
\end{corollary}
\begin{proof}
Apply \cref{tc:prop:independent} to the injective initial tests in
\cref{tc:cor:measurable,tc:cor:preord,tc:cor:closure-convexity} and
the proof of \cref{tc:cor:fuzzy}, taking $L=\{0,1\}$ for each
topological component.
\end{proof}

These examples impose no compatibility axiom between their
components. Additional requirements, such as
closedness of the order relation or connectedness of every convex
set, require a separate verification of the test hypothesis.

For independent structures, a weak quotient uses one equivalence
relation $R$ on the underlying set and an admissible quotient
structure on $X/R$ in each component. Conormal clusters carry all
the induced component structures. For example, a measurable
topological space has quotient data $(R,\sigma,\mathcal A)$,
where $\sigma$ is a subtopology and $\mathcal A$ a
sub-$\sigma$-algebra of the source, both consisting of
$R$-saturated sets. The quotient carries both corresponding
structures. The rules of \cref{tc:prop:quotient-dictionary} hold
componentwise, and the subdirect condition requires joint
injectivity and initial determination of every component.

\subsection{What topologicity alone does not supply}\label{sec:top-scope}
Theorem~\ref{tc:thm:form} answers the existence question for a broad class of
topological categories. Its additional test hypothesis is substantial.
The following necessary consequence provides a way to detect its failure.

The argument below is the standard pushout rule for injectivity
\cite[Example~2.2(3)]{AHeS2007}, applied to a family that detects
initiality.

\begin{proposition}\label{tc:prop:pushout-stable}
Let $U:\C\to\Set$ be topological, and suppose that a set-indexed
family $(Q_i)_{i\in I}$ has nonempty underlying sets, that all maps
$X\to Q_i$ jointly form a $U$-initial source for every $X\in\C$,
and that every map $A\to Q_i$ extends along every $U$-initial
injection $A\to B$. Then the class of $U$-initial injections is
stable under pushout along arbitrary morphisms of $\C$.
\end{proposition}
\begin{proof}
Consider a pushout
\[
\begin{tikzcd}
A\arrow[r,"m"]\arrow[d,"f"']&B\arrow[d,"g"]\\
C\arrow[r,"j"']&Z
\end{tikzcd}
\qquad(m\in\M).
\]
The underlying square is a pushout of sets, so $Uj$ is injective.
For every test $a:C\to Q_i$, extend $af$ across $m$ to
$b:B\to Q_i$. The pair $a,b$ glues to $c:Z\to Q_i$, giving
$cj=a$. Thus every test on $C$ extends along $j$.

If a function $h:UX\to UC$ has $Uj\,h$ underlying a morphism,
then $a h=cj h$ underlies a morphism for every test $a$.
Initial determination of $C$ implies that $h$ underlies a morphism.
Therefore $j$ is initial, and hence belongs to $\M$.
\end{proof}

\begin{example}[A topological category without these tests]\label{tc:ex:horn}
Let $\C_4$ consist of sets with a reflexive binary relation satisfying
\begin{equation}\label{tc:eq:fourcycle}
xRy\ \wedge\ yRz\ \wedge\ zRt\ \wedge\ tRx
\quad\Longrightarrow\quad xRz,
\end{equation}
with relation-preserving maps. For a source of functions to objects
of $\C_4$, take the intersection of the inverse-image relations.
It is reflexive and satisfies \eqref{tc:eq:fourcycle}, since the
implication holds in each target. For an empty source it is the full
relation. This gives all initial lifts, so $\C_4\to\Set$ is
topological. Its initial embeddings are induced-substructure embeddings.

Let $A=\{a,b\}$ have only its loops. Let $B=\{a,x,b\}$ have,
besides loops, exactly $aRx$ and $xRb$, and let $C=\{a,y,b\}$ have,
besides loops, exactly $bRy$ and $yRa$. Each relation satisfies
\eqref{tc:eq:fourcycle}: its only closed directed walks are constant.
The inclusions of $A$ into $B$ and $C$ are initial embeddings.

Their pushout has underlying set $\{a,b,x,y\}$ and contains the
directed cycle
\[
a\longrightarrow x\longrightarrow b\longrightarrow y
\longrightarrow a.
\]
Its relation must therefore contain $aRb$ by
\eqref{tc:eq:fourcycle}. This pair was absent from $B$, so the
canonical map $B\to B\amalg_A C$ is not initial. Thus initial
embeddings in $\C_4$ are not stable under pushout. By
Proposition~\ref{tc:prop:pushout-stable}, $\C_4$ has no family of
injective initial tests. The category is already fibre-small: its
structures on a set $S$ form a subset of $\Pow(S\times S)$.
\end{example}

This example disproves an automatic passage from topologicity to the
test hypothesis. It does not disprove existence of a noetherian form:
the homogeneous pushouts in the transfer theorem involve two
surjections, whereas Proposition~\cref{tc:prop:pushout-stable} concerns
pushouts of embeddings along arbitrary maps. These are different
requirements.

Difunctionality is the relation property expressed by the row-wise
Mal'tsev matrix \cite[Section~2.2]{HJ2024}. It gives a further
limitation of the injective-test construction.

\begin{example}[Difunctional relations and injective tests]
\label{tc:ex:difunctional}
Let $\mathbf{DiffRel}$ be the category of sets with a binary
relation satisfying
\[
R(x,u)\land R(y,u)\land R(y,v)\ \vdash\ R(x,v).
\]
Morphisms are relation-preserving functions. This topological
category has no family of injective
initial tests of the kind used in \cref{tc:thm:form}. To see this,
take the induced embeddings from the empty relation on
$A=\{a,b\}$ into the difunctional relations
\[
 R^B=\{(a,u),(b,u)\},\quad B=\{a,b,u\},
 \qquad
 R^C=\{(a,v),(b,w)\},\quad C=\{a,b,v,w\}.
\]
Their pushout has the set-theoretic union as underlying set and
the least difunctional relation containing the two displayed
relations. Difunctionality forces $(a,w)$ and $(b,v)$, so
$C\to B\amalg_A C$ is not an induced embedding. This contradicts
the pushout stability forced by injective initial tests in
\cref{tc:prop:pushout-stable}. It obstructs that test construction;
it does not prove nonexistence of a noetherian form on
$\mathbf{DiffRel}$.
\end{example}

The classical equivalence between fibre-smallness and
co-wellpoweredness for topological constructs
\cite[Corollary~21.17(2)]{AHS} gives the following obstruction.
We include its direct kernel-cluster proof.

\begin{proposition}[A necessary size condition]\label{tc:prop:size}
Let $U:\C\to\Set$ be a topological functor, and suppose that
$\C$ admits a noetherian form with set-sized fibres inducing the
factorization system of morphisms with surjective underlying
functions and $U$-initial injections. Then $U$ is fibre-small up to
concrete isomorphism: for every set $S$, the objects $X$ with $UX=S$
form a set of classes under isomorphisms whose underlying function
is $\id_S$.
\end{proposition}
\begin{proof}
For a set $S$, let $D(S)$ be its discrete lift: every function
$S\to UX$ underlies a morphism $D(S)\to X$. For every structure $X$
on $S$, the identity function therefore gives an $\E$-quotient
$q_X:D(S)\to X$. Distinct structures up to concrete isomorphism give
distinct quotient classes. By the noetherian axioms, these classes are
determined by their kernel clusters in the set $\N_{D(S)}$: two such
quotients with the same kernel are isomorphic as quotients. Thus there
are only a set of structures on $S$ up to concrete isomorphism.
\end{proof}

\begin{example}[A size obstruction to the natural system]\label{tc:ex:large}
Let $L=\mathrm{Ord}\cup\{\infty\}$ be the class of ordinals with a
new greatest element, viewed as a category under its usual order.
Consider the projection
\[
U:\Set\times L\longrightarrow\Set.
\]
A morphism $(S,\alpha)\to(T,\beta)$ is a function $S\to T$ when
$\alpha\leq\beta$, and there are no such morphisms otherwise. This
category is locally small and the projection is faithful and amnestic.
The initial lift of a source to objects labelled $\alpha_j$ equips
its domain with the label $\inf_j\alpha_j$, with empty infimum
$\infty$. These infima exist even for class-indexed sources: if a
source contains an ordinal label, its ordinal labels have a least
element; otherwise the infimum is $\infty$. Indeed,
$\gamma\leq\inf_j\alpha_j$ holds precisely when
$\gamma\leq\alpha_j$ for all $j$, which verifies initiality.

The projection is therefore topological, but every fibre has a proper
class of pairwise nonisomorphic structures. In particular, the identity
maps $(S,0)\to(S,\alpha)$ give a proper class of distinct underlying
surjections with fixed domain. By Proposition~\ref{tc:prop:size}, no
noetherian form on this category can induce underlying surjections and
initial embeddings. This conclusion concerns that specified system;
it does not rule out a form inducing a different one.
\end{example}

\begin{question}\label{tc:question:all}
Does every topological category over $\Set$ admit a noetherian form
with some associated factorization system? Does every fibre-small
topological category admit one inducing underlying surjections and
initial embeddings?
\end{question}

The present argument leaves these questions unresolved. In particular,
the fibre-small category $\C_4$ avoids the size obstruction, but the
test construction does not apply to it. We obtain neither a noetherian
form on $\C_4$ nor an obstruction to every possible such form there.

\Needspace{12\baselineskip}
\phantomsection

\bigskip
\begingroup
\small
\noindent\textsc{Kishan Dayaram}\\
Mathematics Division, Department of Mathematical Sciences, Stellenbosch University, Private Bag X1, Matieland 7602, South Africa.

\medskip
\noindent\textsc{Zurab Janelidze}\\
Mathematics Division, Department of Mathematical Sciences, Stellenbosch University, Private Bag X1, Matieland 7602, South Africa;\\
National Institute for Theoretical and Computational Sciences
(NITheCS), Private Bag X1, Matieland 7602, South Africa.

\medskip
\noindent\textsc{Francois van Niekerk}\\
Mathematics Division, Department of Mathematical Sciences, Stellenbosch University, Private Bag X1, Matieland 7602, South Africa.
\par
\endgroup

\end{document}